\documentclass{amsart}
\usepackage[arrow, matrix, curve]{xy}
\usepackage[latin1]{inputenc}
\usepackage[makeroom]{cancel}
\usepackage{amsmath, amssymb}
\usepackage{stmaryrd}
\usepackage{color, comment, verbatim}

\title{Twisted Spin Cobordism and positive scalar Curvature}

\newcommand*{\hooklongrightarrow}{\ensuremath{\lhook\joinrel\relbar\joinrel\rightarrow}}

\swapnumbers
\newtheorem{prop}{Proposition}[subsection]
\newtheorem{thm}[prop]{Theorem}
\newtheorem{cor}[prop]{Corollary}
\newtheorem{lem}[prop]{Lemma}
\newtheorem*{uthm}{Theorem}
\newtheorem*{ucor}{Corollary}
\newtheorem*{ulem}{Lemma}
\newtheorem*{uprop}{Proposition}

\theoremstyle{definition}
\newtheorem{defi}[prop]{Definition}
\newtheorem{rem}[prop]{Remark}
\newtheorem*{urem}{Remark}
\newtheorem{con}[prop]{Construction}
\newtheorem{obs}[prop]{Observation}
\newtheorem{conj}[prop]{Conjecture}

\newtheorem*{uconj}{Conjecture}

\DeclareMathOperator{\hofib}{hofib}

\DeclareMathOperator{\id}{id}
\newcommand{\Pin}{Pin}

\newcommand{\sd}{\odot} 
\newcommand{\MSpin}{MSpin}
\newcommand{\MSO}{MSO}

\newcommand{\MSpinK}{MSpin_{K}}
\newcommand{\KOK}{KO_{K}}
\newcommand{\koK}{ko_{K}}

\author{Fabian Hebestreit}
\address{Mathematisches Insitut, Universti\"at Bonn, Endenicher Alle 60, 53115 Bonn, Germany}
\email{f.hebestreit@math.uni-bonn.de}

\author{Michael Joachim}
\address{Fachbereich Mathematik und Informatik, WWU M\"unster, Einsteinstra\ss{}e 62, 48149 M\"unster, Germany}
\email{joachim@math.uni-muenster.de}

\newcommand{\commentout}[1]{}

\begin{document}

\begin{abstract}
We show how a suitably twisted $Spin$-cobordism spectrum connects to the question of existence of metrics of positive scalar curvature on closed, smooth manifolds by building on fundamental work of Gromov, Lawson, Rosenberg, Stolz and others. We then investigate this parametrised spectrum, compute its $mod~2$-cohomology and generalise the Anderson-Brown-Peterson splitting of the usual $Spin$-cobordism spectrum to the twisted case. Along the way we also describe the $mod~2$-cohomology of various twisted, connective covers of real $K$-theory. In an appendix we provide a comparison of our geometric models of twisted $Spin$ cobordism and twisted $K$-theory with others arising from abstract homotopy theory. 
\end{abstract}

\maketitle
\setcounter{tocdepth}{1}
\tableofcontents

\section{Introduction}
The problem of classifying manifolds admitting Riemannian metrics with special features, e.g. certain kinds of symmetry or curvature, is one of the core interests in differential geometry. The present work is motivated by this classification for the case of metrics with positive scalar curvature. Of the three classical types of curvature (sectional, Ricci-, and scalar curvature) the latter is the weakest. It is given by averaging processes from the other two and thus the most robust against manipulation of the metric and even the underlying manifold. The following theorem is arguably the most prominent example of this phenomenon and forms the cornerstone of current work on the existence of positive scalar curvature metrics.

\begin{uthm}[Gromov-Lawson 1980]
Let $(M,g)$ be a smooth $n$-dimensional Riemannian manifold of positive scalar curvature and $\varphi: S^k \times D^{n-k} \hooklongrightarrow M$ an embedding of a $k$-sphere with trivialised normal bundle. Then the manifold arising from $M$ by surgery along $\varphi$ again carries a positive scalar curvature metric as long as $n-k \geq 3$. Indeed, the metric $g$ can be extended to a metric with positive scalar curvature on the trace of the surgery.
\end{uthm}

Combined with the standard techniques of surgery and handlebody theory, this result shows that in order to prove the existence of a metric of positive scalar curvature on a given closed, smooth manifold $M$ one need only exhibit such a manifold cobordant to $M$, provided one can bound the dimensions of the surgeries occuring as one moves through the cobordism. These dimensions can be controlled by reference maps to some background space and we shall follow Kreck by using the normal $1$-type of $M$ as such. It is determined by the fundamental group $\pi$ of $M$ and by a lift $u$ of either the first or the first two Stiefel-Whitney classes of $M$ along the map $M \rightarrow B\pi$ classifying the universal cover of $M$; the second Stiefel-Whitney class is required only when the universal cover of $M$ admits a spin structure (which is of course exactly the case when there is a lift at all). We shall refer to such manifolds as \emph{almost spin} and call $M$ \emph{totally non-spin} otherwise.

We arrive at the well-known statement that for $n \geq 5$ a closed manifold $M$ admits a metric of positive scalar curvature if and only the class $[M] \in \Omega^{(\pi,u)}_n$ can be represented by a manifold with positive scalar curvature at all. When explicit generators of these cobordism groups are known, one can immediately derive classification results, the case of simply connected, non-spin manifolds being the simplest: In this case $\Omega^{(\pi,u)}_* = \Omega_*$ is the oriented cobordism ring, and all classes in dimension $3$ and above have representatives admitting positive scalar curvature metrics, so the same is true for any manifold in this case. However, for the majority of choices $(\pi,u)$, such representatives are not known and further reductions are required. In the case of both totally non-spin and spin manifolds there are rather satisfactory results. In the case of the former $\Omega^{(\pi,u)}_n = \Omega_n(B\pi;u)$, the right hand side denoting oriented cobordism with local coefficients, and it is a folklore result that $M$ admits a metric of positive scalar curvature if and only if $th([M])$ can be represented by a manifold with positive scalar curvature, where 
\[th \colon \Omega_n(B\pi;u) \longrightarrow H_n(B\pi, \mathbb Z^u)\]
is the Thom orientation.  Since the group homology on the right is far smaller than than the cobordism group on the left, this allows for many direct computations in cases of interest, though it is still an open question,  for example, whether or not for finite $\pi$ \emph{every} all totally non-spin manifolds of dimension at least $5$ admit metrics of positive scalar curvature. 
In the case of spin manifolds, one has $\Omega_n^{(\pi,u)} = \Omega_n^{Spin}(B\pi)$ and the analogous characterisation of manifolds admitting positive scalar curvature metrics is afforded by the Atiyah-Bott-Shapiro orientation
 \[\alpha:\Omega^{spin}_n(B\pi) \to ko_n(B\pi),\]
a result due to F\"uhring and Stolz; here $ko$ denotes connective, real K-theory. Already useful by itself (for instance $\alpha$ again usually has a very large kernel) it can also be seen as a way of attacking the Gromov-Lawson-Rosenberg conjecture, which attempts a full classification of spin manifolds that admit positive scalar metric in terms of operator, rather than topological, K-theory.

A similar result is currently lacking for almost spin manifolds and this is the principal motivation for the present paper. We shall identify the cobordism group $\Omega^{(\pi,u)}_*$ with a certain twisted spin cobordism group $\Omega^{Spin}_*(B\pi;u)$ in the case of an almost spin manifold. A twisted version of the Atiyah-Bott-Shapiro orientation with values in twisted $K$-theory then makes it possible to formulate the following

\begin{uconj}[Stolz 1995]
Let $M$ be a connected, closed, smooth, almost spin manifold $M$ of dimension $n\ge 5$ with normal $1$-type $(\pi,u)$ such that
$$0 = {\alpha}_K(M) \in ko_n(B\pi;u).$$
Then $M$ carries a metric of positive scalar curvature.
\end{uconj}

The characterisation of positive scalar curvature manifolds in terms of twisted K-theory analogous to those above would be an immediate consequence.

Presently, the only known ansatz for attacking this conjecture lies in an application of stable homotopy theory, essentially following the lines of Stolz' ideas in \cite{St2}. To approach a generalisation of his work recall that, just as homology theories are represented by spectra, twisted homology theories are represented by parametrised spectra; roughly speaking a parametrised spectrum $E$ is a collection of spectra varying over a parameter space $B$ and the associated twisted homology functor produces from a space $X$ and a twisting datum $\zeta \colon X \rightarrow B$ groups $E_*(X;\zeta)$, for example the twisted spin and K-theory groups above. 

In this article we will give explicit models for the twisted spin cobordism spectrum, which we name $\MSpinK$, and for the twisted real $K$-theory spectrum, which we name $\KOK$. Both are parametrised spectra over the classifying space $K$ for the projective orthogonal group of some infinite dimensional Hilbert space, which is well-known to record the kind of twists used in the forumulation of Stolz' conjecture. These models are built in an operator-theoretic fashion. As a result the twisted version of the Atiyah-Bott-Shapiro orientation $\MSpinK \rightarrow \KOK$ we construct has a straightforward operator theoretic interpretation.

When Stolz originally formulated this conjecture he used slightly different language, and an interpretation of the relevant cobordism groups in his context can be found in Stolz' preprint \cite{St3}. A twisted orientation map as above also already appeared in \cite{Jo0} using the model for $K$-theory, which later was published in \cite{Jo}. We would like to emphasise that using recent machinery in stable homotopy also provides methods to realise the desired homotopy theoretic objects, however, the relation to the geometry in that context is less immediate.
 
To carry over Stolz' line of argument from \cite{St2} to the twisted setting now requires a solid understanding of the parametrised spin bordism spectrum. As a first step in the analysis of the parametrised spin cobordism spectrum we use variants of the $KO$-valued Pontryagin classes of \cite{AnBrPe} to produce maps of parametrised spectra 
$\overline{\theta}_J: \MSpinK \rightarrow \KOK$, that for every partition $J$ admit lifts $\overline{\theta}_J$ to certain connective covers $\koK\langle n_J \rangle$ of $\KOK$. Using these we obtain the following generalisation of the Anderson-Brown-Peterson splitting. 
 
\begin{uthm}[\ref{genspli}]
For any choice of lifts $\overline{\theta}_J: \MSpinK \rightarrow \koK\langle n_J \rangle$ there exist maps $x_i: \MSpinK \longrightarrow K \times \Sigma^{n_i} H\mathbb Z/2$, such that the combined map
$$\MSpinK \longrightarrow \left[\prod_{J} \koK\langle n_J \rangle\right] \times \left[\prod_i K \times \Sigma^{n_i} H\mathbb Z/2\right]$$
is a $2$-local equivalence and in particular induces an isomorphism of twisted homology theories after localisation at $2$.
\end{uthm}

We then use this splitting to describe the mod $2$ cohomology of $\MSpinK$ and $\koK$ (the connective version of $\KOK$). Being parametrised spectra over $K$, their mod $2$ cohomology is a module over the extended Steenrod algebra $\underline{\mathcal A} = H^*(K) \sd \mathcal A$, the semidirect product of the Steenrod algebra $\mathcal A$ acting on $H^*(K)$. 
The algebra $\underline{\mathcal A}$ may be regarded as the stable cohomology operations of ordinary mod $2$ cohomology evaluated on spectra parametrised by $K$. The main structural result here is that they are induced along a certain embedding $\varphi: \mathcal A(1) \hooklongrightarrow \underline{\mathcal A}$. Here $\mathcal A(1)$ denotes the subalgebra of $\mathcal A$ generated by $Sq^1$ and $Sq^2$, but the embedding is \emph{not} the evident one into the left factor.

\begin{uthm}[\ref{cohok}]
Evaluation at the unique non-trivial class in lowest degree gives isomorphisms 
\begin{align*}
\underline{\mathcal A}_\varphi \otimes_{\mathcal A(1)} \mathbb Z/2        & \longrightarrow H^*(\koK,\mathbb Z/2) \\
sh^{2} \underline{\mathcal A}_\varphi \otimes_{\mathcal A(1)} \mathcal A(1)/\langle Sq^3\rangle & \longrightarrow H^*(\koK\langle 2 \rangle,\mathbb Z/2)
\end{align*}
where $\underline{\mathcal A}_\varphi$ denotes the twisted Steenrod algebra viewed as a right module over $\mathcal A(1)$ via $\varphi$, and $sh^{2}$ denotes a degree shift by $2$.
\end{uthm}

Our convention here is that for $E$ a space, or (parametrised) spectrum the $n$-connective cover $E\langle n \rangle$ has vanishing (fibrewise) homotopy groups below, but not necessarily in, degree $n$.

\begin{ucor}[\ref{cohom}]
The $\underline{\mathcal A}$-module $H^*(\MSpinK, \mathbb Z/2)$ is an extended $\mathcal A(1)$-module.
\end{ucor}

These two theorems provide powerful tools, both conceptually and computationally. As an illustration we would like to mention that they can be used to generalise to the twisted context the theorem of Hopkins and Hovey \cite{HoHoHo}, that the Atiyah-Bott-Shapiro orientation $\alpha\colon MSpin \rightarrow KO$ induces a natural isomorphism \[\Omega^{Spin}_*(-) \otimes_{\Omega_*^{Spin}} KO_* \to KO_*(-)\] of homology theories despite $KO_*$ not being a flat $\Omega^{Spin}_*$-module, see the forthcoming \cite{BJKS}. Concerning our main objective, however, our investigation also shows - much to the authors' surprise - that many convenient observations of Stolz that are used in \cite{St2} do not carry over to the twisted situation; one instance of this is given in Theorem \ref{nogo}. We do, however, believe that these unexpected difficulties are surmountable and that Stolz' approach indeed can be carried over, one way or the other, to the case of almost spin manifolds. We hope to come back to it in the near future.

\subsection*{Organisation of the paper}
In Section 2 we review the (mostly well-known) method of Gromov and Lawson to address the question of existence of a positive scalar curvature metric on a manifold in the required generality. We also carry out the translation into parametrised homotopy theory in this section. Details on parametrised homotopy theory then are given at the start of Section 3, before we  define the main objects of study, namely the twisted spin cobordism spectrum $\MSpinK$ and the twisted real $K$-theory spectrum $\KOK$. In Section 4 we present the parametrised Atiyah-Bott-Shapiro orientation $MSpin_K \to \koK$ and discuss fundamental classes in twisted spin cobordism. Section 5 then contains the generalised Anderson-Brown-Peterson splitting, while Section 6 contains our results on the $\mathbb{Z}/2$-cohomology of $\MSpinK$. The final section discusses the application of these results to Stolz' conjecture. 

We accompanied the paper with three appendices. The first one contains a simple, lengthy computation from Section 6, which we seperated to improve readibility of the main text. In the second we give a proof of the cobordism invariance theorem for positive scalar curvature metrics, which we (although the result is well-known) could not find in the literature, and the last one provides a comparison of the geometric and homotopical methods of obtaining  twisted $Spin$ cobordism and twisted $K$-theory.\\

\subsection*{Acknowledgements.}
We would first and foremost like to thank Stephan Stolz for freely sharing his insights, great and small. Furthermore, we want to thank Irakli Patchkoria, Ulrich Pennig and Steffen Sagave for several helpful discussions about the contents of the last appendix. The first author would also like to thank Daniel Kasprowski for many fruitful discussions during our time as graduate students. 

The first author was supported by a postdoctoral scholarship of the German Academic Exchange Service (DAAD) during his year at the University of Notre Dame, where parts of this project were finished.

\section{The relation of positive scalar curvature to cobordism}

\subsection{From surgery to cobordism.}\label{from-surgery-to-cobordism}

Given the surgery theorem of Gromov and Lawson from \cite{GrLa} the next step is to investigate the class of manifolds from which a given manifold $M$ arises by surgeries in codimension $\geq 3$. There is a well-known strategy to do so using cobordism theory. The appropriate type of cobordism theory depends on the manifold $M$, and classically one uses the concept of $B$-bordism. This notion is built on a fibration $B \to BO$ and uses lifts of explicit Gau\ss{} maps for embeddings of $M$ into euclidean space to record the extra structure on $M$, see for example \cite[Chapter II]{Stong}.

We will use a slightly different but equivalent notion for the following two reasons. On the one hand it will turn out rather inconvenient (for example in \ref{choice-of-bordism-gropus}) to require the reference map $B \to BO$ to be a fibration. On the other, for the cycle description of twisted bordism theories in \ref{fund} requiring the reference map $M \rightarrow BO$ (which is a required datum for a cycle) to be the Gau\ss{} map of an embedding does not seem appropriate. There is an additional diagram such a map is required to make commutative and it seems to come with undue difficulties to try and arrange this by adjusting an embedding instead of merely adjusting the reference map directly. We will work directly with the object the map from $B \rightarrow BO$ is supposed to classify, a \emph{stable vector bundle}, and describe the relevant structure in terms of bundle data.

To this end, recall that a stable vector bundle $\xi$ of dimension $k \in \mathbb Z$ over a space $X$ is given by an exhaustive filtration (by cofibrations) $X_{-k} \subseteq X_{-k+1} \subseteq ...$ together with a sequence of vector bundles $\xi_i\to X_i$ of dimension $i+k$ and isomorphisms $\xi_i \times {\mathbb R} \cong \xi_{i+1}|_{X_{i}}$ for every $i$. A map $\xi \to \xi'$ between stable vector bundles corresponds to a compatible family of fibrewise isomorphisms $\xi_i \to \xi'_i$. A concordance between two stable vector bundles on the same space $X$ is another stable vector bundle on $X \times I$, whose restrictions to the boundaries are the original stable vector bundles. There are evident trivial stable bundles of arbitrary dimension (with constant filtration of the base) and any actual vector bundle may be regarded as a stable one of the same dimension (with filtration consisting of $\emptyset$ and $X$ only). Furthermore, one can form the (external) direct sum of stable bundles by glueing the various bundles over $X_i \times Y_{j-i}$ for varying $i$ together to give the $j$th step of a stable bundle over $X \times Y$. 

For a cell complex $X$ a concordance class of stable vector bundles corresponds to a homotopy class of maps $X \rightarrow \mathbb Z \times BO$, with the classifying map of a $k$-dimensional bundle mapping into $\{k\} \times BO \subseteq \mathbb Z \times BO$. Here we regard $\{k\} \times BO$ as the colimit over cofibrations $BO(k+i)\to BO(k+i+1)$, so that it can be equipped with a universal stable vector bundle of dimension $k$ whose $i$th step is the universal bundle $\gamma_{k+i}$ over $BO(k+i)$. Given a stable vector bundle of dimension $k$ over $X$ one can find classifying maps $X_i\to BO(k+i)$ for the $\xi_i$ so as to achieve a map $X \to BO$. Conversely, given a map $f:X \to BO$ we define $X_i$ to be the biggest subcomplex of $X$ which is contained in $f^{-1}(BO(k+i))$. One then obtains a stable vector bundle $\xi$ over $X$ by declaring $\xi_i = f|_{X_i}^*(\gamma_{i})$.

A \emph{stable normal bundle} of a closed smooth $n$-manifold $M$ we then define to be a stable vector bundle $\nu$ of dimension $-n$ over $M$, together with a concordance from $\nu \oplus TM$ to the trivial stable vector bundle over $M$. Now let there be given a reference space $B$ with a stable vector bundle $\xi$ of dimension $0$. For us a $\xi$-structure on an $n$-manifold $M$ is by definition a choice of stable normal bundle $\nu$ of $M$, together with a map of stable vector bundles $\nu \oplus \mathbb R^n \rightarrow \xi$, up to concordance. \\

In the next two paragraphs we will explain how this notion is equivalent to the notion of a $B$-structure on $M$. As we will not make use of this comparison in the remainder of the paper the credulous reader may safely skip ahead to Theorem \ref{bordi}.

Consider then a fibration $\theta \colon B \rightarrow BO$. An embedding $e$ of $M$ into euclidean space canonically produces a stable normal bundle in the sense just described and together with a lift $M \rightarrow B$ of its Gau\ss{} map produces a map of stable bundles $\nu_e \rightarrow \theta^*(\gamma)$, and thus a $\theta^*(\gamma)$-structure. Conversely, given a stable normal bundle $\nu$ together with a map $\nu \rightarrow \theta^*(\gamma)$ pick an embedding $e$ of $M$ into high dimensional euclidean space. By the (stable) uniqueness of complementary bundles there is an essentially unique concordance between $\nu$ and $\nu_e$ compatible with the chosen one from $\nu \oplus TM$ to the trivial bundle and one obtains a $B$-structure on $M$ by taking the map $M \rightarrow B$ underlying $\nu \rightarrow \theta^*\gamma$ and rectifying it to be a lift of the Gau\ss{} map of $e$ via homotopy lifting along the fibration $\theta$.

For an arbitrary stable vector bundle $\xi$ over $B$, pick a classifying map $c_{\xi}: \xi \to \gamma$ and replace its base component $\overline{c_\xi} \colon B \rightarrow BO$ by a fibration $\theta \colon B' \rightarrow BO$. Given then a manifold $M$ and a bundle map $l\colon \nu \rightarrow \xi$, pick an embedding $e$ of $M$ into euclidean space and a concordance between $\nu_e$ and $\nu$ as before and a classifying map $c_\nu: \nu \rightarrow \gamma$. There is an essentially unique homotopy of vector bundle maps $\nu \to \gamma$ between $c_\xi \circ l$ and $c_{\nu}$, since the space of bundles maps $\nu \to \gamma$ is contractible. On the level of base spaces the $\xi$-structure yields a map $\overline{l}: M\to B$, and the choices just made give maps $\overline{c}_\nu: M \to BO$ as well as a homotopy between $\overline{c}_\xi \circ \overline{c}$ and $\overline{c}_\nu$ and the concordance between $\nu$ and $\nu_e$ gives an essentially unique homotopy between $\overline c_\nu$ and the Gau\ss{} map of $e$ (since the space of bundle map from the concordance to $\gamma$ is contractible). These pieces of data assemble into a lift of $M \rightarrow B'$ of the Gau\ss{} maps. Conversely, given a manifold $M$, an embedding $e$ and a lift $l': M \rightarrow B'$ of its Gau\ss{} map we can  lift $l'$ essentially uniquely to a map $l \colon M \rightarrow B$ together with a homotopy from $l'$ to $i \circ l$, where $i \colon B \rightarrow B'$ is the canonical homotopy equivalence. But then we can use $l^*(\xi)$ as the normal bundle of $M$: It comes with a canonical map to $\xi$ and the concordance from $l^*(\xi) \oplus TM$ to the trivial bundle arises via the homotopy of $\overline{c_\xi} \circ l = \theta \circ i \circ l$ to $\theta \circ l'$, which is the Gau\ss{} map of $\nu_e$ by assumption. 

With either notion of a $\xi$-manifold one then has:

\begin{thm}\label{bordi}
Let $M$ be a smooth, closed $\xi$-manifold of dimension $\geq 5$, whose underlying structure map $M \rightarrow B$ is a $2$-equivalence. If the $\xi$-cobordism class of $M$ contains a manifold admitting a metric of positive scalar curvature, then also $M$ admits such a metric.
\end{thm}
The result is well-known, and the content of the theorem is explicitly stated in \cite[Bordism Theorem 3.3]{RoSt} as well as \cite[Theorem 1]{Kr}. In \cite{RoSt} the authors mention in the text that a proof can be obtained by an adaptation of Rosenberg's argument in \cite[Proof of Theorem 2.2]{Ro}, which covers the spin case. The proof presented in \cite{Kr}, however, contains an incorrect argument. Since to the best of our knowledge no complete proof appears in the literature we provide one in the second appendix following the line of argument of \cite{RoSt}.

For a given manifold $M$ one may apply the above theorem directly to some stable normal bundle $\nu$ of $M$. However, in order to use computations of corresponding $\xi$-cobordism groups efficiently one chooses the reference bundle in a way that picks up as little information from the specific manifold $M$ at hand as possible. Canonical choices arise from second Moore-Postnikov factorisations of a classifying map 
$$M \rightarrow B \rightarrow BO$$
of $\nu$; by definition this means that the map $M \rightarrow B$ is a $2$-equivalence and $B \rightarrow BO$ is a $2$-coequivalence (i.e.\ injective on $\pi_{2}$ and an isomorphism on $\pi_{n}$ for $n>2$). These two properties in particular determine the fibre homotopy type of $B \rightarrow BO$. This fibre homotopy type is called the  \emph{normal $1$-type} of the manifold $M$ in \cite{Kr}. We now can choose a stable vector bundle $\xi$ on $B$, whose concordance class is represented by the map $B \to BO$. In view of Kreck's definition we will call the pair $(B, \xi)$ a representative of the normal $1$-type of $M$. Essentially by construction the manifold $M$ can be equipped with a $\xi$-structure whose underlying map $M \to B$ is the one in the above Moore-Postnikov decomposition, and the manifold $M$ with this $\xi$-structure meets the assumptions of Theorem \ref{bordi}. \\

One typically distinguishes two separate cases, namely the case where the universal cover of $M$ admits a $Spin$-structure and the case where the universal cover does not. Note that the conclusion of Theorem \ref{bordi} is independent of the choice of a specific $\xi$-structure on the manifold: what is relevant is simply the existence of a corresponding $\xi$-structure.
From that point of view we will also distinguish between manifolds which admit spin structures and manifolds which are equipped with a specific spin structure: we will call a manifold spinnable in first case, and spin in the latter. If the universal cover of $M$ is spinnable we call $M$ almost spinnable, and following standard terminology we call a manifold totally non-spin if it is not almost spinnable.
The normal $1$-types for the two cases, namely totally non-spin manifolds on the one hand and almost spinnable manifolds on the other, were already considered by Kreck in \cite[Proposition 2]{Kr}, so we will be brief and just mention the results.

Recall first that the first two Stiefel-Whitney classes of the stable normal bundle of $M$ are given by $(w_1(M), w_2(M) + w_1(M)^2)$. Given a connected manifold $M$ with fundamental group $\pi$ there is a unique class $u_1 \in H^1(B\pi, \mathbb Z/2)$ that pulls back to $w_1(M) \in H^1(M,\mathbb Z/2)$ 
under the map classifying of the universal cover of $M$. Representing $w_1\in H^1(BO,\mathbb Z/2)$ and $u_1$ as maps into an Eilenberg-MacLane space we define $B_1$ as the homotopy pullback in following diagram

$$\begin{xy}\xymatrix@-1pc{B_1\ar[d] \ar[r]^c      & BO \ar[d]^{w_1}\\
                      B\pi \ar[r]_-{u_1} & K(\mathbb Z/2,1)}
\end{xy}$$

\begin{prop}
In case $M$ is totally non-spin a pair consisting of $B_1$ and any stable vector bundle whose concordance class is represented by the map $c \colon B_1 \to BO$ is a representative of the normal $1$-type of $M$. 
\end{prop}
In the case when $M$ is almost spinnable the map $M\to B_1$ is not onto on $\pi_2$. However, in that case there is a unique class $u_2 \in H^2(B\pi,\mathbb Z/2)$ pulling back to $w_2(M)+w_1(M)^2 \in H^2(M, \mathbb Z/2)$ under the canonical map $M \rightarrow B\pi$. Representing the cohomology classes by actual maps we obtain a corresponding homotopy pullback diagram
$$\begin{xy}\xymatrix@-1pc{B_2 \ar[d] \ar[rr]^c    &            & BO \ar[d]^{(w_1,w_2)}\\
                      B\pi \ar[rr]_-{(u_1, u_2)} & & K(\mathbb Z/2,1) \times K(\mathbb Z/2,2).}
\end{xy}$$
\begin{prop}\label{almost-spinnable-case}
In case $M$ is almost spinnable a pair consisting of $B_2$ and any stable vector bundle whose concordance class is represented by the map $c \colon B_2 \to BO$ is a representative of the normal $1$-type of $M$.
\end{prop}

In either case, after picking once and for all representatives for $w_1$ and $w_2$, the representatives of the normal $1$-type of $M$ essentially just depend on the fundamental group $\pi$ and a map $u: B\pi \to K$, where $K$ stands for $K(\mathbb Z/2,1)$ or $K(\mathbb Z/2,1) \times K(\mathbb Z/2,2)$, respectively. To work with both cases uniformly and to emphasise the dependence on these data we will from now on write $(\pi,u)$ for the representative of the $1$-type of $M$, denote its base space by $B(\pi,u)$, the associated stable bundle by $\gamma_{\pi,u}$ and the corresponding cobordism groups by $\Omega^{(\pi,u)}_*$. Note that by construction $M$ can be equipped with a $(\pi,u)$-structure, however, the $(\pi,u)$-structure is not unique in general. 

The following theorem is a special case of Theorem \ref{bordi} which we will work with in the sequel.

\begin{thm}\label{bord}
Let $M$ be a smooth, closed, connected manifold $M$ of dimension $\geq 5$ with normal $1$-type determined by $(\pi,u)$. Then $M$ admits a metric of positive scalar curvature if and only if for some (and then any) choice of $(\pi,u)$-structure on $M$, whose underlying map $M \to B(\pi,u)$ is a $2$-equivalence, we can represent the corresponding $(\pi,u)$-cobordism class by some manifold admitting a positive scalar curvature metric.
\end{thm}

We would like to emphasise that by definition the cobordism group $\Omega^{{(\pi, u)}}_*$ depends on the choice of the map $u$, and not just on the cohomology class $u$ it represents. Different choices of $u$ yield isomorphic cobordism groups, however, isomorphisms between the corresponding cobordism groups cannot be chosen canonically, so there is no straight way to compare cobordism classes in the cobordism rings for different choices of $u$. But this is also not relevant for the statement of the theorem. 

The theorem specialises as it should in the two well-studied cases. When $M$ is orientable and totally non-spin we find $\Omega^{(\pi, u)}_*\cong \Omega^{SO}_*(B\pi)$. In the case where $M$ is spinnable one gets $\Omega^{(\pi, u)}_* \cong \Omega^{Spin}_*(B\pi)$. As we will briefly touch upon in section \ref{ocr}, Stolz provided a more geometric interpretation of the groups $\Omega^{(\pi,u)}_*$ in general. Starting in the next section we, however, shall pursue a homotopy theoretic description (finally obtained in corollary \ref{mind}), which will eventually yield structural information in the form of our generalisation of the Anderson-Brown-Peterson splitting in section \ref{calc}. 

\subsection{Interpretation via twists.}
\label{choice-of-bordism-gropus}

To investigate the bordism groups $\Omega_*^{(\pi,u)}$ we use the set-up of parametrised homotopy theory. To make this explicit we first need to introduce a bit of notation.

Let $\xi$ be a vector bundle over a space $X$ that comes equipped with a map $q: X \rightarrow K$, where $K$ is a generalized Eilenberg space as in the previous section. In particular the total space of $\xi$ can naturally be regarded as a space over $K$. Let $M_X(\xi)$ be its fibrewise one-point compactification of the map $\xi \to X$. The space $M_X(\xi)$ is a sphere bundle (over $X$) which comes equipped with a canonical projection to $X$ as well as the section at infinity $s:X \to M_X(\xi)$ which maps a point $x \in X$ to the base point of the one-point compactification in the fibre over $x$. The \emph{fibrewise Thom space over $K$} is obtained from $M_X(\xi)$ by identifying two points $s(x)$ and $s(x')$ in the image of the infinity section, if $q(x) = q(x')$. We denote this fibrewise Thom space over $K$ by $M_K(\xi)$; it can indeed be identified with the pushout of the two maps $s$ and $q$, and thus can canonically be given the structure of a space over and under $K$: 
$$\begin{xy}\xymatrix@-1pc{X \ar[d]_{s} \ar[r]^{q}      & K \ar[d]^{\overline{s}}\\
                      M_X(\xi) \ar[r]_-{\overline{q}} & M_K(\xi) }
\end{xy}$$
We shall usually identify $K$ with the subspace $\overline{s}(K)$ in $M_K(\xi)$. It is also convenient to think of the underlying set of $M_K(\xi)$ as the union of the Thom spaces $M(\xi|_{q^{-1}(k)})$, where $k$ runs over the points in $K$. Note that $M_K(\xi) = X \sqcup K$ for $\xi$ the trivial zero-dimensional vector bundle over $X$.

Given two spaces $S$ and $T$ over $K$ equipped with sections $s$ and $t$ respectively let $S \wedge_K T$ denote the fibrewise smash product, obtained as the following pushout
$$\xymatrix@-1pc{s(K) \times_K T \cup S \times_K t(K) \ar[r] \ar[d] & S \times_K T \ar[d] \\
            K \ar[r] & S \wedge_K T}$$
which again is a space over $K$ with a section. The next lemma then is an easily verified generalisation of the property that the Thom space $M(\xi \times \xi')$ for arbitrary vector bundles $\xi$ and $\xi'$ is canonically homeomorphic to the smash product $M(\xi) \wedge M(\xi')$ of the individual Thom spaces. 
\begin{lem}\label{smash-over-K}
Given vector bundles $\xi$ and $\xi'$ over spaces $X$ and $X'$ and a pullback diagram
$$\begin{xy}\xymatrix@-1pc{B \ar[d]_{p'} \ar[rr]^{p} && X \ar[d]^{q}\\
                      X' \ar[rr]_{q'}            && K}
\end{xy}$$
there is a canonical homeomorphism 
$$M(p^*\xi \oplus {p'}^*\xi') \cong \big[ M_K(\xi) \wedge_K M_K(\xi') \big] / K$$
where the quotient on the right hand side is obtained by identifying the image of the section $\overline{s}$  at infinity to a single point $[K]$. 
\end{lem}

\begin{proof}[Proof of Lemma \ref{smash-over-K}]
The isomorphism sends the point $[K]\in\big[ M_K(\xi') \wedge_K M_K(\xi') \big] / K$ to the point at infinity in $M(p^*\xi \oplus {p'}^*\xi')$, while the complement of the point $[K]$, which can be identified with $\xi \times_K \xi'$, is mapped canonically to $p^*\xi \oplus {p'}^*\xi'$, which is the complement of the point at infinity in $M(p^*\xi \oplus {p'}^*\xi')$. 
\end{proof}
Below we shall apply lemma \ref{smash-over-K} to a specific sequence of examples. Recall from the preceding section that a representative of the normal $1$-type of a manifold $M$ is a stable vector bundle $\gamma^{\pi,u}$ over a suitable space $B(\pi,u)$. 
The structure of the stable vector bundle $\gamma^{\pi,u}$ by definition then yields a commutative diagram
$$\begin{xy}\xymatrix@-1pc{\cdots \ar[r] & \gamma^{\pi,u}_i \ar[d] \ar[r] &   \gamma^{\pi,u}_{i+1} \ar[d] \ar[r] & \cdots \ar[r]  & \gamma^{\pi,u} \ar[d]\\
\cdots \ar[r] & B(\pi,u)_{i} \ar[r] &   B(\pi,u)_{i+1} \ar[r] & \cdots \ar[r] & B(\pi,u)}
\end{xy}$$
The corresponding vector bundle maps $\gamma^{\pi,u}_{i} \times\mathbb R \to \gamma^{\pi,u}_{i+1} $ which also come with the structure of a stable vector bundle induce maps
$$ M(\gamma^{\pi,u}_{i}) \wedge S^1 \cong M(\gamma^{\pi,u}_{i} \times \mathbb R) \to M(\gamma^{\pi,u}_{i+1}) 
$$
We thus obtain a spectrum $M(\pi,u)$, given by the sequence of spaces $M(\gamma^{\pi,u}_{i})$ and the above structure maps. The Pontryagin-Thom theorem gives an isomorphism
$$ \Omega^{(\pi,u)}_* \cong \pi_* M(\pi,u).$$
 
Assume now that $w: BO \to K$ is a fibration. This allows us to define $B(\pi,u)$ to be the actual pullback of $BO$ along $u:B\pi\to K$, as this makes its defining square a homotopy pullback as required. Moreover, we have a canonical filtration of $B(\pi,u)$ given by pullbacks 
$$\begin{xy}\xymatrix@-1pc{B(\pi,u)_{i} \ar[d] \ar[r]^{p_i}    &  BO(i) \ar[d]^w\\
                      B\pi \ar[r]_-{u} &  K}
\end{xy}$$
while the pulbacks $p_i^*(\gamma_i)$ of the universal vector bundles $\gamma_i$ over $BO(i)$ provide a stable vector bundle $\gamma^{\pi,u}$ which is compatible with this filtration. Note that the inclusions $B(\pi,u)_{i} \subseteq B(\pi,u)_{i+1}$ are indeed cofibrations since the inclusions $BO(i) \subseteq BO(i+1)$ are cofibrations by assumption (see e.g.\ \cite{Ki}). In the following we will always work with a representative of the normal $1$-type of $M$ which arises this way.

Applying lemma \ref{smash-over-K} to the $0$-dimensional vector bundle over $B\pi$ and the universal bundle $\gamma_i$ over $BO(i)$ we find the following corollary:

\begin{cor}\label{twicob}
The maps $w$ and $u$ induce homotopy equivalences
	 $$M(\gamma_i^{\pi,u}) \simeq \big[ (B\pi \sqcup K) \wedge_K M_K\gamma_i \big] /K$$
\end{cor}

Note that all spaces $M_K\gamma_i$ come equipped with a reference map to $K$, and thus we can regard them as spaces over $K$. In fact the family of spaces $M_K\gamma_i$ can be equipped quite naturally with the structure of a so-called parametrised spectrum $M_K\gamma$, a notion that will be introduced in section \ref{section_tsc}. In the case of almost spin manifolds the parametrised spectrum  $M_K\gamma$ is equivalent to $\MSpinK$ from the introduction, and we will write $(\MSpinK)_i$ for $M_K\gamma_i$ in that case.  Just like an ordinary spectrum $E$ yields a generalised homology theory via
$$E_*(X) = \pi_*\big[(X \sqcup *) \wedge E\big]$$
for nice enough spaces $X$, parametrised spectra like $M_K\gamma$ represent twisted homology theories for spaces with a reference map to $K$. We will introduce twisted homology from this homotopical viewpoint that is not restricted to bordism theories in \ref{paraspec}. For a fixed point $k \in K$ the restriction of such a functor to spaces, whose reference maps is constant with value $k$, is a homology theory in the usual sense. For the parametrised spectrum $\MSpinK$ these restrictions give ordinary spin bordism. We will give a general description of $MSpin_*(X,\zeta)$ for $\zeta \colon X \rightarrow K$ using manifold cycles in Section \ref{fund}. 

\begin{rem}
The general set-up for a (homotopical) twisted bordism theory starts with a map \emph{out} of $BO$, say $w \colon BO \rightarrow K$, along which one may form the fibrewise Thom spaces $M_K\gamma_i$ as above. The cycle description for twisted spin cobordism given in \ref{fund} immediately extends to that case and if we denote the projection of the homotopy fibre of $w$ by $\theta \colon B \rightarrow BO$, the restriction of $M_K\gamma_*(-;-)$ to spaces with constant structure map is exactly $\theta$-bordism. To focus on matters at hand, we will, however, refrain from discussing details of the cycle models as the previous section encoded all their relevant features into the fibrewise Thom spaces.
\end{rem}

\subsection{Obstructions, conjectures and known results.}\label{ocr}
In this section we put our considerations into greater context (which is well-known to the experts). The aim is to describe the connection of our work with a variant of the Gromov-Lawson-Rosenberg conjecture, which we recall below. Let us first present yet another interpretation of the cobordism groups considered above due to Stolz \cite{St3}. \\

Stolz constructs Lie groups $G(n,\gamma)$ associated to a natural number $n$ and a supergroup $\gamma$. A supergroup he defines to be a triple $(G,w,c)$, where $G$ is a group, $w:G \rightarrow \mathbb Z/2$ is a group homomorphism and $c \in G$ is a central element in the kernel of $w$. Every vector bundle $E$ determines a supergroup $\underline{\pi}(E)$ with $G$ a certain extension of the fundamental group $\pi$ of the base space, and $w$ the orientation character composed with the projection of $G$ onto $\pi$. Both $c$ and the extension $G$ are related to the spinnability of the vector bundle $E$ and the spinnability of its pullback to the universal cover of the base. The case of interest here is the case where $E$ is the (stable) normal bundle of a manifold $M$, and for simplicity we write $\underline\pi$ for the corresponding supergroup in the sequel. Following Stolz' construction it is not hard to see that the space $B(\pi,u)_n$ introduced in section \ref{choice-of-bordism-gropus} is indeed a classifying space of $G(n,\underline\pi)$. It then follows that the cobordism groups $\Omega^{(\pi,u)}_*$ of subsection \ref{from-surgery-to-cobordism} can be intrepreted as groups of equivalence classes of $G(\underline{\pi})$-manifolds modulo $G(\underline{\pi})$-cobordisms. We will not make use of this identification outside this section, and therefore leave the verification to the reader for now. Details will be contained in \cite{HeJo}.

While the cobordism techniques discussed so far give a powerful method for proving existence of metrics with positive scalar curvature, there are also well-known obstructions given by indices of Dirac-type operators. Given a spin manifold $M$ of dimension $n$ with fundamental group $\pi$ these were coalesced into a single invariant $A(M) \in KO_n(C^*_r\pi)$ by Rosenberg in \cite{Ro0}, and this invariant was further generalised to an invariant $A(M) \in KO_n(C^*_r\underline{\pi})$ for almost spin manifolds by Stolz in \cite{St3} using the canonical $G(\underline{\pi})$-structure such a manifold carries; here $KO_n$ denotes topological $KO$-homology and $C^*_r\pi$ denotes the reduced group $C^*$-algebra of the group $\pi$, while $C^*_r\underline{\pi}$ is a certain generalisation of it for the supergroup $\underline{\pi}$ from \cite{St3}.
Improving upon a conjecture of Gromov and Lawson, Rosenberg then made the following conjecture.

\begin{conj}[Gromov-Lawson-Rosenberg conjecture]\label{glr}
A connected, closed, almost spin manifold $M$ of dimension $n \ge 5$ supports a metric of positive scalar curvature if and only if 
$$0 = A(M) \in KO_n(C^*_r\underline{\pi})$$
\end{conj}

While there are counterexamples even to the classical version for spin manifolds (\cite{Sc}) it is an open problem for manifolds with finite fundamental groups and indeed has been verified for several cases (e.g.~see \cite{BoGiSt} or \cite{BoGi}; there are also counterexamples with non-trivial twists, see \cite{JoSc}). For $Spin$-manifolds Rosenberg furthermore showed that his index invariant in $KO_n(C^*_r{\pi})$ depends only on the $Spin$-cobordism class of $M$ in $\Omega^{Spin}_n(B\pi)$ and indeed factors as
$$\Omega^{Spin}_n(B\pi) \stackrel{\alpha}{\longrightarrow} ko_n(B\pi) \stackrel{per}{\longrightarrow} KO_n(B\pi) \stackrel{asbl}{\longrightarrow} KO_n(C^*_r\pi)$$
where $\alpha$ denotes the Atiyah-Bott-Shapiro orientation, $per$ the canonical periodisation map and $asbl$ a certain assembly map. A similar decomposition using twisted $Spin$-cobordism and $KO$-theory also works more generally for almost spinnable manifolds, and a full discussion of the latter will be provided in \cite{HeJo}. \\
This decomposition of $A$ together with conjecture \ref{ourconj} from the introduction, theorem \ref{bord} and the following observation allows for Rosenberg's conjecture to be attacked by cobordism theory:

\begin{lem}
Let $\xi$ be a stable vector bundle over a space $B$ and $a: \Omega_n^\xi \rightarrow G$ a group homomorphism with $n \geq 5$. Assume that every $n$-dimensional $\xi$-manifold in the kernel of $a$, whose structure map to $B$ is a $2$-equivalence, admits a metric of positive scalar curvature. 

Then an $n$-dimensional $\xi$-manifold $M$, whose structure map $M \rightarrow B$ is a $2$-equivalence, admits a positive scalar curvature metric if and only if $a([M]) = a([N])$ for some $n$-dimensional $\xi$-manifold $N$ that does admit a positive scalar curvature metric. 
\end{lem}
\begin{proof}
	
To prove the non-trivial direction assume $a(M) = a(N)$. Then the difference $[M] - [N] = [M \sqcup \overline N]$ lies in the kernel of $a$ and by surgery we can, by the same procedure as in the proof of theorem \ref{bordi} in appendix B, produce a $\xi$-manifold $L$ that is $\xi$-cobordant to $M \sqcup \overline N$ and whose structure map is a $2$-equivalence. By assumption $L$ then carries a positive scalar curvature metric and since $[M] = [L \sqcup N]$, theorem \ref{bordi} yields the claim.
\end{proof}

As the untwisted version of Stolz's conjecture from the introduction (which we shall restate in \ref{ourconj} when all its constituents are defined) is known by the work of F\"uhring and Stolz (\cite{Fu, St2}) the lemma implies that the Gromov-Lawson-Rosenberg conjecture holds for manifolds with a particular fundamental group $\pi$ if and only if the kernel of the assembly map above can be spanned by $\alpha$-invariants of manifolds with positive scalar curvature. The positive results on the Gromov-Lawson-Rosenberg conjecture mentioned above are then mostly obtained by explicit comparison of $ko_n(B\pi)$ with $KO_n(C^*_r\pi)$.  In addition to being useful on its own, one can thus view Stolz' conjecture as a way of investigating conjecture \ref{glr} in the case of almost spin manifolds.\\
	By comparison, for a totally non-spin manifold $M$ Rosenberg tentatively predicted the existence of a metric with positive scalar curvature in general, since there are no Dirac-type operators associated with $M$. While this prediction also does not hold in general, the analogue of conjecture \ref{ourconj} in this case is known: If the orientation class $th([M]) \in H_n(B\pi)$ of an oriented manifold $M$ vanishes, then $M$ admits a metric of positive scalar curvature, compare \cite[Theorem 4.11]{RoSt2}. The corresponding statement using local coefficients in case $M$ is non-orientable holds as well. In that case one needs to consider the fundamental class in singular homology with local coefficients instead.

All in all, studying the existence of a positive scalar curvature metric for a given manifold has been reduced to the study of a much smaller homology group than the relevant bordism group, except for the case where the manifold at hand is almost spin but does not carry a spin structure itself. This paper can be regarded as a first step towards filling in this gap.

\section{Twisted spin cobordism}\label{section_tsc}

In this chapter we will present the models for the twisted $Spin$-cobordism spectrum and the twisted $KO$-spectrum that we will use and analyse in the sequel. First we will recall some basic facts on the parametrised stable homotopy category. Here we use the set-up of May and Sigur\dh{}sson, established in \cite{MaSi}. In particular, we will introduce the stable homotopy category over a base space $K$ and recall how spectra over a base space $K$ give rise to twisted (co)homology theories. Section \ref{prelim} then contains some preliminaries needed for the definition of the desired parametrised spectra in the following sections.
In section \ref{twispi} we finally introduce our model for the parametrised spin spectrum and prove 
$$\Omega^{(\pi,u)}_* \cong \MSpin_*(B\pi; u)$$
in \ref{mind}, where the right hand side denotes the homotopically defined spin cobordism groups of $B\pi$ twisted by $u$. It is this identification which drives the whole line of thought we present in this paper. In section \ref{twik} we present a similar spectrum representing twisted, real $K$-theory, whose connective covers appear in our generalisation of the Anderson-Brown-Peterson splitting. Throughout we will work within the category of compactly generated spaces.\\

We would like to point out that it will be crucial that the parametrised $Spin$-cobordism spectrum will be a module over the unparametrised one. To achieve this property within the set-up of May and Sigur\dh{}sson we indeed have to provide two models for the unparametrised $Spin$-cobordism spectrum. One named $MSpin$ (essentially from \cite{Jo2}), which is a ring spectrum, and a variant called $MSpin'$, which will be a module over $MSpin$ and in addition comes equipped with an action of a group whose homotopy type is $K(\mathbb Z/2,0) \times K(\mathbb Z/2, 1)$. Following May and Sigur\dh{}sson we will use this action to produce a parametrised spectrum via a Borel-construction. 

The individual terms of the ring spectrum $MSpin$ already very naturally enjoy an action of a group with the above homotopy type. That action, however, does not commute with the spectrum or ring structure  (but very nicely interacts with it). The fact that it cannot directly be used as a fibre in the parametrised spectrum  is somewhat unsatisfactory. This issue can probably be resolved in an adaptation of the May-Sigur\dh{}sson set-up.

Let us also mention that if one's goal was only to obtain parametrised spectra with the desired homotopy types one could equally well employ the techniques of \cite{AnBlGeHoRe} in an $\infty$-categorical set-up. It seems, however, that one looses the close connection of our parametrised spectra to geometry and classical index theory. For a discussion of the latter see the forthcoming \cite{HeJo}.\\

\subsection{Parametrised spectra after May and Sigur\dh{}sson.}\label{paraspec}

In this section we follow \cite{MaSi} and review some of the basic set-up of parametrised spectra. We essentially employ orthogonal spectra in their parametrised form as developed by them, however, we shall use right instead of left $\mathbb S$-modules since this fits better with the usual conventions from bundle theory. The entire section is written with a non-expert in mind and can probably safely be skipped by anyone with basic knowledge of parametrised homotopy theory. 

\begin{defi}
An ex-space over $K$ is a space together with a map to $K$ and a section of that map which should be thought of a choice of basepoint in each fibre. A map of ex-spaces $X$ and $Y$ is a map $X \rightarrow Y$ which preserves both fibres and their basepoints.

The external wedge product $X \overline\vee Y$ of ex-spaces $X$ and $X'$ over $K$ and $K'$ with sections $s$ and $s'$, respectively, is given by the following pushout 
$$\xymatrix@-1pc{K\times K' \ar[rr]^{id \times s'} \ar[d]_{s \times id} && K \times X' \ar[d] \\
          X \times K'\ar[rr] && X \overline\vee Y}$$
while the external smash product $\barwedge$ for ex-spaces $X$ over $K$ and $X'$ over $K'$ with sections $s$ and $s'$, respectively is given by the following pushout
$$\xymatrix@-1pc{X \overline\vee Y \ar[rr] \ar[d] && X \times Y \ar[d] \\
          K \times K' \ar[rr] && X \barwedge Y,}$$
both which describe ex-spaces over $K \times K'$. If $K = K'$ pulling the external wedge and smash product back along the diagonal yields the fibrewise wedge and smash product smash product. 
 We will denote the latter by $\wedge_K$ as in section \ref{choice-of-bordism-gropus}, but for the former will suppress the subscript $K$ and write $\vee$ instead of $\vee_K$ to avoid awkward notation in Section \ref{calc} (this is the only place it will appear again); no external wedge products will appear in the remainder of this paper.
Finally, there are function ex-spaces $F_K(X,Y)$ defined so that the functor $F_K(X,-)$ is right adjoint to $X \wedge_K -$. 
\end{defi}
 Note that given a space $X$ over $K$ the disjoint union $X \sqcup K$ is an ex-space in a canonical way.
\begin{defi}
A parametrised orthogonal sequence over $K$ is a sequence of ex-spaces $X_n$ over $K$ together with a continuous left $O(n)$-action on $X_n$ by ex-maps for each $n \in \mathbb N$. A map between two such objects is a sequence of equivariant ex-maps.
The pre-smash product of two orthogonal sequences $X$ and $Y$ is the sequence given by the following wedge product of balanced smash product spaces over $K$
$$(X \wedge_K Y)_n = \bigvee_{i = 0}^n O(n)_+ \underset{O(p) \times O(n-p)}{\barwedge} X_p \wedge_K Y_{n-p}$$ with action induced by the left operation on the left factor. Similarly, there is an external smash product $\barwedge$ of orthogonal sequences with input two orthogonal sequences, one over $K$ and one over $K'$, and output one over $K \times K'$. The two kinds of pre-smash products are related by $X \wedge_K Y = \Delta^*(X \barwedge Y)$ where $X,Y$ are sequences over $K$, $\Delta$ is the diagonal map of $K$ and $\Delta^*$ means the degreewise pullback.
\end{defi}

As in the non-parametrised case we have the fundamental orthogonal sequence ${\mathbb S}_K$ given by $({\mathbb S}_K)_n = K \times (\mathbb R^n)^+$ with the obvious structure maps given by the projection to the left factor and section which maps a point in $K$ to the point at infinity in the fibre, while the $O(n)$-action is given through the action on the right factor. It also comes with a canonical map ${\mathbb S}_K \wedge_K {\mathbb S}_K \rightarrow {\mathbb S}_K$ making it a commutative monoid with respect to the symmetric monoidal structure given by the pre-smash product.

\begin{defi}
A parametrised orthogonal spectrum over $K$ is an orthogonal sequence $X$ over $K$ together with a map $\sigma \colon X \wedge_K {\mathbb S}_K \rightarrow X$, that makes $X$ into a right ${\mathbb S}_K$-module. We denote the category of parametrised spectra over $K$ by $\mathcal S_K$.
\end{defi}

The pre-smash products lift to smash products $\wedge_K: \mathcal S_K \times \mathcal S_K \rightarrow \mathcal S_K$ and $\barwedge: \mathcal S_K \times \mathcal S_{K'} \rightarrow \mathcal S_{K \times K'}$ by taking the coequaliser of the two ${\mathbb S}_K$-multiplications on the pre-smash products of spectra over $K$. Furthermore, a parametrised spectrum can be smashed with an ex-space levelwise resulting in a new parametrised spectrum, a process we shall also denote by $\wedge_K$. The analoguous statement also hold for function spectra.

 May and Sigur\dh{}sson then proceed to put two model structures on the arising category of parametrised spectra, namely the \emph{level}- and the \emph{stable} structure (\cite[Theorem 12.1.7 \& Theorem 12.3.10]{MaSi}). Except for the following paragraph and the proof of proposition \ref{norepl} only the weak equivalences of the stable structure will matter. These are the maps $X \rightarrow Y$ that induce weak homotopy equivalences of homotopy fibre spectra $X^h_{|k} \rightarrow Y^h_{|k}$ for each point $k \in K$; May and Sigur\dh{}sson define these homotopy fibre spectra via the derived functor of the fibre functor $(-)_{|k} \colon \mathcal S_K \rightarrow \mathcal S$ for the level model structure in \cite[Definition 12.3.4]{MaSi}, but one can also give a direct definition: For a parametrised spectrum $X$ with projections $p_n \colon X_n \rightarrow K$ and sections $s_n \colon K \rightarrow X_n$ the various homotopy fibres of the $p_n$ with respect to $k$ fit into a spectrum: Picking the mapping path space $\{(x,w) \in X_n \times K^I \mid w(0) = k, w(1) = p_n(x)\}$ of $p_n$ as the model for the homotopy fibre for definiteness' sake, the adjoint of the suspension map 
$$\mathrm{hofib}(p_n) \rightarrow \Omega\mathrm{hofib}(p_{n+1})$$
sends $(x,w)$ to the loop starting with the path $t \mapsto (s_{n+1}(w(t)),w(t \cdot -))$, then running the loop $t \rightarrow (\sigma_n(x,t),w)$ and running the path backwards; here $\sigma_n \colon X_n \times I \rightarrow X_{n+1}$ denotes the structure map precomposed with the quotient map $X_n \times I \rightarrow X_n \wedge_K (K \times S^1)$.

For a spectrum whose projections $p_n$ are fibrations and whose sections $s_n$ are cofibrations the inclusion $X_{|k} \rightarrow X_{|k}^h$ is a weak equivalence. For the spectra we shall have to consider later this will always be the case, whence stable equivalences between them are detected on their actual fibres. \\

Before we present the definition of the twisted (co)homology theories obtained from a parametrised spectrum, one more comment is in order. When considering smash products, function spectra and similar constructions they all have to be interpreted in the derived sense. May and Sigur\dh{}sson indeed prove that all functors occuring in the following formulae are Quillen functors and hence can be interpreted in the derived sense with one exception: Derived function spectra are constructed directly as an adjoint to the derived functor of $\wedge_K$ via a Brown representability argument. 

They put
\begin{align*}
E_*(X;\zeta) &:= \pi_*\Theta\big((X \sqcup K)\wedge_K E\big) \\
E^*(X;\zeta) &:= \pi_{-*}\Gamma\big(F_K(X \sqcup K,E)\big)
\end{align*}
for $\zeta: X \to K$ obtaining a co- and a contravariant functor from the category $Top/K$ of spaces over $K$ to the category of graded abelian groups $Gr Ab$. Here $\Theta: \mathcal S_K \rightarrow \mathcal S$ denotes the degreewise collapse of the base-section, and $\Gamma: \mathcal S_K \rightarrow \mathcal S$ denotes degreewise sections of the structure maps. These functors are the left- and right-adjoint of the functor $\mathcal S \rightarrow \mathcal S_K$, which degreewise takes the cartesian product with $K$, in other words the pullback functor along $K \rightarrow \ast$.

For the case $K = *$ this definition recovers the classical one for the (co)homology theory associated to a spectrum and May and Sigur\dh{}sson show that the generalisation deserves the name \emph{twisted (co)homology} by establishing generalisation of the many well-known properties known for usual (co)homology, e.g. invariance under weak equivalences over $K$, Mayer-Vietoris sequences, limit sequences for unions, both a Serre-type and an Atiyah-Hirzebruch type spectral sequence, and more.

\begin{rem}\label{remark}
In \cite{MaSi} multiplicative structures are not much investigated. However, in direct generalisation of the unparametrised case, given parametrised spectra $E,F$ and $G$ over spaces $K, K'$ and $L$ and a map $f:K\times K' \rightarrow L$ any map of spectra $g:E \barwedge F \rightarrow f^*G$ gives rise to an exterior product of (co)homology theories. If the spectra all agree, $f$ is an associative multiplication on $K$ and $g$ is fibre-homotopy associative and maybe also fibre-homotopy-commutative, then we would obtain multiplications with the same properties. We do not know how to produce such multiplication maps on the parametrised spectra representing twisted $KO$-theory or $Spin$-cobordism. This seems to be coming from a systematic problem of the set-up given in \cite{MaSi}, see also \ref{honmul}, which is addressed in recent work of first author with Sagave and Schlichtkrull \cite{HSS19}: There, a generalisation of the categories of parametrised spectra of May and Sigur\dh{}sson is given, which allows for a clear treatment of multiplicative structures at the expense of allowing the base to be an $\mathcal I$-space. In a companion paper \cite{HS19} we use this generalisation to produce parametrised commutative ring spectra weakly equivalent to the ones considered here.

\end{rem}

\subsection{Preliminaries}\label{prelim}

In this subsection we collect notations for and properties of the objects that enter in the definitions of the models for twisted $Spin$-cobordism and $K$-theory spectra to be introduced in sections \ref{twispi} and \ref{twik}.

Let $L^2(\mathbb R^n)$ denote the Hilbert space of square integrable, real valued functions on $\mathbb R^n$, which we regard as a $\mathbb Z/2$-graded Hilbert space by its splitting into even and odd functions. Let $Cl_n$ denote the Clifford algebra of $\mathbb R^n$ with $v^2 = \langle v, v\rangle$, where $v$ is a vector in $\mathbb R^n$ and where the brackets on the right hand side denote the standard scalar product. The algebra $Cl_n$ can be regarded as a quotient of the tensor algebra on $\mathbb R^n$ and as such inherits a $\mathbb Z/2$ grading from the usual $\mathbb Z$-grading of the tensor algebra. There also is a canonical anti-involution $(-)^t\colon Cl_n \to Cl_n$ which leaves $\mathbb R^n\subset Cl_n$ invariant. Using the left regular representation $l \colon Cl_n\to End(Cl_n)$ and the normalised trace $\mathrm{Tr}_n: End(Cl_n) \to {\mathbb R}$ one obtains a scalar product on $Cl_n$ by declaring $\langle a,b \rangle =\mathrm{Tr}_n(l(a^tb))$. 
We further denote $L_n = L^2(\mathbb R^n) \otimes Cl_n$. The tensor products of ${\mathbb Z}/2$-graded Hilbert spaces will always receive its usual $\mathbb Z/2$-grading. With these conventions we have canonical isomorphisms $L_n \otimes L_m \rightarrow L_{n+m}$ of graded Hilbert spaces. The Koszul-signed symmetry isomorphism $\tau: X \otimes Y \rightarrow Y \otimes X$, sending homogeneous elements $x,y$ to $(-1)^{|x||y|}y \otimes x$ provides a symmetric monoidal structure on the category of $\mathbb Z/2$-graded Hilbert spaces.
\\

We shall need to consider two types of orthogonal groups acting on $L_n$. Let $\mathcal O(L_n)$ denote the group of continuous, linear, orthogonal automorphisms of $L_n$, which are either even or odd, and let $\mathcal O_n \subset \mathcal O(L_n)$ be the subgroup of right $Cl_n$-linear operators. We endow these groups with the compact-open topology (and not the norm topology) to guarantee that the $O(n)$-action on $L_n$ induces a continuous $O(n)$-action on $\mathcal O_n$ and $ \mathcal O(L_n)$ respectively. By Schur's lemma, the components of $\mathcal O_n$ (and clearly also $\mathcal O(L_n)$) are all homeomorphic (according to the representation type of $Cl_n$) to at most twofold products of ${\mathcal O}(H), {\mathcal U}(H)$ or ${\mathcal Sp}(H)$ for some separable Hilbert spaces $H$, which is of infinite dimension, since the rotation invariant functions form an infinite dimensional subspace of $L^2(\mathbb R^n)$ on which the Clifford algebra acts trivially. Hence by Kuiper's theorem \cite[Proposition A2.1]{AtSe} and its real and quaternionic counterparts all components of $\mathcal O_n$ and $ \mathcal O(L_n)$ are contractible. 

Both groups, $\mathcal O_n$ as well as $\mathcal O(L_n)$ are in fact supergroups in the sense of Stolz \cite{St3}. Recall that a supergroup consists of a triple consisting of a topological group $G$, a grading homomorphims $\pi: G\to \mathbb Z/2$, and a central element $c$, which is even in the sense, that $\pi(c) = 0$, and with $c^2 = e$. In the case at hand the grading homomorphism is the composition $|\cdot| \colon \mathcal O_n \subseteq \mathcal O(L_n) \rightarrow \mathbb Z/2$, and $-\id \in \mathcal O_n \subseteq \mathcal O(L_n)$ the distinguished central element.\\
The identifications $L_n \otimes L_m \cong L_{n+m}$ induce concatenation operations
$$\mathcal O(L_{n}) \widehat \times \mathcal O(L_{m}) \longrightarrow \mathcal O(L_{n+m})$$ 
which are maps of supergroups; here the left hand side denotes the product of supergroups. For two supergroups $G$ and $H$ this product is given by $$G \widehat \times H = (G \times H) /\mathbb Z/2$$ where $\mathbb Z/2$ acts on $G \times H$ diagonally via the specified order $2$ elements $c$ and $d$ of $G$ and $H$. The multiplication is then given by $$[g,h]\cdot [g',h'] = [c^{|h||g'|} gg', hh'] = [gg',d^{|h||g'|} hh'].$$
By design there is an isomorphism 
\[\mathcal O(H) \widehat \times \mathcal O(H') \rightarrow \mathcal O(H \otimes H')\]
given by the graded tensor product of operators. In our case the concatenation operation isomorphism given by sending $(U,U')$ to the operator
$$\xymatrix{
L^2(\mathbb R^{n+m}) \otimes Cl_{n+m} \ar[r]^-\cong & L^2(\mathbb R^n) \otimes L^2(\mathbb R^m) \otimes Cl_n \otimes Cl_m \ar[d]_{\id \otimes \tau \otimes \id}\\
& L^2(\mathbb R^n) \otimes Cl_n \otimes L^2(\mathbb R^m) \otimes Cl_m \ar[d]_{U \widehat{\otimes}  U'} & \\
& L^2(\mathbb R^n) \otimes Cl_n \otimes L^2(\mathbb R^m) \otimes Cl_m \ar[d]_{\id \otimes \tau \otimes \id} & \\
& L^2(\mathbb R^n) \otimes L^2(\mathbb R^m) \otimes Cl_n \otimes Cl_m \ar[r]^-{\cong} & L^2(\mathbb R^{n+m}) \otimes Cl_{n+m}
}$$
This composition introduces a plethora of signs. It is straight-forward to check that it restricts to a pairing $\mathcal O_n \widehat\times \mathcal O_m \rightarrow \mathcal O_{n+m}$.
\\
Similarly the groups $Pin(n) \subseteq Cl_n$ form supergroups using the grading homomorphism $Pin(n) \rightarrow \mathbb Z/2$ induced from the Clifford algebra, and $-1 \in Pin(n)$ as the distinguished central element. The direct sum operation induces an according pairing $\Pin(n) \widehat\times \Pin(m) \to \Pin(n+m)$.\\

Consider now the homomorphism $j: Pin(n) \rightarrow \mathcal O_n$, given by
$$p \longmapsto \big\{f \otimes c \longmapsto (-1)^{|p||f|} f \circ \rho(p)^{-1} \otimes p \cdot c\big\}$$
where $\rho: Pin(n) \rightarrow O(n)$ is the usual surjection. The homomorphism $j$ is in fact a morphism of supergroups and thus descends to a grading preserving homomorphism $j: O(n) \rightarrow P\mathcal O_n$. The grading on $O(n)$, which by construction has to be the one inherited from $Pin(n)$ via $\rho$, coincides with the determinant homomorphism, while $P\mathcal O_n$ stands for the projectivisation of $\mathcal O_n$, i.e.~the graded group $P\mathcal O_n$ which is the quotient of $\mathcal O_n$ by the subgroup $\{\pm \id\}$ generated by the distinguished central element.  Note that $-\id \in O(n)$ is \emph{not} mapped by $j$ to $-\id \equiv \id \in P\mathcal O_n$, whereas $Spin(n)$ and $SO(n)$ are mapped into the even elements in $\mathcal O_n$ and $P\mathcal O_n$ respectively.\\ 
The map $j$ is compatible with the concatenation operation on the $\mathcal O_n$ in the sense that the following diagram commutes
$$\xymatrix@-1pc{\Pin(n) \widehat\times \Pin(m)\ar[d]_{j \times j} \ar[rrr]^-{[x,y] \mapsto x \otimes y} &&& \Pin(n+m) \ar[d]^j\\
             \mathcal O_n \widehat\times \mathcal O_m \ar[rrr] &&& \mathcal O_{n+m}}$$ 
\\
There is a second important homomorphism we have to consider. Let $i: O(n) \rightarrow \mathcal O^{ev}(L_n)$ be given 
$$q \longmapsto \big\{ f \otimes c \longmapsto f \circ q^{-1} \otimes Cl_{q}(c)\big\}$$
where $Cl_{q}: Cl(n) \to Cl(n)$ is the homomorphism induced by $q\in O(n)$.
While not quite Clifford linear the operator $i(q)$ satisfies
$$\big(i(q)\big)\big((g \otimes c) \cdot d\big) = \Big(\big(i(q)\big)(g \otimes c)\Big) \cdot Cl_{q}(d)$$ so that conjugating a Clifford-linear operator with $i(q)$ again yields a Clifford-linear operator, an observation which will become important later.
Note that $-\id \in O(n)$ again does not map to $-\id \in \mathcal O(L_n)^{ev}$ under $i$. Obviously, the diagram
$$\xymatrix@-1pc{O(n) \times O(m)\ar[d]_{i \times i} \ar[rr] && O(n+m) \ar[d]^i\\
           \mathcal O(L_n)^{ev} \times \mathcal O(L_m)^{ev} \ar[d] \ar[rr] && \mathcal O(L_{n+m})^{ev} \ar[d] \\
						\mathcal O(L_n) \widehat\times \mathcal O(L_m) \ar[rr] && \mathcal O(L_{n+m})}$$
commutes as well. \\
We will frequently have to conjugate Clifford-linear operators $T$ on $L_n$ with elements either of the form $i(\rho(p))$ or $j(p)$ for $p \in \Pin_n$. Let us note straight away, that these conjugations are the same: the two elements 
$$i(\rho(p)) \circ T \circ i(\rho(p^{-1})) \text{    and    } (-1)^{|p||T|} j(p) \circ T \circ j(p^{-1})$$
agree. The verification of this statement is a straight forward, tedious calculation which requires a good book keeping of the signs involved. The main ingredient is the simple fact that $Cl_{\rho(p)}(c) = (-1)^{|c||p|}pcp^{-1}$ for $p,c \in Cl_n$. The fact that the two actions agree is implicit in \cite[Lemma 6.1]{Jo2}. However, the leading sign in the definition of the second conjugation action in \cite{Jo2} is given incorrectly.\\

As mentioned above, for constructing the desired parametrised spectra in the following chapters we need to modify the spectra constructed in \cite{Jo2} slightly. For this modification we need a second sequence of graded Hilbert spaces. Let $L_n'$ denote $\ell^2 \otimes L_n$, where $\ell^2$ is the space of square summable sequences, graded by the decomposition into subspaces of functions with support on the even and odd natural numbers respectively. In analogy to the above let $\mathcal O (\ell^2)$ and $\mathcal O(L'_n)$ denote the orthogonal operators on $\ell^2$ and $L'_n$ respectively, which are either even or odd, and let $\mathcal O'_n \subseteq \mathcal O(L'_n)$ be the subgroup of right Clifford linear operators. There is an obvious inclusion $\mathcal O(L_n) \rightarrow \mathcal O(L'_n)$ given by $U \mapsto \id_{\ell^2} \widehat\otimes U$ which in fact is a homotopy equivalence. There also is an obvious concatenation operation
$\mathcal O(L'_n) \widehat\times \mathcal  O(L_m) \rightarrow \mathcal O(L'_{n+m})$, which in particular makes the diagram 
$$\xymatrix@-1pc{\mathcal O(L_n) \widehat\times \mathcal O(L_m) \ar[d] \ar[rr] && \mathcal O(L_{n+m}) \ar[d]\\
            \mathcal O(L'_n) \widehat\times \mathcal O(L_m) \ar[rr] && \mathcal O(L'_{n+m})}$$ 
commute. The inclusion as well as the commutative diagram restrict accordingly to the subgroups of right Clifford linear operator. There are, however, no obvious identifications $L'_n \otimes L'_m \rightarrow L'_{n+m}$ and thus no concatenation operators $\mathcal O(L'_n) \widehat\times \mathcal O(L'_m) \rightarrow \mathcal O(L'_{n+m})$. This will result in the fact that the modified versions of the $Spin$ and $K$-theory spectra will not be ring spectra. The redeeming feature of $\mathcal O(L'_n)$ however is, that the inclusion $\mathcal O(\ell^2) \rightarrow \mathcal O'_n \subset \mathcal O(L'_n)$ which sends $U$ to $U \otimes \id_{L_n}$ creates a subgroup of $\mathcal O(L'_n)$, which commutes with the image of $\mathcal O_n$ and $\mathcal O(L_n)$, respectively. We will always view $\mathcal O(\ell^2)$ and $\mathcal O_n$ as subgroups of $\mathcal O(L'_n)$ in this way. The fact that the two subgroups commute will be crucial in the definition of the desired parametrised spectra below.

\subsection{A model for twisted $Spin$-cobordism.}\label{twispi}

We use a slight variation of the model for $MSpin$ appearing in \cite[Chapter 6]{Jo2} for defining its parametrised counterpart. As already pointed out in the introduction to this chapter we in fact will first introduce two models for the $Spin$-cobordism spectrum. One which we equip with a ring structure, and another one where we trade the strict multiplication for an action of the projective orthogonal group $P\mathcal O(\ell^2)$ by maps of orthogonal spectra. 

\begin{con}\label{twispin}
Recall from the previous section the definition of the homomorphism $j: Spin(n) \rightarrow \mathcal O^{ev}_n$. It is the inclusion of a subgroup. The topological group $\mathcal O^{ev}_n$ is a component of $\mathcal O_n$ and hence contractible, as we know already from the previous section. On the other hand by \cite[Theorem 4.1]{Gl} the projection $\mathcal O^{ev}_n \longrightarrow \mathcal O^{ev}_n / Spin(n)$ has a local section and thus provides $Spin(n)$-prinicipal bundle. Furthermore, the topology of $O^{ev}_n$ is metrizable (e.g.~see \cite[Proposition II.2.7]{Ta}) and hence by \cite[Corollary 1]{Michael} its base space is paracompact. Therefore we can and indeed will from now on use this projection as a model for the universal $Spin(n)$-principal bundle
\[ESpin(n) := \mathcal O^{ev}_n \longrightarrow \mathcal O^{ev}_n / Spin(n) =  P\mathcal O_n / O(n) =: BSpin(n)\]
We then define the $n$-th space of the spectrum $MSpin$ to be the Thom space of the associated vector bundle. The latter can explicitly be written down as the following balanced smash product
$$MSpin_n = (\mathcal O^{ev}_n)_+ \wedge_{Spin(n)}S^n = (P\mathcal O_n)_+ \wedge_{O(n)}S^n$$
which is exactly  as in \cite[Definition 6.3]{Jo2}. The $O(n)$ action on this space is given by mapping $q \in O(n)$ to the automorphism $q\cdot: MSpin_n \to MSpin_n$ given by
$$[U,s] \longmapsto [i(q) \circ U \circ i(q^{-1}), q(s)]$$
There are two things to be checked here. At first one needs to verify that conjugation with $i(q)$ preserves Clifford linearity. This is handled by the comment made after the introduction of $i$ in the previous section. Secondly one needs to check that the automorphism is well defined in the sense that the formula is compatible with taking equivalences classes. To see this it is helpful to recall that the conjugation by $i(q)$ also can be expressed by the conjugation by $j(p)$ for a choice of $p$ with $\rho(p)=q$. The description of the automorphism above then simplifies to 
$$[U,s] \longmapsto [j(p) \circ U \circ j(p^{-1}), \rho(p)(s)] = [(j(p) \circ U, s]$$
One may wonder why we defined the action using $i$ instead of $j$ in the first place. We wanted to stay consistent with the definition introduced in \cite{Jo2}, where a description using the homomorphism $i$ is mandatory due to the functoriality built in in the notion of an orthogonal equivariant spectrum in the sense of \cite{MaMa}, which was the foundational set-up used in \cite{Jo2}.\\ 
The orthogonal sequence $MSpin$ can be given the structure of a commutative monoid using the multiplication maps
$$MSpin_n \wedge MSpin_m \rightarrow MSpin_{n+m}$$ 
that are induced from the concatenation $P\mathcal O_n \times P\mathcal O_m \rightarrow P\mathcal O_{n+m}$ and the usual identification $S^n \wedge S^m \cong S^{n+m}$. That this multiplication maps are well-defined immediately follows from the diagram showing that $j$ is compatible with the concatenation operation on the $\mathcal O_n$. The fact that the multiplication maps are $O(n) \times O(m)$-equivariant follows from the analogous diagram involving $i$. The various multiplication maps together give the orthogonal sequence $MSpin$ the structure of a commutative monoid in the category of orthogonal sequences. 
Furthermore, the maps
$$S^n \rightarrow MSpin_n, \quad s \longmapsto [id_{L_n}, s]$$
which are obviously $O(n)$-equivariant, induce a map of monoids $\mathbb S \to MSpin$ in the category of orthogonal sequences which then gives $MSpin$ the structure of an $\mathbb S$-module. Altogether these definitions turn $MSpin$ into a commutative, orthogonal ring spectrum.\\

We proceed to introduce the second version of the $Spin$ cobordism spectrum, denoted $MSpin'$, by mimicking the above construction using $\mathcal O'_n$ instead of $\mathcal O_n$, i.e.~we put
$$MSpin'_n = (P\mathcal O'_n)_+ \wedge_{O(n)}S^n$$
Through the inclusions $\mathcal O_n \rightarrow \mathcal O'_n$ we obtain a map $MSpin \rightarrow MSpin'$, which is a weak equivalence in every degree. Furthermore the maps $P\mathcal O'_n \times P\mathcal O_m \rightarrow P\mathcal O'_{n+m}$, provide multiplication maps $$MSpin'_n \wedge MSpin_m \rightarrow MSpin'_{n+m}$$
making $MSpin'$ a right $MSpin$-module spectrum. 
The inclusion $P\mathcal O(\ell^2) \rightarrow P\mathcal O'_n$ provides an action of $P\mathcal O(\ell^2)$ on each $MSpin_n$ by left multiplication on the $P\mathcal O'_n$ factor, i.e. with $U\in \mathcal O(\ell^2)$ and $U' \in \mathcal O'_n$ the action above is explicitly given by
\begin{align*}
P\mathcal O(\ell^2) \times (P\mathcal O'_n) \times_{O(n)} S^n &\longrightarrow (P\mathcal O'_n) \times_{O(n)} S^n \\
([U], [U'], v) &\longmapsto \big([U \otimes \id_{L_n}) \circ U'], v \big)
\end{align*}
This action commutes with all structure maps, since it commutes with the multiplication by elements in $P\mathcal O_n$ through which all the additional structure is defined. As a result $MSpin'$ is a $P\mathcal O(\ell^2)$-object in $MSpin$-module spectra.
\end{con}

We now want to employ the construction of \cite[Section 22.1]{MaSi} to produce a parametrised spectrum over the classifying space for $P\mathcal O(\ell^2)$. Even though we will not endow the arising parametrised spectrum $\MSpinK$ with multiplicative structures we need to be slightly careful with our choice for the base space. The reason is that we \emph{do} want to use multiplicative properties of the trivially parametrised spectrum $BP\mathcal O(\ell^2) \times H\mathbb Z/2$ and the set-up of May and Sigur\dh{}sson seems to only accomodate this when the base space is a topological monoid (and not just a space equipped with a coherent multiplication). We therefore analyse the homotopy type of $BP\mathcal O(\ell^2)$ before proceeding:

\begin{lem}\label{twimul}Let $K$ denote the product of two Eilenberg-MacLane spaces, one of type $(\mathbb Z/2,1)$, the other of type $(\mathbb Z/2,2)$ and $BO[2]$ the second Postnikov truncation of $BO$. Then:
\begin{enumerate}
\item[i)] There is a preferred class of equivalences between $BP\mathcal O(\ell^2)$ and $BO[2]$.
\item[ii)] Under that equivalence the multiplication on $BP\mathcal O(\ell^2)$ induced by any choice of isomorphism $i: \ell^2 \rightarrow \ell^2 \otimes \ell^2$ corresponds to the Whitney sum on $BO[2]$.
\item[iii)] There are exactly two homotopy classes of equivalences between $BO[2]$ and $K$ and under both of them the Whitney sum translates into the map $K \times K \rightarrow K$ represented by
\end{enumerate}
$$(\iota_1 \times 1 + 1 \times \iota_1, 1 \times \iota_2 + \iota_2 \times 1 + \iota_1 \times \iota_1) \in H^1(K \times K,\mathbb Z/2) \times H^2(K \times K,\mathbb Z/2)$$
\end{lem}

The analogue of lemma \ref{twimul} for the complex projective unitary group of a separable complex Hilbert space is proved for example in \cite[Proposition 2.3]{AtSe} using slightly different language. We present a proof for the projective orthogonal group $P\mathcal O(\ell^2)$ which fits well with the given context.

\begin{proof} 
The  morphism $j: Pin(n) \to \mathcal O_n$ of supergroups induces a group homomorphism $O(n) \to P \mathcal O_n \hookrightarrow P\mathcal O(L_n)$, that we still will denote by $j$. Now fix some isomorphism $\mu : \ell^2 \otimes \ell^2 \rightarrow \ell^2$. It induces isomorphisms $L_n' \otimes L_m' \rightarrow L_{n+m}'$ by first reshuffling factors, which we also denote by $\mu$. Then consider the following commutative diagram of group homomorphisms
$$\xymatrix{P\mathcal O(\ell^2) \times P\mathcal O(\ell^2) \ar[r]^-\otimes \ar[d]_-\simeq^{(\cdot \ \otimes \id_{L_n}) \times (\ \cdot \ \otimes \id_{L_m})} & P\mathcal O(\ell^2 \otimes \ell^2) \ar[r]^-{\mu \circ \ \cdot \ \circ \mu^{-1}}\ar[d]_-\simeq^{\cdot \ \otimes \id_{L_n} \otimes \ \cdot \ \otimes \id_{L_m}} & P \mathcal O(\ell^2) \ar[d]_-\simeq^{\cdot \ \otimes \id_{L_{n+m}}}\\
            P\mathcal O(L'_n) \times P\mathcal O(L'_m) \ar[r]^-\otimes            & P\mathcal O(L_n' \otimes L_m') \ar[r]^-{\mu \circ \ \cdot \ \circ \mu^{-1}}    & P\mathcal O(L_{n+m}') \\
						P\mathcal O(L_n) \times P\mathcal O(L_m) \ar[r]^-\otimes \ar[u]^-\simeq_{(\id_{\ell^2} \otimes \ \cdot) \times (\id_{\ell^2} \otimes \ \cdot)}     & P\mathcal O(L_n \otimes L_m) \ar[r]^-\cong  \ar[u]^\simeq_{\id_{\ell^2} \otimes \ \cdot \ \otimes \id_{\ell^2} \otimes \ \cdot}      & P\mathcal O(L_{n+m}) \ar[u]^-\simeq_{\id_{\ell^2} \otimes \ \cdot} \\		
						O(n) \times O(m)\ar[u]_{j\times j} \ar[rr]^-\oplus && O(n+m) \ar[u]_j}$$
For $n \geq 3$ the map $Bj: BO(n) \rightarrow BP\mathcal O_n$ induces isomorphisms on the first and second homotopy groups. In particular we have equivalences 
$$BO[2] \leftarrow BO(n)[2] \rightarrow BP\mathcal O_n[2] \leftarrow BP\mathcal O_n \rightarrow BP\mathcal O'_n \leftarrow BP\mathcal O(\ell^2)$$
that identify the $H$-space structure on $P\mathcal O(\ell^2)$, with the one on $BO[2]$ induced from direct sum of bundles. The claims of part iii) are well known for this space: The first and second Stiefel-Whitney classes give an equivalence $BO[2] \rightarrow K$ and the formula for the multiplication on $K$ is exactly the Whitney sum formula. The other equivalence is given by $(w_1, w_2+w_1^2)$ and it is readily checked to be a morphism of $H$-spaces, again by the Whitney sum formula. It is a similarly simple calculation that these are the only equivalences.
\end{proof}

Given this result it is possible to choose a classifying space $K$ for $P\mathcal O(\ell^2)$ that is a topological monoid as required: One may for example use the bar construction to explicitely obtain a product $K$ of two Eilenberg-Mac Lane spaces and equip it with a multiplication given by explicit models of addition and cup-products on bar constructions (compare \cite{Milgram}). Another possibility is to note that the $H$-space structure on $K$ refines to an $E_\infty$-structure, since it can be pulled back from the $H$-space structure on $BO[2]$ induced by the Whitney-sum structure (because of  iii)). One can therefore take $K$ to be the Moore loop space of $B(BO[2])$ to obtain another model for $BP\mathcal O(\ell^2)$ which is a monoid.

For our construction of a twisted $Spin$-cobordism spectrum, we shall now fix a choice of topological monoid $K$ as above. For any such, there two isomorphism classes of universal $P\mathcal O(\ell^2)$-bundle $EP\mathcal O(\ell^2) \rightarrow K$, and we need to fix the choice in order to fit with the conventions of Section 2. We do this as follows:
The spaces $EP\mathcal O'_n := EP\mathcal O(\ell^2) \times_{P\mathcal O(\ell^2)} P\mathcal O'_n$ are weakly contractible, so that they can serve as universal $P\mathcal O'_n$- and in particular universal $O(n)$-spaces using the homomorphism $j: O(n) \rightarrow P\mathcal O'_n$. This means that 
\[BO(n) := EP\mathcal O'_n / j(O(n))\]
is a choice of classifying space for $O(n)$, which we shall fix as the model from now on. It comes equipped with a map $w \colon BO(n) \rightarrow K$ (in fact a fibre bundle) by taking the quotient with respect to the action by the larger group $P\mathcal O'_n$. Depending on the choice of universal bundle there are two possibilities for the homotopy type of this map.

\begin{lem}\label{fibtho}
The map $w \colon BO(n) \rightarrow K$ either represents $(w_1,w_2)$ or $(w_1, w_2 + w_1^2)$ and post-composing with the non-trivial self-homotopy equivalence of $K$ (which is the same as choosing a universal bundle in the other equivalence class) interchanges these.
\end{lem}

\begin{proof}
Since $P\mathcal O_n/O(n)$ is a model for $BSpin(n)$
it follows from the long exact sequence of the fibre bundle $$P\mathcal O_n/O(n) \rightarrow BO(n) \rightarrow K$$ that the induced map $\pi_*(BO(n)) \rightarrow \pi_*(K)$ is an isomorphism in degrees one and two. Now the Hurewicz theorem applied to this (and the corresponding map of universal covers) gives the first claim. The second follows immediately from the explicit description of the non-trivial self-homotopy equivalence given in \ref{twimul}.
\end{proof}

Since such choices are unavoidable in the current context of parametrised spectra and twisted cohomology, we shall keep the monoid $K$, and a universal $P\mathcal O(\ell^2)$-bundle over it, such that the map $w$ above represents $(w_1, w_2)$ fixed throughout the remainder of this paper.

\begin{defi}
The parametrised spectrum $\MSpinK$ given by
$$(\MSpinK)_n = EP\mathcal O(\ell^2) \times_{P\mathcal O(\ell^2)} MSpin'_n$$
with structure maps and $O(n)$-actions induced from $MSpin'$ will be called the twisted $Spin$-cobordism spectrum.
\end{defi}

We define \emph{twisted $Spin$-cobordism} $\MSpin(-;-)$ to be the homology theory represented by this spectrum as explained in section \ref{paraspec}. Using proposition \ref{norepl} the Pontryagin-Thom construction can be used to give a geometric interpretation of these groups using manifolds and bordisms as cycles and relations. In section \ref{fund} we will provide such a description for the cycles.

Note furthermore, that the actions of $MSpin$ and $P\mathcal O(\ell^2)$ on $MSpin'$ explained in \ref{twispin} commute. The concatenation maps $P\mathcal O'_n \times P\mathcal O_m \rightarrow P\mathcal O_{n+m}'$ therefore provide a multiplication
$$\MSpinK \barwedge MSpin \longrightarrow \MSpinK$$
making twisted $Spin$-cobordism $\MSpin_*(-;-)$ into a module theory over (non-twisted) $Spin$-cobordism $\MSpin_*(-)$. 

\begin{obs}\label{MO}
The spaces $(\MSpinK)_n$ are canonically homeomorphic to the fibrewise Thom spaces along the map $BO(n) \rightarrow K$ of the universal bundles described above. In particular, we have that $\Theta(\MSpinK)$ is a model for $MO$.
\end{obs}

\begin{rem}\label{honmul}
Choosing an isomorphism $\mu:\ell^2 \rightarrow \ell^2 \otimes \ell^2$ one obtains a map $P{\mathcal O}(\ell^2) \times P{\mathcal O}(\ell^2) \rightarrow P{\mathcal O}(\ell^2)$ and we can ask for an equivariant bundle map  
$$\begin{xy}\xymatrix@-1pc{EP{\mathcal O}(\ell^2) \times EP{\mathcal O}(\ell^2) \ar[r]\ar[d] & EP{\mathcal O}(\ell^2) \ar[d] \\
                      K  \times K \ar[r]                                                              & K
}\end{xy}$$ 
covering the multiplication of $K$. Since $EP{\mathcal O}(\ell^2) \rightarrow K$ is universal, the space of these bundle maps were contractible if we left the multiplication map of $K$ variable (within its homotopy class) as well. The set-up of May and Sigur\dh{}sson \cite{MaSi}, however, does not allow the multiplication on $K$ to vary, see also Remark \ref{remark}. If we fix the multiplication, there is more than a single fibre homotopy class of such maps. While every one of them induces a multiplication $\MSpinK \barwedge \MSpinK \rightarrow \MSpinK$, none of them will be fibre homotopy associative. The bundle maps still produce a single homotopy class of maps $$\Theta(\MSpinK) \wedge \Theta(\MSpinK) \rightarrow \Theta(\MSpinK)$$ making $\Theta(\MSpinK)$ an orthogonal spectrum with only a homotopy ring structure, despite the fact that $\Theta(\MSpinK) \simeq MO$, which can easily be made into an honest orthogonal ring spectrum.
\end{rem}

Finally, in order to make the connection to the first part of this paper we need the following technical proposition.

\begin{prop}\label{norepl}
Let $\zeta: X \rightarrow K$ be a space over $K$, where $X$ is a cell complex. Then there is a canonical isomorphism
$$\MSpin_n(X;\zeta) \cong \pi_n(\Theta(\MSpinK \wedge_K (X \sqcup K)))$$
\end{prop}

Recall that the left hand side is defined homotopically by deriving all functors that occur on the right. The point of the proposition is that no cofibrant or fibrant resolution has to take place in order to correctly calculate the twisted cobordism groups. 

\begin{proof}[Proof of proposition \ref{norepl}]
The statement of the proposition is entirely internal to the model categorical aspects of \cite{MaSi}. We will therefore freely use their language in this proof. We have to construct a zig-zag of stable equivalences between
$$\Theta(\MSpinK \wedge_K (X \sqcup K))$$
and
$$\Theta c\Delta^* f_s(c\MSpinK \barwedge c(X \sqcup K))$$
where $c$ and $f_s$ denote cofibrant and fibrant resolutions in the stable model structure. Note that cofibrations in the level and stable model structure agree, in particular, cofibrant resolution can be achieved by a level equivalence.

The external smash product preserves homotopy equivalences 
between well-sectioned spaces (by the same proof as \cite[Proposition 8.2.6]{MaSi}) and all spaces in sight have the homotopy type of cell complexes (see the exposition in \cite[Section 9.1]{MaSi} for the resolved objects). Therefore we have a levelwise homotopy equivalence
$$c \MSpinK \barwedge c (X \sqcup K) \rightarrow \MSpinK \barwedge (X \sqcup K).$$
Now consider the composition
$$f_l(c \MSpinK \barwedge c (X \sqcup K)) \rightarrow f_l(\MSpinK \barwedge (X \sqcup K)) \rightarrow \MSpinK \barwedge (LX \sqcup K),$$
where $L$ is the Moore-mapping-path-space functor and $f_l$ is a level fibrant replacement. The map on the right exists since $LX \sqcup K$ and $(\MSpinK)_n$ are ex-fibrations, as is their smash product by 
\cite[Proposition 8.2.3]{MaSi}, and spectra consisting of ex-fibrations are level-fibrant. Now $\Delta^*$ preserves stable equivalences between level fibrant objects (since stable equivalences can then be tested on actual fibres) and as a right Quillen functor also level equivalences between level fibrant objects. Therefore
$$\Delta^*f_s(c \MSpinK \barwedge c (X \sqcup K)) \leftarrow \Delta^*f_l(c \MSpinK \barwedge c (X \sqcup K)) \rightarrow \MSpinK \wedge_K (LX \sqcup K)$$
consists of a stable and a level equivalence, respectively. Now all three terms are well-sectioned: For the right this follows by inspection and for the left two it follows since cofibrant spaces are well-sectioned by 
\cite[Theorem 6.2.6]{MaSi}, and $\Delta^*$ preserves the property of being well-sectioned by \cite[Proposition 8.2.2]{MaSi}. Since $\Theta$ preserves level equivalences between well-sectioned ex-spaces and stable equivalences between cofibrant objects (since it is a left Quillen functor) we obtain stable equivalences
\begin{multline*}
\Theta c\Delta^* f_s(c\MSpinK \barwedge c(X \sqcup K)\leftarrow \Theta c\Delta^* f_l(c\MSpinK \barwedge c(X \sqcup K) \\
\rightarrow \Theta(\MSpinK \wedge_K (LX \sqcup K))
\end{multline*}
The right hand side here is the Thom spectrum associated to the pullback of
$$\begin{xy}\xymatrix@-1pc{   & BO(i) \ar[d] \\
                       LX \ar[r]       & K}\end{xy}$$
Since the canonical map $X \rightarrow LX$ is a homotopy equivalence and the map $BO(i) \rightarrow K$ a fibration, this pullback is homotopy equivalent to that of
$$\begin{xy}\xymatrix@-1pc{  & BO(i) \ar[d] \\
                       X \ar[r]       & K}\end{xy}$$
which concludes the proof.
\end{proof}

Recall from Section \ref{choice-of-bordism-gropus} how an almost spinnable manifold $M$ with fundamental group $\pi$ admits a map $u \colon  B\pi \rightarrow K$, that lifts $(w_1(M),w_2(M)+w_1(M)^2)$. Taken together, Proposition \ref{twicob} and Proposition \ref{norepl} yield

\begin{cor}\label{mind}
For $B(\pi,u)$ the normal $1$-type of an almost spinnable manifold, we have
$$\Omega^{(\pi,u)}_* \cong \MSpin_*(B\pi;u).$$
\end{cor}

\begin{rem}
A similar (but far easier) construction can be used to show that for the normal $1$-type of a totally non-spinnable manifold we have an isomorphism
$$\Omega^{(\pi,u)}_* \cong \MSO_*(B\pi;u),$$
where the right hand side stands for a twisted oriented cobordism groups. The results we are ultimately trying to show, however, are well-known for totally non-spinnable manifolds.
\end{rem}

\subsection{A model for twisted $K$-theory.}\label{twik}

We again use a slight variation of the spectra appearing in \cite[Chapter 6]{Jo2}, introducing a version which trades the nice multiplication for an action by $P{\mathcal O}(\ell^2)$. \\

\begin{con}\label{constructing-KO}
Consider the spaces $KO_n = Hom(C_0(\mathbb R), K(L_n))$, where $C_0(\mathbb R)$ denotes the algebra of real valued functions on the real line vanishing at infinity, which we regard as a ${\mathbb Z}/2$-graded $C^*$-algebra using its decomposition into even and odd functions, and where $K$ denotes the ${\mathbb Z}/2$-graded $C^*$-algebra of right-Clifford-linear compact operators, while $Hom$ denotes the space of degree preserving $C^*$-homomorphisms, pointed by the null map. It is shown in \cite{Jo2}, that these spaces represent $KO$-theory. We shall turn this sequence into an orthogonal spectrum just as before:
The orthogonal group acts on $KO_n$ by conjugation using the same map $i\colon O(n) \rightarrow {\mathcal O}^{ev}(L_n)$, i.e.~the action is explicitly given by
$$O(n) \times KO_n \longrightarrow KO_n \quad (q, \varphi) \longmapsto \big\{f \longmapsto i(q) \circ \varphi(f) \circ i(q)^{-1}\big\}$$
This again preserves the Clifford-linearity of the operator and the discussion after the definition of the $O(n)$-action on $MSpin_n$ applies as well. Indeed, it will be important later that this action factors via $j$ over the action of $P\mathcal O_n$ on $KO_n$ given by the following formula
$$([U],\varphi) \mapsto \big\{f \longmapsto (-1)^{|U||f|} U \circ \varphi(f) \circ U^{-1}\big\}$$

The structure maps will arise via the unit of a commutative multiplication on $KO$. These multiplications arise from the coproduct $\Delta$ on $C_0(\mathbb R)$ and the isomorphisms $K(L_n) \otimes K(L_m) \rightarrow K(L_{n+m})$ and $K(L_n) \otimes K(L'_m) \rightarrow K(L'_{n+m})$ induced by the isomorphisms $L_n \otimes L_m \rightarrow L_{n+m}$ (compare \cite[p.~94]{Jo2}); for the definition of the coproduct on $C_0(\mathbb R)$ e.g.~see \cite{Ha2}. Explicitly the multiplication maps for $KO$ are given by the maps
$$Hom(C_0(\mathbb R), K(L_n)) \wedge Hom(C_0(\mathbb R), K(L_m)) \longrightarrow Hom(C_0(\mathbb R), K(L_{n+m}))$$
which map $(\varphi, \psi)$ to 
$$\xymatrix{C_0(\mathbb R) \ar[r]^-\Delta & C_0(\mathbb R) \otimes C_0(\mathbb R) \ar[r]^-{\varphi \otimes \psi} & K(L_n) \otimes K(L_m) \ar[r]^-{\widehat \otimes}_-\cong & K(L_{n+m})}$$
where the final map induces a plethora of signs, just like the concatenation operation on the $\mathcal O_n$'s. Unwinding these definitions immediately gives that the product maps are $O(n) \times O(m)$-equivariant. The unit $u: S \to KO$ for the product on $KO$ is given by extending 
$$u_n \colon \mathbb R^n \rightarrow KO_n, \quad w \mapsto \big\{f \mapsto p_n \widehat\otimes f(w) \cdot - \big\}$$
by zero, where $p_n$ denotes the projection operator of $L^2(\mathbb R^n)$ onto the (even!) subspace generated by the function $v \mapsto \exp(-|v|^2)$, and $f(w) \cdot -$ denotes left multiplication by $f(w)$, while $f(w)$ is given by functional calculus. Note that elements $u\in {\mathbb R}^n$ of norm $1$ are odd, unitary and satisfy $u^2=1$. Hence they are all odd and selfadjoint, and therefore the same applies to all vectors $w\in {\mathbb R}^n$. In particular one can apply functional calculus to all vectors $w\in {\mathbb R}^n$. The unit is $O(n)$-equivariant because $Cl_q(f(w)) = f(q(w))$ for all $f \in C_0(\mathbb R), w \in \mathbb R^n, q \in O(n)$ and the subspace generated by the function $v \mapsto \exp(-|v|^2)$ is $O(n)$-invariant. Altogether one obtains a commutative orthogonal ring spectrum $KO$. \\
Just as before we shall need a modified version of $KO$ to produce twisted $KO$-theory. To this end let $KO'_n = Hom(C_0(\mathbb R), K(L'_n))$ and let us mimick the constructions above to obtain an action $KO'_n \wedge KO_m \rightarrow KO'_{n+m}$, which turns $KO'$ into an $KO$-module spectrum. We can also construct a map $KO \rightarrow KO'$, but in contrast to the situation for the $Spin$-spectra the map depends upon an arbitrary choice of an (even) rank $1$ projection operator $p$ on $\ell^2$\label{fixing-the-projector}, that we fix once and for all. It is explicitly given by sending a homomorphism $\varphi \in Hom(C_0(\mathbb R), K(L_n))$ to the homomorphism
$$\varphi':C_0(\mathbb R) \rightarrow K(L'_n) \text{, given by } f \longmapsto p \widehat\otimes \varphi(f)$$
Note that the space of rank $1$ operators is path-connected so the homotopy class of this equivalence is uniquely determined; it is a weak equivalence, since the map on compact operators $K(H) \rightarrow K(H' \otimes H)$ given by tensoring an operator with a rank $1$ projection on $H'$ is a homotopy equivalence of $C^*$-algebras, being homotopic to the isomorphism arising from any identification $H \cong H' \otimes H$, see e.g. \cite{Me}.
The spectrum $KO'$ carries an action by $P\mathcal O(\ell^2)$. It is induced by conjugating a homomorphism $\varphi: C_0(\mathbb R) \rightarrow KO'_n$ with elements $U \in \mathcal O(\ell^2)$, i.e.~ the action is given explicitly by
$$f \longmapsto (-1)^{|U||f|} (U \otimes \id_{L_n}) \circ \varphi(f) \circ (U^{-1} \otimes \id_{L_n})$$
for which in particular the $KO$-action maps are equivariant and thus turn $KO'$ into a $P\mathcal O(\ell^2)$-object in $KO$-module spectra.
\end{con}

From this data we produce a parametrised spectrum just as above.

\begin{defi}
The spectrum $\KOK$ given by
$$(\KOK)_n = EP{\mathcal O}(\ell^2) \times_{P{\mathcal O}(\ell^2)} KO'_n$$
with structure maps and $O(n)$-actions induced form $KO'$ we call the twisted, real $K$-theory spectrum.
\end{defi}

As in the case of the twisted $Spin$-spectrum $\MSpinK$, this spectrum $\KOK$ is a module spectrum over $KO$ and the spectrum $\Theta(\KOK)$ inherits a homotopy multiplication. 

\begin{rem}\label{agg}
In the appendix of \cite{AGG} the authors claim to construct an action of $P{\mathcal O}(\ell^2)$ on the model of the $KO$-spectrum from \cite{Jo} through maps of ring spectra avoiding the passage to a `free rank one module' that we went through above. Closer inspection, however, shows that their action does not leave the unit map invariant and thus does not commute with the structure maps of the spectrum invalidating \cite[Proposition 4.1]{AGG}. In the third appendix, on the way to our direct comparison between the homotopical definition of twisted $KO$-theory along the lines of \cite{AnBlGeHoRe} and our geometric one above, we will explain how our models fit in with their framework.
\end{rem}

\section{Fundamental classes and orientations}\label{fco}

Having constructed the spectra relevant to our work we now turn to the fundamental map relating them, the twisted Atiyah-Bott-Shapiro orientation. We will explain how this transformation of twisted homology theories gives rise to Thom classes and fundamental classes for arbitrary bundles and manifolds, respectively, in twisted $KO$-theory and its connective covers. We shall also discuss the additional choices one has to make to specify these classes. Most of the material of this section can again be safely skipped by any reader familiar with twisted generalised homology. Exceptions are the definition of the twisted Atiyah-Bott-Shapiro orientation in \ref{twisted-ABS-orientation:construction} and the final statement \ref{ourconj}, which is Stolz' conjecture about the existence of metrics of positive scalar curvature that we ended the introduction with and which we now can finally state formally.

\subsection{The twisted Atiyah-Bott-Shapiro orientation.} \label{twisted-ABS-orientation:construction}
Recall that $Spin$-bundles have Thom classes in real $K$-theory by the work of Atiyah, Bott and Shapiro. This gives rise to a (multiplicative) transformation of homology theories $\Omega_*^{Spin} \rightarrow KO_*$, represented by some map of spectra $MSpin \rightarrow KO$. In \cite{Jo2} the second author constructed such a (ring-)map explicitly using the spectra described above (implying in particular that this orientation is an $E_\infty$-map). We will modify his construction slightly to obtain a map $\MSpinK \rightarrow \KOK$, which will provide us with Thom classes in twisted $K$-theory. 

\begin{con}
The construction begins with the following observation from the construction of $KO$: The $O(n)$ action on $KO_n$ factors through $j: O(n) \rightarrow P\mathcal O_n$, where the $P{\mathcal O}_n$ action on $KO_n$ is given by
$$P{\mathcal O}_n \times KO_n \rightarrow KO_n, \quad (U,\varphi) \mapsto \big\{f \longmapsto (-1)^{|U||f|} U \circ \varphi(f) \circ U^{-1}\big\}$$
We then can use the $O(n)$-equivariance of the maps $u_n: S^n \to KO_n$, which provide the unit of $KO$,  to set $\alpha_n$ to be the composite
\begin{align*} 
&MSpin_n = P\mathcal O_{n+} \wedge_{O(n)} S^n \\
             &\quad \quad \quad\stackrel{id \wedge u_n}{\longrightarrow}P\mathcal O_{n+} \wedge_{O(n)} KO_n\\
             &\quad \quad \quad \quad\quad\quad \longrightarrow P\mathcal O_{n+} \wedge_{P \mathcal O_n} KO_n \stackrel{\cong}{\longrightarrow} KO_n
\end{align*}
which is $O(n)$ invariant by the second description of the action on $MSpin_n$.
We similarly have a map $\alpha': MSpin' \rightarrow KO'$, depending on a rank $1$ projector $p$ in $\ell^2$ (see \ref{fixing-the-projector} for the role of $p$ in the definition of $KO'$).
\begin{align*}
&MSpin'_n = P\mathcal O'_{n+} \wedge_{O(n)} S^n \\
             &\quad \quad \quad\stackrel{id \wedge u_n}{\longrightarrow} P\mathcal O'_{n+} \wedge_{O(n)} KO_n\\
				     &\quad \quad \quad\quad \quad \quad\longrightarrow P\mathcal O'_{n+} \wedge_{O(n)} KO'_n\\
             &\quad \quad \quad\quad \quad \quad \quad\quad\quad \longrightarrow P\mathcal O'_{n+} \wedge_{P \mathcal O'_n} KO'_n  \stackrel{\cong}{\longrightarrow} KO'_n
\end{align*}
 These maps make the following diagram commute
$$\begin{xy}\xymatrix@-1pc{ MSpin \ar[r]^\alpha \ar[d] & KO \ar[d] \\
                       MSpin' \ar[r]^{\alpha'}    & KO'} \end{xy}$$
where both maps to $KO'$ come from the same choice of $p$. 
\end{con}

\begin{thm}[Theorem 6.9 \cite{Jo2}]\label{mictho}
The map $\alpha: MSpin \rightarrow KO$ induces the $KO$-theory Thom classes for $Spin$-bundles of Atiyah, Bott and Shapiro. The same is therefore true for the map $\alpha': MSpin' \rightarrow KO'$.
\end{thm}

The crucial observation is that $\alpha'$ is equivariant under the action of $P{\mathcal O}(\ell^2)$, hence induces a map of fibre bundles:

\begin{defi}
We call the map $\alpha_K: \MSpinK \rightarrow \KOK$ given by 
$$\mathrm{id} \times \alpha': EP{\mathcal O}(\ell^2) \times_{P{\mathcal O}(\ell^2)} MSpin(n)' \to EP{\mathcal O}(\ell^2) \times_{P{\mathcal O}(\ell^2)} KO'_n$$
 the \emph{twisted Atiyah-Bott-Shapiro-orientation}.
\end{defi}

\begin{rem}
We remark that such a map can easily be constructed along the lines of \cite{AnBlGeHoRe}. By considering the fibre sequence $K(\mathbb Z/2,0) \times K(\mathbb Z/2,1) \rightarrow BSpin \rightarrow BO$ and studying the induced map of Thom spectra as we do in appendix C one obtains a map $K(\mathbb Z/2,1) \times K(\mathbb Z/2,2) \rightarrow BGl_1(MSpin)$. The $E_\infty$ ring map $MSpin \rightarrow KO$ induces a map $BGl_1(MSpin) \rightarrow BGl_1(KO)$ from which one obtains a map $MSpin_K \longrightarrow KO_K$, if the homotopical definition (of loc. cit.) for the parametrised spectra are used. However, from this description it is not clear how to get at the geometric content of the induced transformation, which is very relevant to our work. We show both constructions to agree in appendix C.
\end{rem}

Let us now explain how the twisted Atiyah-Bott-Shapiro orientation gives rise to Thom classes. Recall that in \ref{MO} we have identified $\MSpinK(n)$ with the fibrewise Thom space of the universal vector bundle $\gamma_n$ over $BO(n)$ along a fibration $BO(n) \rightarrow K$ representing the first and second Stiefel-Whitney class. Given now an $n$-dimensional bundle $V \rightarrow B$ together with a commutative diagram
$$\xymatrix@-1pc{ V \ar[d]_p \ar[rr]  &   & \gamma_n \ar[d]  \\
             B \ar[rd]_{w(V)} \ar[rr] &   & BO(n) \ar[ld]^w \\
                                 & K &               }$$
we obtain a canonical fibre homotopy class $(DV,SV) \rightarrow (\MSpinK(n),K)$ over $K$. 
Using the canonical map $$[X,E_n]_K \to [\mathbb S_K \wedge_K (X \sqcup K),\mathbb S_K^n \wedge_K E]_K = \pi_{-n}\Gamma \big( F_K(X \sqcup K, E)\big) = E^n(X,\zeta)$$
for a space $\zeta: X \to K$ and a parametrised spectrum $E$ and taking the fibrewise quotient by the sphere bundle, we obtain classes $$t \in \MSpin^n(DV, SV;p^*w(V)) \quad \text{and} \quad {\alpha}_K(t) \in KO^n(DV, SV;p^*w(V)).$$ 
Just as in the classical situation Thom classes are defined as those classes
$t \in E^n(DV,SV;p^*w)$ pulling back to generators under the restriction maps
$$E^n(DV,SV;p^*w) \rightarrow E^n(DV_{|b},SV_{|b};w(V)_{|b})$$ for all $b\in B$. 
\begin{prop}
The classes $t$ and ${\alpha}_K(t)$ are Thom classes for $V$ with respect to the pairings 
\[\MSpin_*(-;-) \otimes \MSpin_*(-) \longrightarrow \MSpin_*(-;-)\] 
\[KO_*(-;-) \otimes KO_*(-) \longrightarrow KO_*(-;-)\]
coming from the module structures of twisted $Spin$-cobordism and $KO$-theory over their untwisted versions, respectively.\\ 
Therefore, any bundle $V \stackrel{p}{\rightarrow} B$ together with a representing map $B \stackrel{w(V)}{\rightarrow} K$ of its first and second Stiefel-Whitney class admits a Thom class in both twisted $Spin$-cobordism and twisted $KO$-theory.
\end{prop}

\begin{proof}
Since being a Thom class is a pointwise condition the proposition follows immediately from the standard construction of Thom classes in bordism theories (which the above reduces to in each fibre) and theorem \ref{mictho}.
\end{proof}

\begin{rem}
While Thom classes in general are not canonical but depend on the vertical homotopy class of a lifting $(DV,SV) \to (\MSpinK(n),K)$ over $w(V)$, we do have a canonical Thom classes in $\MSpin^n(\MSpinK(n),K;w)$ for the universal bundles in the present context: They are represented via identity maps.
\end{rem}

The fact that there is a Thom isomorphism attached to a Thom class in twisted homology and cohomology follows from the existence of a Serre spectral sequence for twisted homology and cohomology \cite[20.4.1]{MaSi} just as in the untwisted case. In the terminology of loc.~cit.~$X$ is the fibrewise one-point compactification $V^+$ of $V$, $J = w^*E$, and one has to invest the composition of isomorphisms 
$$(w(V)^*E)_*(V^+) \cong E^*(w(V)_!(V^+)) \cong E^*(DV,SV;p^*(w(V)))$$
where the first isomorphism is given by \cite[20.2.6]{MaSi}, and the second is just a change of notation. In particular we find the following.
\begin{prop}
There are isomorphisms 
$$KO^{k+n}(\MSpinK(n),K;w) \cong KO^{k+n}(D\gamma_n,S\gamma_n;w) \cong KO^k(BO(n))$$
given by excision and multiplication with the canonical Thom class just mentioned.
\end{prop}

This is the twisted analogue of the classical $K$-theory Thom isomorphism
$$KO^{k+n}(MSpin(n),*) =\widetilde{KO}^{k+n}(MSpin(n)) \cong KO^k(BSpin(n))$$
We shall see in the next section that this also works for connective $K$-theory, but there are a few technical difficulties we have to address first.

\subsection{Twisted connective covers.}

We need a construction of connective covers $E\langle n \rangle$ for an orthogonal spectrum $E$ that forms a lax monoidal continuous functor. An example of such is given by connective covers built using simplicial spaces as we explain below. Assuming this for the moment we can immediately produce twisted connective covers of $\MSpinK$ and $\KOK$ using the spaces
\begin{align*}
			(\MSpinK\langle n \rangle)_m &:= EP{\mathcal O}(\ell^2) \times_{P{\mathcal O}(\ell^2)} (MSpin'_m\langle n + m \rangle)\\
						(\KOK\langle n \rangle)_m &:= EP{\mathcal O}(\ell^2) \times_{P{\mathcal O}(\ell^2)} (KO'_m\langle n + m\rangle), 
\end{align*}
respectively. Recall that our convention for $n$-connective covers has $\pi_i(E\langle n \rangle)$ vanish for $i<n$ but not necessarily for $i = n$. From the functoriality we immediately obtain a lift of $\alpha': MSpin' \rightarrow KO'$ to $MSpin'\langle n \rangle \rightarrow KO'\langle n \rangle$ and preservation of continuous group actions then produces a map $\MSpinK\langle n \rangle \rightarrow \KOK\langle n \rangle$. In accordance with standard notation we set $\koK:=\KOK\langle 0 \rangle$. However, the map $\MSpinK\langle 0 \rangle \rightarrow \MSpinK$ is obviously a weak equivalence and we shall suppress it in much of what follows. All in all we obtain a transformation of twisted homology theories
$${\alpha}_K: \MSpin_*(-;-) \longrightarrow ko_*(-;-)$$
Furthermore, we have pairings $ko' \wedge KO\langle n \rangle \rightarrow KO' \langle n \rangle$ and consequently
$$\koK \barwedge KO\langle n \rangle \longrightarrow \KOK \langle n \rangle$$
 up to homotopy as follows: Consider the commutative diagram
$$\xymatrix@-1pc{ (ko' \wedge KO\langle n \rangle)\langle n \rangle \ar[r]\ar[d]_{\simeq} & (KO' \wedge KO)\langle n \rangle \ar[r]\ar[d] & KO'\langle n \rangle \ar[d] \\
              ko' \wedge KO\langle n \rangle \ar[r] & KO' \wedge KO \ar[r] & KO'}$$
where the first line arises from the second by taking the $n$-connective covers. As above we produce a map in the homotopy category
\begin{align*} \koK \barwedge KO\langle n \rangle &\stackrel{\simeq}{\longleftarrow} EP{\mathcal O}(\ell^2) \times_{P{\mathcal O}(\ell^2)} (ko' \wedge KO\langle n \rangle)\langle n \rangle \\
 						&\longrightarrow EP{\mathcal O}(\ell^2) \times_{P{\mathcal O}(\ell^2)} (KO' \wedge KO)\langle n \rangle \\
						&\longrightarrow \KOK\langle n \rangle
\end{align*}

As in the non-parametrised case, the resulting pairing is homotopy associative, but it also satifies all the other properties present in the non-parametrised case, so that we obtain a pairing of homology and cohomology theories. 

By inspection we find the following compatibility for Thom classes.

\begin{prop}\label{morethom}
The transformation ${\alpha}_K: \MSpin_*(-;-) \longrightarrow ko_*(-;-)$ maps Thom classes to Thom classes for the above pairing.
\end{prop}

Unraveling this we find that for a rank $k$ vector bundle $V$ as above, we have an isomorphism $$ko\langle l \rangle^{n+l}(DV,SV;p^*w(V)) \cong ko\langle l \rangle^n(B)$$

Finally, we present the definition of the $n$-connective cover functor we used above. 

\begin{prop}
Realising the simplicial \emph{space}
$$k \mapsto X\langle n \rangle_k = \{f:\Delta^k \rightarrow X \mid f\left((\Delta^k)^{(n-1)}\right) = x\}$$
associated to a pointed space $(X,x)$ produces an $n$-connective cover of $X$ with all the stated properties; here $(\Delta^k)^{(n-1)}$ denotes the simplicial $n-1$-skeleton of $\Delta^k$.
\end{prop}

\begin{proof}
All the assertions are more or less evident, except for maybe the fact that $| X\langle n \rangle_\bullet |$ is indeed a connective cover. This is wellknown for the untopologised version of the above construction (compare e.g. \cite[Paragraph 8]{Ma}) and can be reduced to that by the following trick which we thank Michael Weiss for pointing out to us: \\
The bisimplicial \emph{set} 
$$Z_{k,l} = \{f:\Delta^k \times \Delta^l \rightarrow X \mid f\left(\Delta^k \times (\Delta^l)^{(n-1)}\right) = x\}$$
in the one direction realises to $|Z_{k,\bullet}| = |Sing_\bullet(C(\Delta^k,X))\langle n \rangle|$, where $C(\Delta^k,X)$ denotes the \emph{space} of continuous maps and $\langle n \rangle$ denotes the simplicial \emph{set} construction of connective covers. The arising simplicial \emph{space} admits the obvious levelwise weak equivalence (by constant maps) from the constant simplicial space given by $|Sing_\bullet(X) \langle n \rangle|$. We conlude that $|Z_{\bullet,\bullet}|$ is a connective cover of $X$. \\
Realising in the other direction first produces $|Z_{\bullet,l}| = |Sing_\bullet(X\langle n \rangle_l)|$, the singularisation of the $l$-space of the singular space in question. Therefore, we conclude that $|Z_{\bullet,\bullet}|$ and $|X\langle n \rangle_\bullet|$ are weakly equivalent and a quick check of the maps involved yields the proposition.
\end{proof}

\subsection{Fundamental classes.}\label{fund}

Given a closed, connected, smooth $n$-manifold $M$ with stable normal bundle $\nu$ (as defined in \ref{from-surgery-to-cobordism}), a choice of bundle map $c_\nu: \nu \rightarrow \gamma$, where $\gamma$ is a universal $-n$-dimensional stable vector bundle over $BO$ determines a class
\[[M] \in \MSpin_n(M; w \circ \overline{c_\nu}),\]
the \emph{fundamental class} of $M$, by the Pontryagin-Thom construction. Recall that $w \colon BO \rightarrow K$ denotes our chosen representative for the first and second Stiefel-Whitney class so that $w \circ \overline{c_\nu}$ is then a representative of the first two normal Stiefel-Whitney classes of $M$. Let us also warn the reader that while the choice of $c$ is a contractible one, merely fixing the base space part $\overline{c_\nu}$ of $c_\nu$ (and thus the twist $w \circ \overline{c_\nu}$) does \emph{not} uniquely determine $[M]$; the additional datum corresponds to the choice of spin structure in the untwisted case (i.e. when a null-homotopy of $w \circ \overline{c_\nu}$ is given). The push forward of $[M]$ generates $\MSpin_*(M,M-x;w \circ \overline c_\nu)$ freely as a module over $\MSpin_*$ for every $x \in M$, because first the inclusion  $$\MSpin_*(U,U-x;w \circ \overline c_\nu) \rightarrow \MSpin_*(M,M-x;w \circ \overline c_\nu)$$
induces an isomorphism by excision, secondly for $U$ a ball around $x$, we find 
$$\MSpin_*(U,U-x;(w \circ \overline c_\nu)_{|U}) \cong \MSpin_*(U,U-x)$$
induced by some nullhomotopy of $w \circ \overline{c_\nu}_{|U}$, thirdly all of this is compatible with the $\MSpin_*$-module structure and finally $[M]$ corresponds to the image of a fundamental class in $\MSpin_n(U,\partial U)$. \\ 

Given in addition a map $f \colon M \rightarrow X$ for some space $X$ over $K$, say via $\zeta \colon X \rightarrow K$, such that the diagram
\[\xymatrix@-1pc{M \ar[d] \ar[rr]^f && X \ar[lld]^\zeta \\
            K && }\]
commutes, we obtain a class
\[f_*([M]) \in \MSpin_n(X;\zeta).\]
In other words any commutative diagram of the form
$$\xymatrix@-1pc{BO \ar[rrd]_{w} && M \ar[ll]_-c \ar[rr] && X \ar[lld]^\zeta \\
                                    && K&&}$$
(or indeed one with a specified homotopy between the two compositions) together with a concordance between $c^*(\gamma) \oplus TM$ and the trivial $n$-dimensional stable bundle provides a cycle for $\MSpin_n(X;\zeta)$.

Via the Pontryagin-Thom construction proposition \ref{norepl} shows that for $X$ a cell complex, all elements of $\MSpin_n(X;\zeta)$ are realised by such cycles and indeed that $\MSpin_n(X;\zeta)$ is given by quotienting these cycles by the appropriate bordism relation, which we will refrain from spelling out. Because of this description we will use the notation $\Omega^{Spin}_*(-;-)$ interchangably with $\MSpin_*(-;-)$ from now on; we shall not have to deal with spaces not of the homotopy type of a cell complex.
Now, if $M$ is connected and the universal cover of $M$ is spin, then we saw at the end of section \ref{from-surgery-to-cobordism}, that there is a commutative diagram
$$\xymatrix@-1pc{ M \ar[d] \ar[rr]^-p &&  B\pi \ar[lld]^{u} \\
                             K && }$$
where $\pi$ is the fundamental group of $M$ and $p$ classifies the universal cover of $M$. Extending this diagram to a cycle for $u: B\pi \to K$ we obtain a class 
$$p_*([M]) \in \MSpin_n(B\pi;u)$$

We are now finally able to formulate the conjecture that is the long term goal of our project.

\begin{conj}[Stolz 1995]\label{ourconj}
If for a connected, closed, smooth manifold $M$ of dimension $n \ge 5$, whose universal cover is spin, we have
$$0 = {\alpha}_K(p_*([M])) \in ko_n(B\pi; u)$$
for some (and then every) choice of necessary auxiliary data,
then $M$ carries a metric of positive scalar curvature.  
\end{conj}

\section{The generalised Anderson-Brown-Peterson splitting}

\subsection{Twisted $K$-Theory Pontryagin classes.}

With the Thom classes in place, we can give a generalisation of the Anderson-Brown-Peterson splitting. Note however, that while the focus of the original paper by Anderson, Brown and Peterson \cite{AnBrPe2} lies on the calculation of the $Spin$-cobordism ring, their result also facilitated computations of other $Spin$-cobordism groups by reduction to calculations in $KO$-theory. It is this second part that our result generalises to the case of twisted cobordism groups. We begin by recalling the essentials from \cite{AnBrPe2}. They define characteristic classes $\pi_k$ for vector bundles of rank $n$ (which is suppressed in the notation because of part (2) of the following proposition) with values in $KO^0$.

\begin{prop}\label{prop-of-Pontryagin-classes}
The following holds for vector bundles $E$ and $F$:
\begin{enumerate}
 \item $\pi_k(E \oplus F) = \sum_{i = 0}^k \pi_i(E) \cup \pi_j(F)$ 
 \item $\pi_k(E \oplus \mathbb R^m) = \pi_k(E)$
 \item $\pi_k(E) = 0$, if $2k > rank(E)$ and $E$ orientable 
\end{enumerate}
\end{prop}

\begin{proof}
These properties are all stated right in or around \cite[Proposition 5.1]{AnBrPe2}.
\end{proof}

Because of \ref{prop-of-Pontryagin-classes}(2) the $\pi_j$ for varying $n$ induce a compatible system of classes in $KO^0(BO(n))$ when applied to the universal bundle. By the Milnor-sequence one obtains classes $\pi_j \in KO^0(BO)$, since the $lim^1$-term vanishes by application of the real version of the Atiyah-Segal completion theorem and explicit computation of the representation rings involved. For a finite sequence (possibly empty) of natural numbers $J= (J_1, \dots , J_k)$ denote $\pi_J = \prod_i \pi_{J_i}$ and $n(J) = \sum_i J_i$. Obviously, sequences that arise by rearrangement give the same class. Anderson, Brown and Peterson then prove the following.

\begin{thm}\label{lift}
The class $\pi_J \in KO^0(BSO)$ lies in the image of the map $$ko\langle n_J \rangle^0(BSO) \rightarrow KO^0(BSO)$$ where 
$$n_J = \left\{\begin{array}{cl} 4n(J)        & n(J) \text{ even} \\
                                 4n(J) - 2    & n(J) \text{ odd} \end{array}\right.$$
\end{thm}

The same statement thus also holds for the restrictions to $BSpin$. The indeterminacy of such lifts is also determined in \cite{AnBrPe2}, but we shall not make use of that. Using the Thom classes for spin bundles along the pairing $ko^*(-) \times ko\langle n_J \rangle^*(-) \rightarrow ko\langle n_J \rangle^*(-)$, that we generalised in the last section, we find
\begin{align*} 
   [MSpin, ko\langle n_J \rangle] \longrightarrow & \lim_n ko\langle n_J\rangle^n(MSpin(n),*) \\
                                          \cong \ & \lim_n ko\langle n_J\rangle^0(BSpin(n),*) \\
 		                   \longleftarrow & ko\langle n_J\rangle^0(BSpin)
\end{align*}
where the first map is a surjection by the Milnor-sequence. We thus obtain a transformation $\MSpin_*(-) \rightarrow ko\langle n_J \rangle_*(-)$ for each $J$. We will use the obvious analogue of the above calculation
\begin{align*} 
   [\MSpinK, \koK\langle n_J \rangle] \longrightarrow & \lim_n ko\langle n_J\rangle^n(\MSpinK(n),K;w) \\
                                           \cong \ & \lim_n ko\langle n_J\rangle^0(BO(n)) \\
 		                    \longleftarrow & ko\langle n_J\rangle^0(BO)
\end{align*}
to analyse the twisted theory $\MSpin_*(-;-)$. However, we do not know whether theorem \ref{lift} also holds for $BO$ instead of $BSO$ or $BSpin$. To circumvent this we consider the following characteristic classes instead:

\begin{defi}
Put $\overline{\pi}_j(E) = \pi_j(E \oplus \Lambda^nE)$ for any vector bundle $E$ of rank $n$.
\end{defi}

By the second stated property of the $\pi_j$'s we have $\overline{\pi}_j(E) = \pi_j(E)$ for any orientable bundle. With the same discussion as above we thus obtain classes $\overline{\pi}_j \in KO^0(BO)$ and by the comment just made these restrict to our original $\pi_j \in KO^0(BSO)$. A lift of $\pi_J$ in $ko\langle n_J\rangle^0(BSO)$ thus determines a lift of $\overline{\pi}_j$ in $ko\langle n_J\rangle^0(BO)$ and a transformation $\MSpin_*(-;-) \rightarrow ko\langle n_J\rangle_*(-;-)$ of twisted homology theories. 

\begin{cor}\label{lift2}
The class $\overline{\pi}_J \in KO^0(BO)$ lies in the image of the map $ko\langle n_J \rangle^0(BO) \rightarrow KO^0(BO)$, where $n_J$ is as above.
\end{cor}

Underlying this construction is of course the homotopy equivalence
$$BO \rightarrow BSO \times K(\mathbb Z/2,1),$$
whose finite steps (i.e. the restrictions $BO(n) \rightarrow BSO(n+1) \times K(\mathbb Z/2,1)$) represent adding (in the first component) and remembering (in the second) the top exterior power of a bundle.

\subsection{The Anderson-Brown-Peterson splitting and its twisted generalisation.}\label{calc}

The investigation of the $Spin$-cobordism ring in \cite{AnBrPe2} starts with the following theorem.

\begin{thm}
If $\theta_J: MSpin \longrightarrow ko\langle n_J \rangle$ corresponds to $\pi_J$ under the Thom isomorphism, then, as $J$ runs through all non-decreasing sequences with $J_1 > 1$, the induced map
$$\xymatrix{ \bigoplus_J H^*(ko\langle n_J \rangle, \mathbb Z/2)  \ar[r] & H^*(MSpin, \mathbb Z/2)}$$
is injective with graded-free cokernel over $\mathcal A_2$.
\end{thm}

Choosing a split over $\mathcal A_2$ of the arising short exact sequences and a homogeneous basis of the cokernel we obtain maps $z_i: MSpin \longrightarrow \Sigma^{n_i} H\mathbb Z/2$. 

\begin{cor}
Given $\theta_J$ and $z_i$ as above the induced map
$$\xymatrix{ \bigoplus_J H^*(ko\langle n_J\rangle, \mathbb Z/2)  \oplus \bigoplus_i H^*(\Sigma^{n_i} H\mathbb Z/2, \mathbb Z/2)  \ar[r] & H^*(MSpin, \mathbb Z/2)}$$
is an isomorphism.
\end{cor}

Therefore the combined map 
\[(\theta_J,z_i)_{J,i} \colon MSpin \rightarrow \bigvee_J ko\langle n_J\rangle \vee \bigvee_i \Sigma^{n_i} H\mathbb Z/2\]
is an equivalence after $2$-completion (here the infinite coproducts are also products as the spectra involved grow higher connected with larger indices). Since the spectra on either side are of finite type we find:

\begin{cor}[Anderson-Brown-Peterson-splitting]\label{abp}
The transformation $$\MSpin_*(-) \longrightarrow \bigoplus_J ko\langle n_J \rangle_*(-) \oplus \bigoplus_i H_{*-n_i}(-,\mathbb Z/2)$$ given by a choice of $\theta_J$ and $z_i$ as above, is an isomorphism after localisation at $2$.
\end{cor}

To obtain a similar splitting for the twisted spin cobordism spectrum, note that, given such a choice of $z_i$, we obtain (upon further choices) homotopy classes
$$\MSpinK \longrightarrow K \times \Sigma^{n_i} H\mathbb Z/2$$
as follows: Recall the functor $\Theta: \mathcal S_K \rightarrow \mathcal S$, that collapses the base section and that $\Theta(\MSpinK) = MO$. One then is tempted to proceed as follows: It is well known that the canonical map
$$H^*(MO,\mathbb Z/2) \longrightarrow H^*(MSpin,\mathbb Z/2)$$
is a surjection, so pick representatives $\overline{z_i}: MO \rightarrow \Sigma^{n_i} H\mathbb Z/2$ of inverse images of the $z_i$.
Now $\Theta$ is left adjoint to $- \times K$ (also on the homotopy category), whence these should correspond to maps $\MSpinK \rightarrow \Sigma^{n_i} H\mathbb Z/2 \times K$ as claimed. However, we do not know whether our spectrum $\MSpinK$ is cofibrant, i.e.~it is not clear that $\Theta(\MSpinK) \cong MO$ in the homotopy category. To verify that this is nevertheless the case we need to peek into the internal workings of the model structures on parametrised spectra once more. We can cofibrantly resolve $\MSpinK$ in the level model structure, since it shares cofibrations with the stable structure and has a stronger notion of weak equivalence. However, as we observed in the proof of \ref{norepl} the base sections in our spectrum $\MSpinK$ are cofibrations and so we certainly obtain a levelwise homology equivalence after applying $\Theta$ to the chosen resolution, and hence a weak equivalence. \\

Note that the parametrised spectrum $K \times H\mathbb Z/2$ represents the twisted homology theory obtained by ignoring the twist and taking singular mod $2$ homology, i.e. ~if $X$ is a space over $K$ we have $(K \times H\mathbb Z/2)^*(X) = H^*(X,\mathbb Z/2)$. 
\begin{cor}\label{genspli}
If $\overline{\theta}_J : \MSpinK \longrightarrow \koK\langle n_J \rangle$ corresponds to a lift of $\overline{\pi}_J \in KO^0(BO)$ in $ko\langle n_j \rangle^0(BO)$ and $\overline{z_i}: \MSpinK \rightarrow K \times \Sigma^{n_i} H\mathbb Z/2$ corresponds to a class in $H^{n_i}(MO,\mathbb Z/2)$ that restricts to $z_i \in H^{n_i}(MSpin, \mathbb Z/2)$, then the induced map
$${\big((\overline{\theta}_J,\overline{z}_i)_{J,i}\big)}_* \colon \MSpin_*(-;-) \longrightarrow \bigoplus_J ko\langle n_J \rangle_*(-;-) \oplus \bigoplus_i H_{*+n_i}(-,\mathbb Z/2)$$
is an isomorphism of twisted homology theories after localising at $2$. 
\end{cor}

So again the combined map
\[(\overline{\theta}_J,\overline{z}_i)_{J,i} \colon MSpin_K \rightarrow \bigvee_J ko_K\langle n_J\rangle \vee \bigvee_i K \times \Sigma^{n_i} H\mathbb Z/2\]
is a $2$-local equivalence.

\begin{proof}
To have an isomorphism in parametrised homology theory corresponds to having a weak equivalence of parametrised spectra. To check this it suffices to verify, that the statement holds for the fibre over each point $k \in K$ and as before, for every point the claim reduces to the classical Anderson-Brown-Peterson splitting, albeit in a non-canonical way.
\end{proof}

\section{The cohomology of twisted $Spin$-cobordism}\label{seccoh}

\subsection{The twisted Steenrod algebra.} Using our generalisation of the Ander- son-Brown-Peterson splitting we shall now set out to compute the $\mathbb Z/2$-cohomology of the twisted $Spin$-cobordism spectrum as a module over the twisted Steenrod algebra, which we denote by $\underline{\mathcal A}_2$. Before doing so we shall therefore describe $\underline{\mathcal A}_2$. For simplicity we shall from now on suppress the coefficients $\mathbb Z/2$ and the related subscript $2$ at the Steenrod algebra from our notation.

\begin{defi} \label{twisted-Steenrod-algebra}
The graded algebra $\underline{\mathcal A}$ is defined to be the family of natural transformations $H^*(-) \rightarrow H^*(-)$ commuting with the connecting transformation, where $H^*$ is regarded as a functor $Top_K \rightarrow Ab$ by ignoring the twist datum.
\end{defi}

As in the untwisted setting $\underline{\mathcal A}$ is an algebra under composition and carries a canonical coproduct coming from the multiplicativity of $H^*$, making it a co-commutative Hopf algebra. Representability easily yields:

\begin{prop}
The inclusions $H^*(K) \rightarrow \underline{\mathcal A}$ (acting via multiplication) and $\mathcal A \rightarrow \underline{\mathcal A}$ induce an isomorphism 
$$H^*(K) \otimes \mathcal A \cong \underline{\mathcal A}$$
of vector spaces and coalgebras, but not of algebras.
\end{prop}

\begin{proof}
What we have to compute is $\lim_n H^{*+n}(K \times K(\mathbb Z/2,n), K \times *)$, which by the K\"unneth-formula is isomorphic to $H^*(K) \otimes \mathcal A$, and since the multiplication of $K \times H\mathbb Z/2$ is componentwise, this is an isomorphism of coalgebras. The final statement follows from the next proposition.
\end{proof}

The multiplication, however, is also easily described in terms of the isomorphism above.

\begin{prop}
For $k,l \in H^*(K)$ and $a,b \in \mathcal A$, with $\Delta(a) = \sum_i a'_i \otimes a''_i$ we have
$$(k \otimes a) \cdot (l \otimes b) = \sum_i (k \cup a'_i(l)) \otimes (a''_i \circ b)$$
\end{prop}

\begin{proof}
Given a pair $(X,A)$ with twist $\zeta$, we find for every $x \in H^*(X,A,\zeta)$
\begin{align*} 
                 (k \otimes a) \big((l \otimes b) x)\big) &= (k \otimes a) (\zeta^*l \cup b(x)) \\
							     &= \zeta^*k \cup a(\zeta^*l \cup b(x)) \\
							     &= \sum_i \zeta^*k \cup a'_i(\zeta^*l) \cup a''_i(b(x)) \\
							     &= \sum_i \zeta^*(k \cup a'_i(l)) \cup (a''_i \circ b)(x) \\
							     &= \sum_i \big((k \cup a'_i(l)) \otimes (a''_i \circ b)\big) x
\end{align*}
\end{proof}

\begin{rem}
From these formulas we see that $\underline{\mathcal A}$ is precisely the semidirect product $H^*(K) \sd \mathcal A$ of ${\mathcal A}$ acting on $H^*(K)$, see \cite{MaPe}. As usual we conclude that the inclusions $H^*(K) \hookrightarrow \underline{\mathcal A}$ and $\mathcal A \hookrightarrow \underline{\mathcal A}$ are Hopf algebra maps, whereas of the projection maps $\underline{\mathcal A} \rightarrow \mathcal A$ (killing $H^+(K) \otimes \mathcal A$) and $\underline{\mathcal A} \rightarrow H^*(K)$ 
(killing $H^*(K) \otimes \mathcal A_+$) only the first is multiplicative, while both are comultiplicative (since both ${\mathcal A}$ and $H^*(K)$ are cocommutative); here $H^+(K)$ and $A_+$ denote positive degree elements.
\end{rem}

To describe our results for $H^*(\MSpinK)$ we will construct a nontrivial automorphism of a subalgebra of $\underline{{\mathcal A}(1)}$, that emulates the changes under Thom isomorphisms: 
To this end let ${\mathcal A}(1)$ denote the subalgebra of $A$ generated by $Sq^1$ and $Sq^2$, which indeed is a sub-Hopf-algebra. 
Similarly, let $\underline{\mathcal A(1)}$ denote the subalgebra of $\underline{{\mathcal A}}$ generated by ${\mathcal A}(1)$ and $H^*(K)$; it is also readily checked to be a sub-Hopf-algebra of $\underline{{\mathcal A}}$.
Now note that given the spectrum $\MSpinK$, we can apply the Thom isomorphism $H^*(\MSpinK) = H^*(MO) \cong H^*(BO)$, which (by our convention from right after lemma \ref{fibtho}) is $H^*(K)$-linear for the action of $H^*(K)$ on $H^*(BO)$ given via the first and second Stiefel-Whitney class (considered as a map $BO \rightarrow K$).
The module structure over the Steenrod algebra, however, changes. By Thom's formula $Sq^i u = w_i \cup u$, the change in the action over ${\mathcal A}(1)$ can be described by an automorphism $\psi$ of $\underline{{\mathcal A}(1)}$. 
It needs to satisfy
\begin{align*}
\psi(1 \otimes Sq^1) &= 1 \otimes Sq^1 + \iota_1 \otimes 1 \\
\psi(1 \otimes Sq^2) &= 1 \otimes Sq^2 + \iota_1 \otimes Sq^1 + \iota_2 \otimes 1 \\
\psi(k \otimes 1)    &= k \otimes 1.
\end{align*}

To verify that this stipulation really defines an automorphism of $\underline{{\mathcal A}(1)}$ it turns out to be more convenient to work with the inverse of $\psi$, which we denote by $\varphi$. It is given by
\begin{align*}
\varphi(Sq^1) &= 1 \otimes Sq^1 + \iota_1 \otimes 1 \\
\varphi(Sq^2) &= 1 \otimes Sq^2 + \iota_1 \otimes Sq^1 + \iota_1^2 \otimes 1 + \iota_2 \otimes 1.
\end{align*}

\begin{lem}\label{twiemb}
The above stipulation indeed defines a unique homomorphism $\varphi: {\mathcal A}(1) \rightarrow \underline{{\mathcal A}(1)}$ of algebras. It is a morphism of Hopf algebras.
\end{lem}

\begin{proof}
This is a lengthy but straight forward computation, since a presentation for $\mathcal A(1)$ is given by $(Sq^1)^2 = 0, Sq^1 Sq^2 Sq^1 = (Sq^2)^2$. For completeness' sake we give it in Appendix A.
\end{proof}

\begin{lem}\label{twimod}
Extending $\varphi$ by the identity on $H^*(K)$ produces a Hopf algebra automorphism of $\underline{A(1)}$ of order 4. 
\end{lem}

\begin{proof}
 To make the extension well-defined, there is again a relation to be checked and we do this in Appendix A as well. 
 The extension is an isomorphism, since it is for example $H^*(K)$-linear and induces the identity on $H^*(K)$-indecomposables. That it has order 4 is easily verified on generators.
\end{proof}

The verification that the inverse of $\varphi$ behaves as desired is a little calculation that we leave to the reader.

\begin{prop}
As modules over $\underline{\mathcal A(1)}$ we have  $H^*(\MSpinK) \cong {} _{\psi}H^*(BO)$, where the lower case $\psi$ indicates the pulled back module structure.
\end{prop}

\subsection{The cohomology of $\koK$.}

We begin with a few simple observations. By the adjointness of the derived versions of the functors $\Theta \colon \mathcal S_K \rightarrow \mathcal S$ and  $K \times - \colon \mathcal S \rightarrow \mathcal S_K$ we have
$$H^*(E) = H^*(\Theta cE)$$
for every spectrum $E$ over $K$, where the left hand side denotes the twisted cohomology theory represented by $K \times H\mathbb Z/2$ and the right hand side the untwisted one represented by $H\mathbb Z/2$ (and $c$ is cofibrant resolution). Since we may as well replace all spectra occuring in the following by their cofibrant replacements, we shall suppress it from notation.

\begin{lem}\label{bashom}The spectrum $\Theta(\koK)$ is connective with zeroth homotopy group isomorphic to $\mathbb Z/2$. The corresponding Postnikov section yields a homotopy class of maps of parametrised spectra $$\koK \longrightarrow K \times H\mathbb Z/2$$
This map induces a surjection $\underline{\mathcal A} \longrightarrow H^*(\koK)$ upon passage to cohomology. Furthermore, the Poincar\'e series of $H^*(\koK)$ equals the product of that of $H^*(ko)$ and that of $H^*(K)$.
\end{lem}

\begin{proof}
The statement about the homotopy groups of $\Theta(\koK)$ follows from the fact that the Atiyah-Bott-Shapiro orientation $MSpin \rightarrow ko$ is an $8$-equivalence integrally (and not just $2$-locally). By the long exact sequences of the fibrations
\[(MSpin_K)_n \rightarrow K \text{    and    } (\koK)_n \rightarrow K\] 
the same then holds true for the twisted Atiyah-Bott-Shapiro orientation. For example by considering homology we see that the induced map $MO = \Theta(\MSpinK) \rightarrow \Theta(\koK)$ is then also an $8$-equivalence and the homotopy groups of the left side are as claimed. \\
The second statements follow from the Serre spectral sequences: Stong's calculations of $H^*(BO\langle k \rangle)$ (compare \cite[Chapter XI, Proposition 6]{Stong} or \cite{St 63}) show that the map $\mathcal A \rightarrow H^*(ko)$ induced by the zeroth Postnikov section of $ko$ is a surjection with kernel generated by $Sq^1$ and $Sq^2$. The transformation of spectral sequences induced by the Postnikov section $\koK \longrightarrow K \times H\mathbb Z/2$ therefore also is a surjection on the second pages; this is clear once we know that the coefficient system in the spectral sequence for $\koK$ is constant, which we show below. Since the domain spectral sequence (that of $H^*(K \times H\mathbb Z/2)$) collapses, this implies that the target spectral sequence collapses as well and we have a surjection on the limit pages and thus a surjection on the abutment, i.e. $\underline{\mathcal A} \rightarrow H^*(\koK)$ is onto. 

Finally, let us prove that the coefficients in the spectral sequence for $\MSpinK$ (which implies the same for $\koK$, as $\koK$ is a summand in $\MSpinK$ by \ref{genspli}) are indeed constant. To this end observe that the grading map $P{\mathcal O}(\ell^2) \rightarrow \mathbb Z/2$ splits (one such split is given by letting the nontrivial element in $\mathbb Z/2$ act by swapping the standard base elements $\delta_{2n} \in (\ell^2)^{ev}$ and $\delta_{2n+1} \in (\ell^2)^{odd}$). Call a split $s$, let $g$ be the image on the non-trivial element and observe that $s$ induces an isomorphism on path components and thus an isomorphism on $\pi_1$ after taking classifying spaces. Considering the arising diagram
$$\begin{xy}\xymatrix@-1pc{MSpin'(n) \ar[r]\ar[d]                                   & MSpin'(n) \ar[d] \\
                    E\mathbb Z/2 \times_{\mathbb Z/2} MSpin'(n) \ar[r]\ar[d] & EP{\mathcal O}(\ell^2) \times_{P{\mathcal O}(\ell^2)} MSpin'(n) \ar[d] \\
                    B\mathbb Z/2 \ar[r]                                     & BP{\mathcal O}(\ell^2)}
\end{xy}$$
we find that the action of the nontrivial element in $\pi_1(BP{\mathcal O}(\ell^2))$ on the cohomology of the fibre is induced by the action of the $g$ on $MSpin'(n)$. By the Thom isomorphism we therefore need to study the action of $g$ on $BSpin'(n)$ and this is given by left multiplication of $g$ on
$$BSpin'(n) = {\mathcal O}'^+_n/Spin(n) = P{\mathcal O}(\ell^2)'_n/O(n)$$
This action is covered by a bundle isomorphism of the universal $O(n)$-bundle over $BSpin'(n)$, which will yield the claim since its Stiefel-Whitney classes generate the mod $2$ cohomology ring of $BSpin'(n)$. The universal principal $O(n)$-bundle is given by 
$$\begin{xy}\xymatrix@-1pc{{\mathcal O}'^+_n \times_{Spin(n)} O(n) \ar[d]\ar[rr]^-\cong && {\mathcal O}'_n \times_{Pin(n)} O(n) \\
                              BSpin'(n)                       &&}
\end{xy}$$
and the isomorphism covering the left multiplication by $g$ is again just given by left multiplication with $g$ using the right hand side description of the total space. Note that this does not work for the universal $SO(n)$ or $Spin(n)$ bundles since $Pin(n)$ does not act on these. The coefficient system is indeed non-trivial using integral coefficients (since Pontryagin classes are sensitive to orientations). 
\end{proof}

In particular, we find $H^0(\koK) \cong \mathbb Z/2$. Let us denotes by $\kappa$ the unique non-zero class in that group. Using the splitting of $\MSpinK$ once more we find:

\begin{lem}
We have $\varphi(Sq^1) \kappa = 0$ and $\varphi(Sq^2) \kappa = 0$.
\end{lem}

\begin{proof}
Since the Atiyah-Bott-Shapiro orientation is injective, we can check $\varphi(Sq^1)u = 0$ and $\varphi(Sq^2)u = 0$ in $H^*(\MSpinK) = H^*(MO)$, where the computation reads
$$\varphi(Sq^1) u = Sq^1(u) + \iota_1 u = w_1u + w_1u = 0$$
$$\varphi(Sq^2) u = Sq^2(u) + \iota_1Sq^1(u) + \iota_1^2u + \iota_2u = w_2u + w_1^2u + w_1^2u + w_2u = 0$$
\end{proof}

\begin{thm}\label{632}
There is a unique non-trivial class $\kappa_{8n} \in H^{8n}(\koK \langle 8n \rangle)$ and the map $sh^{8n}\underline{\mathcal A} \rightarrow H^*(\koK\langle 8n \rangle)$ given by evaluation on $\kappa_{8n}$ induces an isomorphism
$$sh^{8n}\underline{\mathcal A}_\varphi \otimes_{\mathcal A(1)} \mathbb Z/2 \rightarrow H^*(\koK\langle 8n \rangle)$$
of left $\underline{\mathcal A}$-modules.
\end{thm}

\begin{proof}
Bott periodicity immediately reduces the claim to the case $n = 0$. 
By the above lemma the map indeed factors through $\underline{\mathcal A} / \langle\varphi(Sq^{1}),\varphi(Sq^2)\rangle$ and it is surjective by the proposition above that. 
We will thus be done once we have computed the Poincar\'e series of both sides. 
For the right hand side this was also done in the lemma above, and for the left hand side all that remains to be noted is that the Poincar\'e-series of
$\underline{\mathcal A}_\varphi \otimes_{\mathcal A(1)} \mathbb Z/2$ agrees with that of $\underline{\mathcal A} \otimes_{\mathcal A(1)} \mathbb Z/2$
by two applications of \cite[Proposition 1.7]{MiMo} and that the latter is indeed the tensor product of $H^*(K)$ and $\mathcal A \otimes_{\mathcal A(1)} \mathbb Z/2$ 
at least as a $\mathbb Z/2$ vector space. At this point we are done since $H^*(ko) \cong \mathcal A \otimes_{\mathcal A(1)} \mathbb Z/2$ by Stong's calculation \cite{St 63}.
\end{proof}

By similar arguments we can also deal with $\koK\langle 2 \rangle$.

\begin{thm}\label{cohok}
There is a unique non-trivial class $\kappa_{8n+2} \in H^{8n+2}(\koK \langle 8n+2 \rangle)$ and the map $sh^{(8n+2)}\underline{\mathcal A} \rightarrow H^*(\koK\langle 8n + 2 \rangle)$ given by evaluation on $\kappa_{8n+2}$ induces an isomorphism
$$sh^{(8n+2)} \underline{\mathcal A}_\varphi \otimes_{\mathcal A(1)} \mathcal A(1)/\langle Sq^3\rangle \longrightarrow H^*(\koK\langle8n+2\rangle)$$
of left $\underline{\mathcal A}$-modules.
\end{thm}

\begin{proof}
By Bott periodicity it suffices again to consider a single value of $n$, which we choose to be $1$, since $\koK\langle 10 \rangle$ is a summand of $\MSpinK$. 
In particular, the Serre spectral sequence of $\koK\langle 10 \rangle$ has constant coefficients (as observed in the proof of \ref{bashom}), 
giving us the element $\kappa_{10}$. The computation of the Poincar\'e series is also basically the same as in the previous argument, 
however, $\varphi(Sq^3)\kappa_{10} = 0$ does not follow directly, since we do not know the image of $\kappa_{10}$ in $H^{10}(\MSpinK)$ under a lift 
$\overline{\theta}_3: \MSpinK \rightarrow \koK\langle 10 \rangle$ as in the generalised Anderson-Brown-Peterson splitting.
To proceed we shall decompose $\overline{\theta}_3$ into its components:
$$\MSpinK \stackrel{\Delta}{\longrightarrow} BO_+ \barwedge \MSpinK \xrightarrow{\tilde{\pi}_3 \barwedge {\alpha}_K} \Omega^\infty ko\langle 10 \rangle \barwedge \koK \stackrel{\mu}{\longrightarrow} \koK\langle 10 \rangle$$
where $\tilde{\pi}_3 \in ko\langle 10 \rangle^0(BO)$ corresponds to $\overline{\theta}_3$ under the Thom isomorphism. Because $\mu^*\kappa_{10}$ clearly equals $\lambda_{10} \times \kappa$, where $\lambda_{10}$ refers to the fundamental class in $H^{10}(\Omega^\infty ko\langle 10 \rangle)$ ($\Omega^\infty ko\langle 10 \rangle$ is $9$-connected with $10$th homotopy group isomorphic to $\mathbb Z/2$), the following calculation gives the claim.
\begin{align*}
\varphi(Sq^3)(\lambda_{10} \times \kappa) &= \varphi(Sq^1)\varphi(Sq^2)(\lambda_{10} \times \kappa) \\
                                          &= \varphi(Sq^1)\Big((1 \otimes Sq^2)(\lambda_{10} \times \kappa) + (\iota_1 \otimes Sq^1)(\lambda_{10} \times \kappa) \\
                                          &\ \quad + (\iota_1^2 \otimes 1)(\lambda_{10} \times \kappa) + (\iota_2 \otimes 1)(\lambda_{10} \times \kappa)\Big) \\
                                          &= \varphi(Sq^1)\Big(\big[Sq^2(\lambda_{10}) \times \iota + \underbrace{Sq^1(\lambda_{10}) \times Sq^1(\kappa)}_{= \cancel{Sq^1(\lambda_{10}) \times \iota_1\kappa}} + \underbrace{\lambda_{10} \times Sq^2(\kappa)}_{= \cancel{\lambda_{10} \times \iota_2\kappa}}\big] \\
                                          &\ \quad + \big[\cancel{Sq^1(\lambda_{10}) \times \iota_1\kappa} 
                                          + \underbrace{\lambda_{10} \times \iota_1Sq^1(\kappa)}_{\cancel{\lambda_{10} \times \iota_1^2\kappa}}\big] 
                                          + \cancel{\lambda_{10} \times \iota_1^2\kappa} + \cancel{\lambda_{10} \times \iota_2\kappa}\Big) \\
                                          &= (1 \otimes Sq^1)(Sq^2(\lambda_{10}) \times \kappa) + (\iota_1 \otimes 1)(Sq^2(\lambda_{10}) \times \kappa) \\
                                          &= \big[\underbrace{Sq^1Sq^2(\lambda_{10}) \times \kappa}_{= 0} 
                                          + \underbrace{Sq^2(\lambda_{10}) \times Sq^1(\kappa)}_{= Sq^2(\lambda_{10}) \times \iota_1\kappa}\big] 
                                          + Sq^2(\lambda_{10}) \times \iota_1\kappa \\
                                          &= 0
\end{align*}
where $Sq^1Sq^2(\lambda_{10}) = 0$ by Stong's calculations of $H^*(BO\langle n \rangle)$ in \cite{St 63}.
\end{proof}
 
With these results it seems natural to guess the following.

\begin{conj}
There are isomorphisms
\begin{align*}
sh^{8n+1}\underline{\mathcal A}/ \langle\varphi(Sq^2)\rangle     &\stackrel{\cong}{\longrightarrow} H^*(\koK\langle 8n+1 \rangle) \\
sh^{8n+4}\underline{\mathcal A}/\langle\varphi(Sq^{1}),\varphi(Sq^5)\rangle &\stackrel{\cong}{\longrightarrow} H^*(\koK\langle 8n+4 \rangle) 
\end{align*}
of $\underline{\mathcal A}$-modules generalising Stong's computation of $H^*(ko\langle n \rangle)$.
\end{conj}

However, the arguments used above break down. 
Another approach would be to use the Maholwald's observation that $ko\langle n \rangle \simeq X_n \wedge ko$ for some small space $X_n$ and extend the map
$X_n \rightarrow ko\langle n \rangle$ to an equivalence $\koK\langle n \rangle \simeq X_n \barwedge \koK$ using the $ko$-module structure on $\koK\langle n \rangle$. 
We shall not investigated the details of this here.

\subsection{The cohomology of $\MSpinK$.}

Our generalised Anderson-Brown-Peterson splitting allows us to compute the cohomology of $\MSpinK$ as a module over the twisted Steenrod algebra. The isomorphism, however, we cannot uniquely specify, just as Anderson, Brown and Peterson could not determine theirs due to the non-uniqueness of lifts in theorem \ref{lift}.

\begin{cor}\label{cohom}
Any choice of splitting from \ref{abp} determines an isomorphism
$$H^*(\MSpinK) \cong \bigoplus_{J, n(J)\atop even} sh^{n_J} \underline{\mathcal A}/\langle\varphi(Sq^1), \varphi(Sq^2)\rangle \oplus \bigoplus_{J, n(J)\atop odd} sh^{n_J} \underline{\mathcal A}/\langle\varphi(Sq^3)\rangle \oplus \bigoplus_i sh^{n_i} \underline{\mathcal A}$$
of modules over the twisted Steenrod algebra.
\end{cor}

In particular, we see that $H^*(\MSpinK) \cong \underline{\mathcal A}_\varphi \otimes_{\mathcal A(1)} M$ for the $\mathcal A(1)$-module $M$ given as a direct
sum of $\mathbb Z/2$'s, jokers (that is $\mathcal A(1)/Sq^3$'s) and free modules according to the above decomposition. Since $H^*(MSpin) \cong \mathcal A_\varphi \otimes_{\mathcal A(1)} M$ for that very same $M$ we shall refer to it as the \emph{Anderson-Brown-Peterson module}. \\
This explicitly given structure as an extended module over the twisted Steenrod algebra can be used to get a hand on the second page 
of a $K\times H\mathbb Z/2$-based Adams spectral sequence. We hope to take advantage of this to approach conjecture \ref{ourconj} along the lines of 
Stolz' arguments in \cite{St2}, which uses the Adams spectral sequence as a principal tool.

\section{A first glance at the twisted version of Stolz' transfer.}

As described in the introduction our interest in twisted $Spin$-cobordism groups arises from their connection to the existence of metrics of positive scalar curvature.
The case of spin manifolds in conjecture \ref{ourconj} from the first chapter was proven by F\"uhring and Stolz. To state their results we need to rephrase the conjecture as the degree-greater-than-4 case of the inclusion
$$ker\big(\alpha_K: \Omega^{Spin}_*(B\pi;u) \rightarrow ko_*(B\pi;u)\big) \subseteq \Omega_*^{Spin}(B\pi;u)^+$$
for every finitely presented group $\pi$, and any twist $u: B\pi \to K$, where $\Omega_*^{Spin}(-)^+ \subseteq \Omega_*^{Spin}(-)$ denotes those elements which can be represented 
by manifolds admitting positive scalar curvature metrics. In this rather algebraic formulation we can split up the statement by localising at and inverting $2$, respectively, and Stolz proved the $2$-primary in \cite{St2} and F\"uhring the odd-primary part in \cite{Fu}, both under the assumption $u=0$. \\
We now set out to explain the method Stolz employed and indicate how one might hope it generalises
to include the case of almost spin manifolds. One should also be able to adapt F\"uhring's arguments, but we will not discuss that here as his proof is by different, far more geometric methods. \\
Since explicit generators even for the plain $Spin$-cobordism ring are hard to come by (and still not known outside small degrees), Stolz' idea was to consider total spaces
of fibre bundles with fixed fibre $F$ and structure group $G$ over varying base spaces. For appropriately chosen fibre these will always admit metrics of positive scalar curvature.
The entirety of such bundles in a cobordism theory can generically be described as the image of the following transfer map
\begin{align*}
T: \Omega_*(X \times BG) &\longrightarrow \Omega_{*+dim(F)}(X) \\
[M,(f,g):M \rightarrow X \times BG] &\longmapsto [g^*(EF), f \circ p: g^*(EF) \rightarrow M \rightarrow X]
\end{align*}
where $EF = EG \times_G F$ denotes the universal bundle over $BG$ with fibre $F$; imposing decorations like twists and $Spin$-structures needs additional assumptions on $G$ and $F$ of course. 
For the problem at hand Stolz chooses $F$ as $\mathbb HP^2$ and $G = Isom(\mathbb HP^2)$, where the isometry group (which in fact is $PSp(3)$) is taken with respect to the Fubini-Study metric on $\mathbb HP^2$, which has positive scalar curvature.
He then goes on to show that the vertical tangent bundle of $EF \rightarrow BG$ admits a $Spin$-structure, wherefore the above map makes sense with $Spin$-decorations. We will review this in a moment.
The $2$-primary part of \ref{ourconj} for spin manifolds now follows from
\begin{uthm}[Stolz 1994]
We have 
$$ker\big(\alpha: \Omega^{Spin}_*(X) \rightarrow ko_*(X)\big)_{(2)} \subseteq im\big(T: \Omega_{*-8}^{Spin}(X \times BPSp(3)) \rightarrow \Omega_{*}^{Spin}(X)\big)_{(2)}$$
for every space $X$.
\end{uthm}
The transfer map is induced by an explicitly given map of spectra 
$$t \colon MSpin \wedge S^8 \wedge BG_+ \longrightarrow MSpin$$
and the composition
$$MSpin\wedge S^8 \wedge BG_+ \stackrel t\longrightarrow MSpin \stackrel{\alpha}{\longrightarrow} ko$$
is nullhomotopic by the family version of the Atiyah-Singer index theorem (compare \cite[Section 1 \& 2]{St}). Therefore, one can pick a lift $\hat t: MSpin\wedge S^8 \wedge BG_+ \rightarrow \hofib(\alpha)$ of the transfer map. To obtain the result above Stolz then establishes the following two results:
\begin{uthm}[Proposition 1.3 \cite{St}]
Any such $\hat t$ induces a injection in $mod\ 2$-cohomology, which splits over $\mathcal A$.
\end{uthm}
\begin{uthm}[Proposition 8.3 \cite{St2}]
Any split from the previous theorem is realised by a map of ($2$-localised) spectra.
\end{uthm}
From these two results it follows immediately that $t$ induces a split surjection of homology theories (localised at $2$), which has the desired statement as a corollary. 

\subsection{The twisted transfer.}
In this very short section we will sketch the construction of the twisted transfer map, we refer the reader to the exposition in \cite{St} and the references therein for details and proofs, 
which work equally well in the parametrised setting. \\
Recall that given a smooth manifold bundle $p:E \rightarrow B$ with compact $k$-dimensional fibres there is an associated stable Pontryagin-Thom map $\Sigma^\infty B_+ \rightarrow M(-T_vp)$ 
for some choice of inverse of the vertical tangent bundle $T_vp$ of $p$; here $M(-T_vp)$ by our convention from Section \ref{from-surgery-to-cobordism} is considered as a $-k$-connective spectrum (i.e. its lowest non-trivial homotopy group is in degree $-k$). If this inverse of the vertical tangent bundle comes equipped with a $Spin$-structure one obtains a homotopy class 
$S^k \wedge M(-T_vp) \rightarrow MSpin$. Using these two maps we can form the following composition:
\begin{multline*}
\MSpinK \barwedge (S^k \wedge {B_+}) \longrightarrow \MSpinK \barwedge (S^k \wedge M(-T_vp)) \\
\longrightarrow \MSpinK \barwedge MSpin \longrightarrow \MSpinK
\end{multline*}
Interjecting this with the Thom diagonal $M(-T_vp) \rightarrow M(-T_vp) \wedge E_+$ produces the usual integration along the fibres, 
but we shall not need that here. Just as in the untwisted situation the geometric interpretation of the induced map $T_K \colon \Omega^{Spin}_*(X \times B;\zeta \circ \mathrm{pr}_1) \rightarrow \Omega^{Spin}_{*+8}(X;\zeta)$ for $\zeta \colon X \rightarrow K$ in terms of the cycles from Section \ref{fund} is simple enough: It sends
a cycle of the form
$$\xymatrix@-1pc{BO \ar[rrd]_{w} && M \ar[ll] \ar[rr] && X \ar[lld]^\zeta \\
                                    && K                           &&                 }$$
together with a map $g:M \rightarrow B$ to the solidly drawn part of 
$$\xymatrix@-1pc{ BO(k) \times BSpin(l) \ar@{..>}[d]    &&  &&  M \ar@{..>}[d] \\
             BO(k+l) \ar[drr]_{w} && g^*(E) \ar[rr] \ar@{..>}[rru] \ar@{..>}[ull] \ar[ll] && X \ar[dll]^\zeta \\
                                             && K                          & &}$$
where the map $g^*(E) \rightarrow BO(k) \times BSpin(l)$ is given by $g^*(E) \rightarrow M \rightarrow BO(k)$ in the first coordinate and a classifying map for some destabilisation of $-T_vp$ over $g^*(E)$ in the second. 
Note that for some large enough $l$ all such destabilisations produce the same fibre-homotopy class of maps $g^*(E) \rightarrow BO(k+l)$, so the class represented by this cycle does
not depend on the choice made. Furthermore, by our construction of the $BO(i)$ earlier the diagram still commutes strictly.\\
We want to apply this to the bundle mentioned in the introduction: Let $G = PSp(3) = Isom(\mathbb HP^2)$ and consider $EG \times_G \mathbb HP^2 \rightarrow BG$. We obtain the twisted transfer map
$$t_K: \MSpinK \barwedge (S^8 \wedge {BG_+}) \longrightarrow \MSpinK$$
The image of this transfer in $\Omega^{Spin}_*(X,\zeta)$ will consist entirely of manifolds admitting metrics of positive scalar curvature, 
since they are total spaces of bundles over compact manifolds, whose fibres $\mathbb HP^2$ admit such metrics.
Finally, we have the following evident generalisation:
\begin{prop}
 The composition
 $$\MSpinK \barwedge (S^8 \wedge {BG_+}) \stackrel{t_K}\longrightarrow \MSpinK \xrightarrow{{\alpha}_K} \koK$$
 is nullhomotopic.
\end{prop}

\begin{proof}
 Since ${\alpha}_K$ is an $MSpin$-module map, we find the above map equal to 
\begin{multline*}
\MSpinK \barwedge (S^8 \wedge {BG_+}) \longrightarrow \MSpinK \barwedge (S^8 \wedge M(-T_vp)) \\
\longrightarrow \MSpinK \barwedge MSpin \xrightarrow{{\alpha}_K \barwedge \alpha} \koK \barwedge ko \longrightarrow \koK
\end{multline*}
 However, by the Atiyah-Singer index theorem for families the composition $S^8 \wedge BG_+ \rightarrow S^8 \wedge M(-T_vp) \rightarrow MSpin \rightarrow ko$ is null, since it represents the family index of a 
 bundle with uniformly positive scalar curvature in the fibres, as explained in \cite[Section 2]{St}.
\end{proof}

\subsection{Splitting the twisted transfer.}

In order to carry out Stolz' program for the twisted case, one needs to show next that the induced map in (co)homology of any lift of the transfer (to the homotopy fibre of the Atiyah-Bott-Shapiro orientation) splits over $\underline{\mathcal A}$. Already here, however, extending Stolz' results producing such a split to the twisted situation, is surprisingly subtle, as evidenced by theorem \ref{nogo} below. We do have:

\begin{prop}\label{trasur}
 Any lift $\widehat{t_K}:\MSpinK \barwedge (S^8 \wedge {BG_+}) \rightarrow \hofib(\alpha_K)$ of the twisted transfer induces an injection in $mod\ 2$-cohomology and splits over $H^*(K)$.
\end{prop}

\begin{proof}
Oberserve that both domain and target are free $H^*(K)$-modules and the map is split injective on $H^*(K)$-indecomposables: First of all $H^*(\MSpinK)$ is free over $H^*(K)$ by \cite[Proposition 1.7]{MiMo}. 
It immediately follows that the domain is free as well. For the codomain, notice first that we have short exact sequences
$$0 \longrightarrow H^*(\koK) \stackrel{\alpha_K^*}\longrightarrow H^*(\MSpinK) \longrightarrow H^*(\hofib(\alpha_K)) \longrightarrow 0$$
since $\alpha_K$ is split injective in cohomology by \ref{genspli}. Since by \cite[Proposition 1.7]{MiMo} any system of representatives for the $H^*(K)$ indecomposables forms a basis, we see that we can actually choose the $H^*(K)$-basis of $H^*(\MSpinK)$ 
to contain a $H^*(K)$-basis of the image of $H^*(\koK)$, which gives the claim. \\
Alternatively, our description of $H^*(\MSpinK)$ in terms of $\underline{\mathcal A}$, $\underline{\mathcal A}/(\varphi(Sq^1), \varphi(Sq^2))$ and $\underline{\mathcal A}/\varphi(Sq^3)$ is compatible with $\alpha_K$ by construction. Since all three building blocks 
are obviously free over $H^*(K)$ on the appropriate Steenrod operations the claim follows. \\
Finally, taking the quotient by the $H^*(K)$-decomposables in the situation of the twisted transfer map exactly produces the untwisted situation, in which Stolz has shown the transfer to be split injective (over $\mathcal A$).
Calling such a split $s$, we arrive at the following diagram
$$\begin{xy}\xymatrix{
 H^*(\hofib(\alpha_K)) \ar@{->>}[d]          \ar[r]^-{\widehat{t_K}^*} & H^{*-8}(\MSpinK  \barwedge BG_+) \ar@{->>}[d]                     \\
 H^*(\hofib(    \alpha)) \ar@{..>}@/^1pc/[u] \ar[r]^-{\hat t^*}        & H^{*-8}(MSpin \wedge BG_+) \ar@{..>}@/^1pc/[l]^-s \ar@{..>}@/_1pc/[u] }
\end{xy}$$
where the upwards arrows just denote $\mathbb Z/2$-splits of the downwards maps. Map then the image of the split on the right to $H^*(\hofib(\alpha_K))$ using the other two splits and extend the resulting map $H^*(K)$-linearly; this is possible because of the freeness of the upper right corner over $H^*(K)$. We have thus produced a map 
\[{s_K}: H^{*-8}(\MSpinK  \barwedge BG_+) \longrightarrow H^*(\hofib(\alpha_K)),\]
such that $s_K \circ \widehat{t_K}^*$ induces the identity on indecomposables and therefore is an isomorphism.
The map $(s_K \circ \widehat{t_K}^*)^{-1} \circ s_K$ now is a split of $\widehat{t_K}^*$.
\end{proof}

\begin{rem}\label{aftertrasur}
If one could choose the vertical splits in the above square diagram to be $\mathcal A$-linear, the split $s_K$ constructed would be $\underline{\mathcal A}$-linear. 
This is, however, clearly impossible since for example $Sq^1$ acts trivially on the unique nontrivial class in $H^8(\hofib(\alpha))$ but nontrivially on the unique nonzero class in $H^8(\hofib(\alpha_K))$.
\end{rem}

To see that the situation in the twisted context, however, behaves somewhat differently from the untwisted one let us recall how Stolz produced a split of the untwisted transfer map in \cite{St}. He used the observation of Pengelley that there exists a certain split $r$ of $\alpha^*: H^*(ko) \rightarrow H^*(MSpin)$, which is a map of coalgebras. 
This split $r$ makes it possible to define the $H^*(ko)$-primitives $\overline{H^*(X)}$ in the cohomology of any $MSpin$-module spectrum $X$. These primitives form an $\mathcal A(1)$-submodule of $H^*(X)$ and Stolz showed the following structural statement.

\begin{uprop}[Corollary 5.5 \cite{St}]
For any $MSpin$-module spectrum $X$ with bounded below and degreewise finite cohomology the obvious map
$$H^*(X) \longleftarrow \mathcal A \otimes_{\mathcal A(1)} \overline{H^*(X)}$$
is an isomorphism, that is natural in $MSpin$-module maps.
\end{uprop}

In order to split the transfer it therefore suffices to split its induced map on $H^*(ko)$-primitives $\mathcal A(1)$-linearly, which Stolz does by direct computation, using the fact that $\overline{H^*(MSpin)} \cong M$, which in turn he seems to derive from the Anderson-Brown-Peterson splitting (see, however, \ref{gap} for a discussion of this). \\
One can now hope that one might use some form of $H^*(\koK)$-primitives in the general situation, reducing the splitting of the transfer to the very same $\mathcal A(1)$-modules Stolz considered. This hope, however, is in vain, as we shall explain next.

We saw in remark \ref{honmul} that $H_*(\MSpinK)$ and similarly $H_*(\koK)$ carry multiplicative structures, agreeing with that of $H_*(MO)$ in the former case. 
Dualising our cohomological calculations from above one can perform some low-dimensional calculations and find:

\begin{prop}[Proposition I.6.4.1 \cite{He}]\label{noringsplit}
The homology of $\koK$ contains an element $0 \neq y \in H_2(\koK)$ with $y^4 = 0$. In particular, there can be no surjective homomorphism $H^*(\MSpinK) \rightarrow H^*(\koK)$ of coalgebras.
\end{prop}

\begin{proof}[Sketch of proof]
This is a straight forward computation given that the Postnikov-section $\koK \rightarrow K \times H\mathbb Z/2$ induces an injection in homology with known image by \ref{632}; $y$ is the preimage of $x_1 \otimes \zeta_1 + x_2 \otimes 1$, where $x_1$ is dual to $\iota_1 \in H^1(K)$, $x_2$ to $\iota_1^2$ and $\zeta_1$ to $Sq^1$. \\ 
The implication that there can be no coalgebra-split of the twisted Atiyah-Bott-Shapiro orientation $H^*(\koK) \rightarrow H^*(\MSpinK)$ can in fact also be inferred from the next proposition, since the existence of a split would allow a generalisation of Stolz' proposition above contradicting \ref{nogo}.
\end{proof}

One therefore has to make do with the $H^*(ko)$-primitives, which for a parametrised $MSpin$-module spectrum $X$ form an $\underline{\mathcal A(1)}$-submodule of $H^*(X)$. With the same consideration as in \ref{trasur} it is easy to see, that the transfer still splits over $H^*(K)$ on the module of primitives and it is also readily seen that both domain and codomain are free over $\mathcal A(1)$. Since free $\mathcal A(1)$-modules are injective we find:

\begin{prop}
The induced map
$$\overline{H^*(\hofib(\alpha_K))} \xrightarrow{\widehat{t_K}^*} \overline{H^{*-8}(\MSpinK  \barwedge BG_+)}$$
splits separately over both $\mathcal A(1)$ and $H^*(K)$.
\end{prop}

With these considerations in place it seems natural to return to the square from the proof of \ref{trasur}, and take $H^*(ko)$-primitives everywhere to try and produce a split over $\underline{\mathcal A(1)}$:
$$\begin{xy}\xymatrix{
 \overline{H^*(\hofib(\alpha_K))} \ar@{->>}[d]          \ar[r]^-{\widehat{t_K}^*} & \overline{H^{*-8}(\MSpinK  \barwedge BG_+)} \ar@{->>}[d]                     \\
 \overline{H^*(\hofib(    \alpha))} \ar@{..>}@/^1pc/[u] \ar[r]^-{\widehat t^*}        & \overline{H^{*-8}(MSpin \wedge BG_+)} \ar@{..>}@/^1pc/[l]^-s \ar@{..>}@/_1pc/[u] }
\end{xy}$$
where a comparison of the Poincar\'e-series reveals that the downward maps are again isomorphisms after passing to $H^*(K)$-indecomposables in the upper row. 
In the proof of \ref{trasur} we only used that the upwards maps can be chosen $\mathbb Z/2$-linear and arguing similarly to remark \ref{aftertrasur} we see that the transfer would split over $\underline{\mathcal A}$ if they could be chosen $\mathcal A(1)$-linear (which of course they cannot).
It seems, however, a reasonable guess that they may be chosen $\varphi$-linear, i.e. such that $s(a \cdot m) = \varphi(a) \cdot s(m)$, (which would still yield an $\underline{\mathcal A}$-linear split of the transfer!) for the following reason:
From Stolz' result above we see
\begin{align*}
 H^*(\MSpinK) &\cong \underline{\mathcal A} \otimes_{\underline{\mathcal A(1)}} \overline{H^*(\MSpinK)}
\intertext{as modules over the twisted Steenrod algebra, whereas our generalisation of the Anderson-Brown-Peterson directly yields}
 H^*(\MSpinK) &\cong \underline{\mathcal A}_\varphi \otimes_{\mathcal A(1)} M \\
           &\cong \underline{\mathcal A} \otimes_{\underline{\mathcal A(1)}} \left[\underline{\mathcal A(1)}_\varphi \otimes_{\mathcal A(1)} M \right]
\end{align*}
where $M$ denotes the Anderson-Brown-Peterson-module (i.e. a certain direct sum of $\mathbb Z/2$'s, jokers and a free $\mathcal A(1)$-module) and similarly for the derivatives of $H^*(\MSpinK)$ that occur in the analysis of the transfer. This strongly suggests 
$$\underline{\mathcal A(1)}_\varphi \otimes_{\mathcal A(1)} M \cong \overline{H^*(\MSpinK)}$$
and the left hand side in particular admits a $\varphi$-linear split of the projection to its $H^*(K)$-indecomposables. Indeed, $\varphi$-linear splits of the projection 
$$\overline{H^*(\MSpinK)} \rightarrow \overline{H^*(MSpin)}$$
 are easily seen to be in one-to-one correspondence with isomorphisms as just predicted. But alas, this hope is dashed by the following very surprising result:

\begin{thm}[Theorem II.4.2.1 \cite{He}]\label{nogo}
There is no isomorphism $$\underline{\mathcal A(1)}_\varphi \otimes_{\mathcal A(1)} \overline{H^*(MSpin)} \cong \overline{H^*(\MSpinK)}$$
of $\underline{A(1)}$-modules.
\end{thm}

\begin{proof}[Sketch of proof]
After switching to homology and applying the inverse of the Thom isomorphism the existence of an isomorphism as in the statement is equivalent to
$$H_*(K) \otimes \overline{H_*(BSpin)} \cong \overline{H_*(BO)}$$
as $\underline{\mathcal A(1)}$-modules, where the action of $\underline{\mathcal A(1)}$ on $\overline{H_*(BSpin)}$ factors through $\mathcal A(1)$; the overline here indicates $H_*(ko)$-indecomposables. 

Under the additional assumption that the isomorphism $\underline{\mathcal A(1)}_\varphi \otimes_{\mathcal A(1)} \overline{H^*(MSpin)} \cong \overline{H^*(\MSpinK)}$ commutes with the projections to $\overline{H^*(MSpin)}$, the isomorphism in homology will commute with the inclusions of $\overline{H_*(BSpin)}$. This would in particular yield a splitting of the inclusion $\overline{H_*(BSpin)} \rightarrow \overline{H_*(BO)}$ over $\mathcal A(1)$, which can be excluded by observing that the induced map on $\mathcal A(1)$-indecomposables fails to be injective in degree $16$. In the notation of \cite{He, St2} (originating from \cite{GiPeRa}), the element $x_2^8 \in \overline{H_{16}(BSpin)}$ is indecomposable, but becomes decomposable already in $\overline{H_{16}(BSO)}$, since there $Sq_2(x_{18}) = x_2^8$. 

To remove the additional assumption it is shown in \cite[Theorem II.4.2.1]{He} that the $H^*(K(\mathbb Z/2,1))$-indecomposables of the two modules in question are not isomorphic as $\mathcal A(1)$-modules; applying the Thom isomorphism  the duals of these indecomposables are given by
\[H_*(K(\mathbb Z/2,2)) \otimes \overline{H_*(BSpin)} \quad \text{and} \quad \overline{H_{*}(BSO)}\] 
and the image of $x_{18} \in H_{18}(BSO)$ in $H_{18}(K(\mathbb Z/2,2))$ yields the element $x_{18} \otimes 1$ on the left hand side. Direct computations now show that it is a $Sq_2$- and $Q_1$-cycle, whose $Sq_1$ is not a $Q_1$-boundary, whereas no such element exists in $\overline{H_{18}(BSO)}$.
\end{proof}

\begin{rem}\label{gap}
With the above theorem in mind, one may wonder how Stolz derives the isomorphism $\overline{H^*(MSpin)} \cong M$ from the isomorphism 
\[\mathcal A \otimes_{\mathcal A(1)} \overline{H^*(MSpin)} \cong \mathcal A \otimes_{\mathcal A(1)} M\] 
given by the Anderson-Brown-Peterson-splitting and \cite[Corollary 5.5]{St}. The existence of said isomorphism is stated in \cite[Corollary 6.4]{St} without proof and used extensively in \cite{St}, but in fact does not directly follow from the context. However, in contrast to the twisted situation, the result is true as claimed; the first author provided a proof in \cite[Section 3.1]{He}.
\end{rem}

One may now wonder what these negative results leave one with. In forthcoming work we will demonstrate that from the isomorphisms
$$\underline{\mathcal A} \otimes_{\underline{\mathcal A(1)}} \overline{H^*(\MSpinK)} \cong H^*(\MSpinK) \cong \underline{\mathcal A} \otimes_{\underline{\mathcal A(1)}} \left[\underline{\mathcal A(1)}_\varphi \otimes_{\mathcal A(1)} M \right]$$
one can construct a filtration on $\overline{H^*(\MSpinK)}$ of $\underline{\mathcal A(1)}$-coalgebras, whose associated graded $\underline{\mathcal A(1)}$-module is isomorphic to $\underline{\mathcal A(1)}_\varphi \otimes_{\mathcal A(1)} M$ and compatible with the map induced by the transfer. Therefore the strategy indicated above leads to a
$\underline{\mathcal A(1)}$-linear split $s$ of the induced transfer on the associated graded modules
$$\xymatrix{E\big(\overline{H^*(\hofib(\alpha_K))}\big) \ar[r]^-{E(\widehat{t_K}^*)} & \ar@/_1.5pc/[l]_{s_K} E\big(\overline{H^{*-8}(\MSpinK  \barwedge BG_+)}\big)}$$
Associated to this filtration is a spectral sequence, which determines whether such an $s_K$ can be descended to a map on the actual modules. The analysis of that spectral sequence and then of the parametrised Adams spectral sequence, which controls whether a split in homology comes from a map of the underlying spectra is ongoing work of the first author and Stephan Stolz. All in all we therefore hope to return to the analysis of the twisted $\mathbb HP^2$-transfer map in future work.

\appendix

\section{A calculation from Section \ref{calc}}

We here finish the purely calculational proofs of lemmata \ref{twiemb} and \ref{twimod}. The former of which reads as follows:

\begin{ulem}[\ref{twiemb}]
The stipulation
\begin{align*}
Sq^1 &\longmapsto 1 \otimes Sq^1 + \iota_1 \otimes 1 \\
Sq^2 &\longmapsto 1 \otimes Sq^2 + \iota_1 \otimes Sq^1 + \iota_1^2 \otimes 1 + \iota_2 \otimes 1
\end{align*}
specifies a unique morphism of Hopf algebras $\varphi: \mathcal A(1) \longrightarrow \underline{\mathcal A}$.
\end{ulem}

\begin{proof}
As mentioned in \ref{twiemb} the following two identities 
\begin{align*}
    \varphi(Sq^1)^2 &= 0 \\
    \varphi(Sq^2)^2 &=  \varphi(Sq^1)\varphi(Sq^2)\varphi(Sq^1)
\end{align*} 
will show both multiplicativity of $\varphi$ and that it is even well defined.
\begin{align*} 
\varphi(Sq^1)^2 &= (1 \otimes Sq^1 + \iota_1 \otimes 1)^2 \\
                &= (1 \otimes Sq^1)^2 + (1 \otimes Sq^1)(\iota_1 \otimes 1) + (\iota_1 \otimes 1)(1 \otimes Sq^1) + (\iota_1 \otimes 1)^2 \\
                &= \underbrace{1 \otimes (Sq^1)^2}_{=0} + \big[\underbrace{Sq^1(\iota_1) \otimes 1}_{=\iota_1^2 \otimes 1} + \iota_1 \otimes Sq^1\big] + \iota_1 \otimes Sq^1 + \iota_1^2 \otimes 1 \\
		&= 0
\end{align*}

\begin{align*}
&\varphi(Sq^1)\varphi(Sq^2)\varphi(Sq^1) \\
        &= \ (1 \otimes Sq^1 + \iota_1 \otimes 1)(1 \otimes Sq^2 + \iota_1 \otimes Sq^1 + \iota_1^2 \otimes 1 + \iota_2 \otimes 1)(1 \otimes Sq^1 + \iota_1 \otimes 1) \\
        &= \Big((1 \otimes Sq^1)(1 \otimes Sq^2) + (1 \otimes Sq^1)(\iota_1 \otimes Sq^1) + (1 \otimes Sq^1)(\iota_1^2 \otimes 1) \\
         &\ \quad + (1 \otimes Sq^1)(\iota_2 \otimes 1) + (\iota_1 \otimes 1)(1 \otimes Sq^2) + (\iota_1 \otimes 1)(\iota_1 \otimes Sq^1) \\
         &\ \quad + (\iota_1 \otimes 1)(\iota_1^2 \otimes 1) + (\iota_1 \otimes 1)(\iota_2 \otimes 1)\Big)(1 \otimes Sq^1 + \iota_1 \otimes 1) \\
        &= \Big(1 \otimes Sq^1Sq^2 + \big[\underbrace{Sq^1(\iota_1) \otimes Sq^1}_{= \cancel{\iota_1^2 \otimes Sq^1}} + \underbrace{\iota_1 \otimes (Sq^1)^2}_{=0}\big] + \big[\underbrace{Sq^1(\iota_1^2) \otimes 1}_{=0} + \  \cancel{\iota_1^2 \otimes Sq^1}\big] \\
         &\ \quad + \left[Sq^1(\iota_2) \otimes 1 + \iota_2 \otimes Sq^1\right] + \iota_1 \otimes Sq^2 + \iota_1^2 \otimes Sq^1 \\
         &\ \quad + \iota_1^3 \otimes 1 + \iota_1\iota_2 \otimes 1\Big)(1 \otimes Sq^1 + \iota_1 \otimes 1) \\
        &= \ (1 \otimes Sq^1Sq^2)(1 \otimes Sq^1) + (Sq^1(\iota_2) \otimes 1)(1 \otimes Sq^1) + (\iota_2 \otimes Sq^1)(1 \otimes Sq^1) \\
         &\ \quad + (\iota_1 \otimes Sq^2)(1 \otimes Sq^1) + (\iota_1^2 \otimes Sq^1)(1 \otimes Sq^1) + (\iota_1^3 \otimes 1)(1 \otimes Sq^1) \\
         &\ \quad + (\iota_1\iota_2 \otimes 1)(1 \otimes Sq^1) + (1 \otimes Sq^1Sq^2)(\iota_1 \otimes 1) + (Sq^1(\iota_2) \otimes 1)(\iota_1 \otimes 1) \\
         &\ \quad + (\iota_2 \otimes Sq^1)(\iota_1 \otimes 1) + (\iota_1 \otimes Sq^2)(\iota_1 \otimes 1) + (\iota_1^2 \otimes Sq^1)(\iota_1 \otimes 1) \\
         &\ \quad + (\iota_1^3 \otimes 1)(\iota_1 \otimes 1) + (\iota_1\iota_2 \otimes 1)(\iota_1 \otimes 1) \\			
        &= \ \underbrace{1 \otimes Sq^1Sq^2Sq^1}_{= 1 \otimes Sq^3Sq^1} + Sq^1(\iota_2) \otimes Sq^1 + \underbrace{\iota_2 \otimes (Sq^1)^2}_{=0} + \ \iota_1 \otimes Sq^2Sq^1 + \underbrace{\iota_1^2 \otimes (Sq^1)^2}_{=0} \\
         &\ \quad + \iota_1^3 \otimes Sq^1 + \cancel{\iota_1\iota_2 \otimes Sq^1} + \big[\underbrace{Sq^1Sq^2(\iota_1) \otimes 1}_{=0} + \underbrace{Sq^1(\iota_1) \otimes Sq^2}_{\cancel{\iota_1^2 \otimes Sq^2}} + \underbrace{Sq^2(\iota_1) \otimes Sq^1}_{=0} \\
         &\ \quad + \underbrace{\iota_1 \otimes Sq^1Sq^2}_{\iota_1 \otimes Sq^3}\big] + Sq^1(\iota_2)\iota_1 \otimes 1 + \big[\underbrace{\iota_2Sq^1(\iota_1) \otimes 1}_{\cancel{\iota_1^2\iota_2 \otimes 1}} + \cancel{\iota_2\iota_1 \otimes Sq^1}\big] \\
         &\ \quad + \big[\underbrace{\iota_1Sq^2(\iota_1) \otimes 1}_{=0} + \underbrace{\iota_1Sq^1(\iota_1) \otimes Sq^1}_{\cancel{\iota_1^3 \otimes Sq^1}} + \cancel{\iota_1^2 \otimes Sq^2}\big] + \big[\underbrace{\iota_1^2Sq^1(\iota_1) \otimes 1}_{\cancel{\iota_1^4 \otimes 1}} + \cancel{\iota_1^3 \otimes Sq^1}\big] \\
         &\ \quad + \cancel{\iota_1^4 \otimes 1} + \cancel{\iota_1^2\iota_2 \otimes 1}\\
        &= \ 1 \otimes Sq^3Sq^1 + Sq^1(\iota_2) \otimes Sq^1 + \iota_1 \otimes Sq^2Sq^1 + \iota_1^3 \otimes Sq^1 + \iota_1 \otimes Sq^3 \\
         &\ \quad + \iota_1Sq^1(\iota_2) \otimes 1  
\end{align*}

\begin{align*}
&\varphi(Sq^2)\varphi(Sq^2) \\
        &= \ (1 \otimes Sq^2 + \iota_1 \otimes Sq^1 + \iota_1^2 \otimes 1 + \iota_2 \otimes 1)(1 \otimes Sq^2 + \iota_1 \otimes Sq^1 + \iota_1^2 \otimes 1 + \iota_2 \otimes 1) \\
        &= \ (1 \otimes Sq^2)^2 + (\iota_1 \otimes Sq^1)(1 \otimes Sq^2) + (\iota_1^2 \otimes 1)(1 \otimes Sq^2) + (\iota_2 \otimes 1)(1 \otimes Sq^2) \\
         &\ \quad + (1 \otimes Sq^2)(\iota_1 \otimes Sq^1) + (\iota_1 \otimes Sq^1)^2 + (\iota_1^2 \otimes 1)(\iota_1 \otimes Sq^1) + (\iota_2 \otimes 1)(\iota_1 \otimes Sq^1) \\
         &\ \quad + (1 \otimes Sq^2)(\iota_1^2 \otimes 1) + (\iota_1 \otimes Sq^1)(\iota_1^2 \otimes 1) + (\iota_1^2 \otimes 1)^2 + (\iota_2 \otimes 1)(\iota_1^2 \otimes 1) \\
         &\ \quad + (1 \otimes Sq^2)(\iota_2 \otimes 1) + (\iota_1 \otimes Sq^1)(\iota_2 \otimes 1) + (\iota_1^2 \otimes 1)(\iota_2 \otimes 1) + (\iota_2 \otimes 1)^2 \\
        &= \ \underbrace{1 \otimes (Sq^2)^2}_{= 1 \otimes Sq^3Sq^1} + \underbrace{\iota_1 \otimes Sq^1Sq^2}_{= \iota_1 \otimes Sq^3} + \cancel{\iota_1^2 \otimes Sq^2} + \cancel{\iota_2 \otimes Sq^2} + \big[\underbrace{Sq^2(\iota_1) \otimes Sq^1}_{=0} \\
         &\ \quad + \underbrace{Sq^1(\iota_1) \otimes (Sq^1)^2}_{=0} + \iota_1 \otimes Sq^2Sq^1\big] + \big[\underbrace{\iota_1Sq^1(\iota_1) \otimes Sq^1}_{= \iota_1^3 \otimes Sq^1} + \underbrace{\iota_1^2 \otimes (Sq^1)^2}_{=0} \big] \\
         &\ \quad + \cancel{\iota_1^3 \otimes Sq^1} + \cancel{\iota_2\iota_1 \otimes Sq^1} + \big[\underbrace{Sq^2(\iota_1^2) \otimes 1}_{= \cancel{\iota_1^4 \otimes 1}} + \underbrace{Sq^1(\iota_1^2) \otimes Sq^1}_{=0} + \cancel{\iota_1^2 \otimes Sq^2}\big] \\
         &\ \quad + \big[\underbrace{\iota_1Sq^1(\iota_1^2) \otimes 1}_{=0} + \cancel{\iota_1^3 \otimes Sq^1}\big] + \cancel{\iota_1^4 \otimes 1} + \cancel{\iota_2\iota_1^2 \otimes 1} + \big[\underbrace{Sq^2(\iota_2) \otimes 1}_{= \cancel{\iota_2^2 \otimes 1}} \\
         &\ \quad + Sq^1(\iota_2) \otimes Sq^1 + \cancel{\iota_2 \otimes Sq^2}\big] + \big[\iota_1Sq^1(\iota_2) \otimes 1 + \cancel{\iota_1\iota_2 \otimes Sq^1}\big] + \cancel{\iota_1^2\iota_2 \otimes 1} \\
         &\ \quad + \cancel{\iota_2^2 \otimes 1} \\   
        &= \ 1 \otimes Sq^3Sq^1 + \iota_1 \otimes Sq^3 + \iota_1 \otimes Sq^2Sq^1 + \iota_1^3 \otimes Sq^1 + Sq^1(\iota_2) \otimes Sq^1 \\ 
         &\ \quad + \iota_1Sq^1(\iota_2) \otimes 1
\end{align*}

These two outcomes obviously agree.

For the comultiplicativity recall that the inclusion maps induce a coalgebra isomorphism $H^*(K) \otimes \mathcal A \rightarrow \underline{\mathcal A}$. The coalgebra structure of $H^*(K)$ is of course determined by \ref{twimul}:
\begin{align*}
\Delta(\iota_1) &= \iota_1 \otimes 1 + 1 \otimes \iota_1 \\
\Delta(\iota_2) &= \iota_2 \otimes 1 + \iota_1 \otimes \iota_1 + 1\otimes \iota_2
\end{align*}
With these data it is trivial to verify $\varphi(\Delta(Sq^{1,2})) = \Delta(\varphi(Sq^{1,2}))$. By multiplicativity of $\varphi$ and the coproducts we are done.
\end{proof}

\begin{ulem}[\ref{twimod}]
Extending $\varphi$ by the identity on $H^*(K)$ produces a Hopf algebra automorphism of $\underline{\mathcal A(1)}$. 
\end{ulem}

\begin{proof}
Putting $\varphi(k \otimes \alpha) = (k \otimes 1)\varphi(\alpha)$ just leaves the verificaton of the multiplicativity. To this end calculate
\begin{align*}
 \varphi(k \otimes \alpha)\varphi(l \otimes \beta) &= (k \otimes 1)\varphi(\alpha)(l \otimes 1)\varphi(\beta) \\
 \varphi((k \otimes \alpha)(l \otimes \beta))      &= \sum \varphi(k \cdot \alpha'(l) \otimes \alpha'' \circ \beta) \\
                                                   &= \sum (k\alpha'(l) \otimes 1)\varphi(\alpha'' \circ \beta) \\
                                                   &= \sum (k \otimes 1)(\alpha'(l) \otimes 1)\varphi(\alpha'')\varphi(\beta)
\end{align*}
Thus it suffices to show that $\varphi(\alpha)(l \otimes 1) = \sum (\alpha'(l) \otimes 1)\varphi(\alpha'')$ for all $\alpha \in \mathcal A, l \in H^*(K)$. Now observe observe that this statement propagates under multiplying $\alpha$'s and then calculate once more:
\begin{align*}
 \varphi(Sq^1)(l \otimes 1) =&\ (\iota_1 \otimes 1 + 1 \otimes Sq^1)(l \otimes 1) \\
                            =&\ (\iota_1l \otimes 1) + (Sq^1(l) \otimes 1) + (l \otimes Sq^1) \\
                            =&\ (Sq^1(l) \otimes 1) + (l \otimes 1)(\iota_1 \otimes 1 + 1 \otimes Sq^1)\\
 \varphi(Sq^2)(l \otimes 1) =&\ (\iota_2 \otimes 1 + \iota_1^2 \otimes 1 + \iota_1 \otimes Sq^1 + 1 \otimes Sq^2)(l \otimes 1) \\
                            =&\ \iota_2l \otimes 1 + \iota_1^2 \otimes 1 + \iota_1Sq^1(l) \otimes 1 + \iota_1l\otimes Sq^1 + Sq^2(l) \otimes 1 \\
                            +&  Sq^1(l) \otimes Sq^1 + l \otimes Sq^2 \\
                            =&\ (Sq^2(l) \otimes 1) + (Sq^1(l) \otimes 1)(\iota_1 \otimes Sq^1) \\
                            +&  (l \otimes 1)(\iota_2 \otimes 1 + \iota_1^2 \otimes 1 + \iota_1 \otimes Sq^1 + 1 \otimes Sq^2)
\end{align*}
Of course this entire calculation is nothing but the verification of compatibility of $\varphi$ and $id_{H^*(K)}$ in the universal property of a semidirect product of Hopf algebras.
\end{proof}

\section{A proof of the cobordism invariance theorem}

Theorem \ref{bordi} is a direct consequence of the following result and the surgery theorem of Gromov and Lawson:
\begin{uprop} \label{codest}
Let $n\geq 5$ and $M$ be a smooth, closed $n$-manifold with a $\xi$-structure such that the underlying map $M \rightarrow B$ is a $k$-equivalence 
where $k < \frac n 2$. Let furthermore $N$ be another $n$-dimensional, closed $\xi$-manifold that is $\xi$-bordant to $M$. 
Then $M$ arises from $N$ by a sequence of surgeries of codimension $>k$.
\end{uprop}

As mentioned in the main text, this statement can be found in various places in the literature, yet there seems to be no full proof of it. The $Spin$ case is explicitely proved in \cite{Ro0} and we closely follow its strategy. The proof in \cite{Kr} contains a gap as we explain in the remark below.

\begin{proof}
One can first perform surgery on the interior of a given cobordism $W$ to make the map $W \rightarrow B$ into a $k$-equivalence $W' \rightarrow B$ as explained for example in \cite[Theorem 3.61]{Lu} or \cite[Chapter 1]{Wa}; the assumption that the target be a finite complex with Poincar\'e duality may be relaxed to it being of type $F_k$, i.e. admitting a $k$-equivalence from a finite complex: This still implies the necessary finiteness of the relevant relative homotopy groups by the Hurewicz theorem and Poincar\'e duality is not used at all below the middle dimension. \\
To proceed pick a finite generating system of $\pi_k(W',M)$ (this is again possible by Hurewicz' theorem). The map $\pi_k(W') \rightarrow \pi_k(W',M)$ is surjective as by construction of $W'$ the inclusion $M \rightarrow W'$ induces an isomorphism on the first $k-1$ homotopy groups.
$$\xymatrix@-1pc{\cdots \ar[rr]&& \pi_k(M) \ar[rr] \ar@{->>}[rrd] && \pi_k(W') \ar@{->>}[d] \ar@{->>}[rr] &&\pi_k(W',M) \ar[rr] && \dots \\
                         &&                && \pi_k(B)               && &&}$$
 We now claim that lifts of the generators of $\pi_k(W',M)$ can be chosen in the kernel of $\pi_k(W) \rightarrow \pi_k(B)$, whence they have trivial stable normal bundle and thus trivial normal bundle since $k < n/2$. Indeed for an arbitrary lift the failure to lie in the kernel can be corrected by an element in $\pi_k(M)$ by assumption on $M$. \\
The cobordism $W''$ arising from this surgery now has $M \rightarrow W''$ a $k$-equivalence and Smale's theorem yields the claim.
\end{proof}

\begin{urem}
Kreck's proof of our corollary \ref{bordi}, which corresponds to \cite[Theorem 1]{Kr}, as mentioned, is not quite correct. In the proof it is claimed that the map $W \rightarrow B$ above can be surgered into a $k+1 = 3$ equivalence by an application of \cite[Proposition 4]{Kr}. This is not true in general and also not implied by the quoted proposition: For example any spin manifold $M$ is equipped with a structure map to $BSpin \times B\pi_1(M)$ and if it can be surgered into a $3$-equivalence it follows immediately that $\pi_1(M)$ satisfies Serre's finiteness property $F_3$, namely that there exists a finite complex admitting a $3$-equivalence to its classifying space. However, any finitely presented group ($= F_2$) can occur as the fundamental group of a compact spin manifold, since for a spin manifold every map $M \rightarrow B\pi$ gives a map $M \rightarrow BSpin \times B\pi$, which may be surgered into a $2$-equivalence as explained in the proof of the above proposition or indeed is implied by \cite[Proposition 4]{Kr}. 

An application of \cite[Proposition 4]{Kr} could indeed replace the first paragraph of the proof above. We, however, deliberately refrained from citing \cite[Proposition 4]{Kr} in the proof of the above proposition, since its proof unfortunately contains another erroneous assertion about finite generation of certain homotopy groups.
\end{urem}

\section{A comparison to homotopical models of twisted $K$-theory}

Let us finally explain the comparison of our construction of twists for $MSpin$ and $KO$ with the one arising in the infinity-categorical setting of \cite{AnBlGeHoRe}. For the simply connected case, i.e.\ twists of $KO$ by maps into $K(\mathbb Z/2,2)$ such a comparison was first attempted by Antieau, Gepner and Gomez in \cite{AGG}. They showed that there are only two homotopy classes of maps \[K(\mathbb Z/2,2) \rightarrow BGl_1(KO)\] and similarly a $\mathbb Z$-worth of maps $K(\mathbb Z,3) \rightarrow BGl_1(KU)$. Their proof that a geometric construction indeed produces such a map, however, contains a mistake as we will explain at the end of this section. Our comparison proceeds quite differently from theirs and is stronger in several respects: It provides a comparison of the two ways of twisting as $E_\infty$ maps to $BGl_1(KO)$, works for the entire space $BO[2]$ and not just its $2$-connective cover $K(\mathbb Z/2,2)$ and in fact works at the level of $Spin$-cobordism and not just $K$-theory. It also works equally well in the complex case for the two maps $B(O\sslash Spin^c) \rightarrow BGl_1(MSpin^c) \rightarrow BGl_1(KU)$, where $\sslash$ denotes the homotopy orbit construction; this is analogous to the real situation, since $BO[2] \simeq B(O\sslash Spin)$. To keep the exposition brief we shall refrain from spelling this out and we will assume familiarity with both the methods of \cite{AnBlGeHoRe} and the technology of $\mathcal I$-spaces from \cite{SS, SS2}. 

Let us briefly sketch the homotopical method for obtaining twisted homology theories. One starts with an $E_\infty$ ring spectrum $R$, and forms the universal $R$-line bundle over $BGl_1(R)$ or even more generally over $Pic(R)$, of which $BGl_1(R)$ is the unit component; here $Gl_1(R)$ denotes the $E_\infty$ space of (derived) $R$-linear self equivalences of $R$. This universal bundle can now be pulled back along maps $K \rightarrow BGl_1(R)$ to more manageable spaces. In the case of the Thom spectrum $M\theta$ of an $E_\infty$ map $\theta: B \rightarrow BO$ (or more generally to $B\mathcal G$, the classifying space of spherical fibrations) the following construction suggests itself (we learned of it from M. Ando, compare \cite{AnBlGe}): Letting $F_\theta$ denote the homotopy fibre of $\theta$ over the unit in $BO$, one obtains $E_\infty$ maps 
$$\Sigma^\infty_+ F_\theta \rightarrow M\theta \rightarrow MO$$
by applying Thom spectra. The first map in this sequence adjoins to a map $F_\theta \rightarrow Gl_1(M\theta)$ which may be delooped to yield an $E_\infty$ map
$$BF_\theta \rightarrow BGl_1(M\theta)$$
In the case of the Thom spectrum $MSpin$, i.e.\ $B = BSpin$ and $\theta$ the canonical projection, this yields $F_{Spin} = O\sslash Spin \simeq O[1]$. From $BF \simeq BO[2]$ we then on the one hand obtain a map 
$$BO[2] \rightarrow BGl_1(MSpin).$$ 
On the other hand the action of $P\mathcal{O}(\ell^2)$ on the `free rank $1$ $MSpin$-module' $MSpin'$ we produced in the main text yields a map $P\mathcal{O}(\ell^2) \rightarrow \mathrm{hAut}_{MSpin}(MSpin')$ into the $MSpin$-linear homotopy self-equivalences of $MSpin'$, and thus a derived $E_1$ map $P\mathcal{O}(\ell^2) \rightarrow Gl_1(MSpin)$. This deloops to a map 
$$BP\mathcal{O}(\ell^2) \rightarrow BGl_1(MSpin)$$
and similarly we obtain
$$BP\mathcal O(\ell^2) \rightarrow BGl_1(KO)$$
from the $KO$-module $KO'$.
We will endow these maps with $E_\infty$ refinements below. The second map fits into the framework of \cite{AGG} upon taking the $2$-connective cover in the source and thus the composite
\[K(\mathbb Z/2,2) \rightarrow BP\mathcal{O}(\ell^2) \rightarrow BGl_1(KO)\]
is easily checked to agree with the composite
\[K(\mathbb Z/2,2) \rightarrow BO[2] \rightarrow BGl_1(MSpin) \rightarrow BGl_1(KO)\] 
(when not considering $E_\infty$ structures). We shall proceed in a more direct fashion: We established in the proof of lemma \ref{twimul} an equivalence $j \colon BO[2] \rightarrow BP\mathcal{O}(\ell^2)$ and we claim:

\begin{uthm}
The diagram
$$\xymatrix{BO[2] \ar[r] \ar[d]_j & BGl_1(MSpin) \ar[d]^{\alpha_*} \\
            BP\mathcal{O}(\ell^2) \ar[ur] \ar[r] & BGl_1(KO)}$$
containing the two constructions outlined above is homotopy commutative as a diagram of $E_\infty$ spaces and maps. The same statement holds for
$$\xymatrix{B(O\sslash Spin^c) \ar[r] \ar[d]_j & BGl_1(MSpin^c) \ar[d]^{\alpha^c_*} \\
            BP\mathcal{U}(\ell^2) \ar[ur] \ar[r] & BGl_1(KU).}$$
\end{uthm}

The composition of the top horizontal and right vertical map is the usual way of constructing twists of $KO$-theory homotopically as explained above, whereas the lower horizontal map is the usual way of doing so geometrically, e.g. in \cite{AtSe}. 

Let us note right away that (at the level of spaces) the lower triangle commutes essentially by definition of the $P\mathcal O(\ell^2)$-action on $KO'$ as it makes the map $\alpha'\colon MSpin' \rightarrow KO'$ equivariant. We will proceed by producing a better model of the diagonal map using the technology of $\mathcal I$-spaces developed by Sagave and Schlichtkrull in \cite{SS}. This will give the desired $E_\infty$ structure to the diagonal and upgrade the commutativity of the lower triangle to include $E_\infty$ structures. The homotopical way of obtaining twists can also be modelled in the category of $\mathcal I$-spaces in an obvious fashion and we shall see that in this model the upper triangle also commutes strictly. 

We will start by giving these two $\mathcal I$-space models and checking the commutativity. More (and completely independent) work will then be required to show that this construction agrees up to homotopy with the diagonal map arising from our actions of $P\mathcal O(\ell^2)$ on $MSpin'$.

\begin{urem}
As mentioned in Remark \ref{remark}, based on the observations of this appendix a more refined category of parametrised spectra is constructed in joint work of the first author with Sagave and Schlichtkrull \cite{HSS19}. In this set-up the construction of $MSpin_K$ and $KO_K$ can be carried out without first passing to the primed free rank 1 modules and in a second paper \cite{HS19} the results of this appendix are then strengthened to hold at the level of twisted cohomology, rather than just at the level of twists.
\end{urem}

To proceed recall then the construction of $Gl_1(R)_\bullet$ for $R$ a symmetric ring spectrum due to Schlichtkrull: $Gl_1(R)_\bullet$ is the sub-$\mathcal I$-space of $$\Omega^\infty(R)_\bullet: k \longmapsto \Omega^k R_k$$ consisting of those components mapping to units in $\pi_0(R)$. The multiplication on $R$ gives it an $\mathcal I$-space monoid structure, i.e.\ it is a monoid under convolution (denoted by $\boxtimes$) in $\mathcal I$-spaces. If $R$ is commutative, the same holds for $Gl_1(R)$. We also immediately note that the spaces $P\mathcal{O}_n$ naturally form an $\mathcal I$-group. In fact, using the structure maps introduced in sections \ref{prelim} and \ref{twispi} they form an Eckman-Hilton-$\mathcal I$-group as defined by Dardalat and Pennig in \cite{DP}, i.e.~a commutative monoid under convolution in $\mathcal I$-groups. Dardalat and Pennig go on to show that an action of an Eckmann-Hilton-$\mathcal I$-group $G_\bullet$ on a symmetric ring spectrum $R$ gives rise to a map of commutative $\mathcal I$-space-monoids \[DP \colon G_\bullet \rightarrow Gl_1(R)_\bullet;\] such an action consists of maps 
$$G_n \times R_n \longrightarrow R_n$$
satisfying certain axioms which we recall (and slightly generalise) below. Their map $DP$ is then given by 
$$g \longmapsto \big(s \longmapsto g \cdot u_n(s)\big)$$ where $u_n: S^n \rightarrow R_n$ is part of the unit of $R$. In our case the left multiplication of $P\mathcal O_n$ on $MSpin_n$ yields a map $$P\mathcal{O}_\bullet \longrightarrow Gl_1(MSpin)_\bullet.$$
Recall also that an $\mathcal I$-space has a realisation (its homotopy colimit), and the arising functor from $\mathcal I$-spaces to spaces is lax monoidal, though not strictly but rather only coherently symmetric, with respect to convolution in $\mathcal I$-spaces and the cartesian product in spaces. Commutative $\mathcal I$-space monoids therefore realise to $E_\infty$ spaces (this is explicitely explained in \cite[Proposition 6.5]{Sch}) and we use the realisation
\[|P\mathcal{O}_\bullet| \longrightarrow |Gl_1(MSpin)_\bullet|\]
as the model of geometric twists.

As the first step, we now show that the realisation of the map $P\mathcal{O}_\bullet \longrightarrow Gl_1(MSpin)_\bullet$ just described models the $E_\infty$ map 
$$O[1] \simeq F_{Spin} \longrightarrow Gl_1(MSpin).$$
In the main text we observed that the stages
\[F_{Spin}(n) \rightarrow BSpin(n) \rightarrow BO(n)\]
of the fibre sequence $F_{Spin} \rightarrow BSpin \rightarrow BO$ can be represented by the sequence
$$P\mathcal{O}_n \rightarrow P\mathcal{O}_n/O(n) \rightarrow EP\mathcal{O}_n/O(n)$$
and these form the stages of a sequence of commutative $\mathcal I$-space monoids as we explain below. The sequence is covered by the vector bundle maps
$$P\mathcal{O}_n \times \mathbb R^n \rightarrow P\mathcal{O}_n \times_{O(n)} \mathbb R^n \rightarrow EP\mathcal{O}_n \times_{O(n)} \mathbb R^n$$
and the Thomification of the former is the projection $$(P\mathcal{O}_n)_+ \wedge S^n \rightarrow MSpin_n$$ (recall $MSpin_n = (P\mathcal O_n)_+ \wedge_{O(n)} S^n$). This adjoins to the map of commutative $\mathcal I$-space-monoids
$$P\mathcal{O}_n \longrightarrow \Omega^n MSpin_n, \quad \quad p \longmapsto \big(s \longmapsto [p,s]\big),$$
which (upon realisation) models the homotopical twistings sketched before the theorem essentially by construction. It also manifestly equals the map of Dardalat and Pennig by definition of the action of $P\mathcal O_n$ on $MSpin$ on the left factor.

In remains to provide the sequence 
\[P\mathcal{O}_n \rightarrow P\mathcal{O}_n/O(n) \rightarrow EP\mathcal{O}_n/O(n)\]
with the structure of commutative $\mathcal I$-space-monoids such that the realisation of the latter map obtains the standard Whitney sum $E_\infty$ structure $BSpin \rightarrow BO$. As we explained above the first term is in fact an $\mathcal I$-group, which is nothing but an $\mathcal I$-space together with a monoid structure with respect to the levelwise cartesian product. Now, the identity functor on $\mathcal I$-spaces with source given the cartesian and target the convolution product admits a tautological lax symmetric monoidal structure, i.e. a symmetric natural transformation
$$- \boxtimes - \longrightarrow  - \times -,$$
so that every $\mathcal I$-group gives rise to an $\mathcal I$-space-monoid, thus providing the desired structure on $P\mathcal O_\bullet$. Explicitely, the $\mathcal I$-space monoid structure on an $\mathcal I$-group $G_\bullet$ is given by
\[G_n \times G_m \rightarrow G_{n+m} \times G_{n+m} \rightarrow G_{n+m}\]
with the first map induced by the inclusions 
\[\{1,\dots,n\} \rightarrow \{1,\dots,n+m\} \leftarrow \{1,\dots,m\}\] 
as the first and last entries, respectively, and the second the group multiplication.

Now it is readily checked that the $\mathcal I$-space monoid structure on $P\mathcal O_\bullet$ is commutative, making it an Eckmann-Hilton $\mathcal I$-group in sense of Dardalat and Pennig. In fact, for an Eckmann-Hilton-$\mathcal I$-group it is readily checked that the $\mathcal I$-space-monoid structure induced from its $\mathcal I$-group structure always recovers the one given by the a priori $\mathcal I$-space-monoid structure. This implies that Eckmann-Hilton-$\mathcal I$-groups are precisely those $\mathcal I$-groups whose associated $\mathcal I$-space-monoid is commutative and in particular no extra structure needs to be specified (which makes some of the axioms in \cite{DP} redundant).

Similarly, the orthogonal groups $O(n)$ fit into an Eckmann-Hilton-$\mathcal I$-group $O(\bullet)$. The maps $j_n \colon O(n) \rightarrow P\mathcal O_n$ provide a map
\[j \colon O(\bullet) \rightarrow P\mathcal O_\bullet\]
of $\mathcal I$-groups and thus Eckmann-Hilton-$\mathcal I$-groups. The quotient $P\mathcal{O}_\bullet/O(\bullet)$ then retains a commutative $\mathcal I$-space monoid structure and choosing $EP\mathcal O_n$ as the bar-construction $B(*,P\mathcal O_n,P\mathcal O_n)$ provides a commutative $\mathcal I$-space-monoid structure to $EP\mathcal O_\bullet$, which also passes to the quotient $EP\mathcal O_\bullet/O_\bullet$. That the $E_\infty$ structures we obtain on the realisations are the usual ones under the equivalences $|P\mathcal O_\bullet/O(\bullet)| \simeq BSpin$ and $|EP\mathcal O_\bullet/O(\bullet))| \simeq BO$ follows from the commutativity of the diagram
\[\xymatrix@-1pc{Spin(\bullet) \ar[r]^j \ar[d] & \mathcal O_\bullet^{ev} \ar[r] \ar[d] & P\mathcal O_\bullet/O(\bullet) \ar[d]\\
            O(\bullet)     \ar[r]^j       & EP\mathcal O_\bullet  \ar[r] & EP\mathcal O_\bullet/O(\bullet)}\]
of commutative $\mathcal I$-space monoids, since its rows model the two universal bundles upon realisation. This finishes the construction of the commutative square involving the geometric and homotopical twists modelled in $\mathcal I$-spaces.\\ 

We proceed to compare the map $P\mathcal{O}_\bullet \longrightarrow Gl_1(MSpin)_\bullet$ to the (derived) map $P\mathcal{O}(\ell^2) \longrightarrow Gl_1(MSpin)$ constructed in the body of the main text (\ref{twispi}), in particular providing an $E_\infty$-structure on the latter. 

Just as the groups $P\mathcal{O}_n$ form an $\mathcal I$-group, so do the $P\mathcal{O}'_n$. The associated $\mathcal I$-space-monoid, however, is not commutative because of the extra $\ell^2$-factor. To obtain a proof of the theorem from the considerations so far we shall use (a version of) Shipley's fibrant replacement $\omega MSpin'$ (an $\Omega$-spectrum) of the spectrum $MSpin'$ and produce a commutative diagram
$$\xymatrix@-1pc{ P\mathcal{O}(\ell^2) \ar[d] \ar[r]^-\simeq & |P\mathcal{O}'_\bullet| \ar[d]               & |P\mathcal{O}_\bullet| \ar[d] \ar[l]_-\simeq \\
             \mathrm{Aut}_{MSpin}(MSpin') \ar[r]       & \mathrm{hAut}_{\omega MSpin}(\omega MSpin') & |Gl_1(MSpin)_\bullet| \ar[l]_-\simeq\ 
}$$
of monoids and groups. Its construction and the verification of commutativity of this diagram occupies essentially the rest of the appendix. The proof of the theorem is immediate from here and given below.\\

To construct the replacement $\omega M$ (denoted $M$ instead of $\omega$ in \cite[Section 3.2]{Sh}) for any $R$-module $M$ proceed as follows: Its $n$th space is the realisation of the $\mathcal I$-space $\omega_nM = \Omega^{\mathcal I}(sh^n M)_\bullet$, which in turn is given by 
$$k \longmapsto \Omega^k M_{k+n}$$
This $\mathcal I$-space retains a $\Sigma_n$-action through the last coordinates ($sh^n M$ is a $\Sigma_n$-spectrum after all) and thus $|\omega_n M|$ is indeed a $\Sigma_n$-space. The $R$-module structure (and thus the suspension maps) arise by realising the map
$$R_k \wedge (\omega_n M)_l \longrightarrow (\omega_{k+n} M)_l$$
given by mapping $r \in R_k, g: S^l \rightarrow M_{l+n}$ to
$$S^l \stackrel g\longrightarrow M_{l+n} \stackrel{r \cdot}\longrightarrow M_{k+l+n} \longrightarrow M_{l+k+n}$$
where the last map is induced by the obvious block permutation. The ring structure on $R$ induces pairings $\omega_k R \boxtimes \omega_l R \rightarrow \omega_{k+l} R$, which in turn give $\omega R$ a ring structure (though even if $R$ was commutative, $\omega R$ is not as the realisation is not symmetrically lax monoidal). Similarly, $\omega M$ becomes a module over $\omega R$; in formulas this action is given via 
$$(\omega_m R)_k \times (\omega_n M)_l \longrightarrow (\omega_{m+n} M)_{k+l}$$
sending
$$f: S^k \longrightarrow R_{k+m}, g:S^l \longrightarrow M_{l+n}$$
to 
$$S^{k+l} \cong S^k \wedge S^l \stackrel{f \wedge g}\longrightarrow R_{k+m} \wedge M_{l+n} \longrightarrow M_{k+m+l+n} \longrightarrow M_{k+l+m+n},$$
the last arrow again being a block permutation.
Restricting this action to $Gl_1(R)_\bullet$ sitting inside $(\omega_0 R)_\bullet$ produces an action of $|Gl_1(R)_\bullet|$ on $\omega M$ by homotopy equivalences. If $M$ is semi-stable (and as orthogonal spectra both $MSpin$ and $KO$ are) then $\omega M$ is readily checked to be an $\Omega$-spectrum (using \cite[Proposition 2.6]{SS2}, this also follows from the evident comparison to Shipley's detection functor $D$ and \cite[Theorem 3.1.5]{Sh}, but $D$ does not have good monoidality properties).

Analogously, suppose an action of an $\mathcal I$-group $G'$ on $M$, that is $\Sigma_n$-equivariant, base point preserving actions $G'_n \times M_n \rightarrow M_n$, such that the diagram
$$\xymatrix@-1pc{G'_n \times R_k \times M_n \ar[r] \ar[d]& R_k \wedge M_n \ar[d] \\
            G'_{k+n} \times M_{k+n} \ar[r]       & M_{k+n}}$$
commutes, where the map $G'_n \rightarrow G'_{k+n}$ is induced by the inclusion into the last $n$ coordinates as the notation suggests. From this we can produce maps $G'_\bullet \boxtimes \omega_k M \rightarrow \omega_k M$ via
$$G'_n \times (\omega_k M)_l \longrightarrow (\omega_k M)_{n+l}$$
sending $g \in G_{n}, f: S^l \rightarrow M_{l+k}$ to
$$S^{n+l} \cong S^n \wedge S^l \stackrel{\id \wedge f}\longrightarrow S^n \wedge M_{l+k} \longrightarrow M_{n+l+k} \stackrel{\cdot \iota(g)}{\longrightarrow} M_{n+l+k}$$
where $\iota\colon G_n \rightarrow G_{n+l+k}$ is induced by the inclusion into the first $n$ coordinates. This yields an associative action of the $\mathcal I$-space-monoid associated to the $\mathcal I$-group $G'$ on $\omega_k M$ and thus homomorphisms $|G'_\bullet| \rightarrow \mathrm{End}(\omega_k M)$, that are readily checked to fit together into a map $|G'_\bullet| \rightarrow \mathrm{End}_{\omega R}(\omega M)$.

Finally, to construct the vertical map in the six-term diagram above we need to recall (and slightly generalise) the construction by Dardalat and Pennig of a morphism of commutative $\mathcal I$-space monoids $DP \colon G_\bullet \rightarrow Gl_1(R)_\bullet$ from an action of an Eckmann-Hilton$-\mathcal I$-group $G_\bullet$ on a commutative symmetric ring spectrum $R$: As already mentioned their construction is given by
\[g \longmapsto (s \mapsto g \cdot u(s))\]
for $u_n: S^n \rightarrow R_n$ the unit of $R$, and indeed turns out to work for an arbitrary $\mathcal I$-group and an action as above with $M = R$ a not necessarily commutative ring spectrum. In this case, however, it yields a map $DP \colon G_\bullet \rightarrow Gl_1(R^{op})_\bullet$ of $\mathcal I$-space monoids (as should be expected since the diagram above expresses left linearity of the $G_\bullet$-action). The verification is a tedious but straight forward check using the centrality of the $\mathbb S$-action on $R$. 

If now given actions of $G_\bullet$ on both $R^{op}$ and $M$, we obtain two natural maps
\[G_\bullet \boxtimes \omega_k M \longrightarrow \omega_k M,\]
one as the composite
\[\xymatrix{G_\bullet \boxtimes \omega_k M \ar[r]^-{DP \boxtimes \mathrm{id}} & Gl_1(R)_\bullet \boxtimes \omega_k M \ar[r] & \omega_k M}\]
and the other straight from the discussion above. 

To compare these, we define an \emph{Eckmann-Hilton action} of an $\mathcal I$-group on an $R$-module $M$ to consist of actions $G_n \times R_n \rightarrow R_n$ and $G_n \times M_n \rightarrow M_n$ that give $G_\bullet$-actions on the $R$-modules $R^{op}$ and $M$, respectively, in the sense defined above and in addition make the diagram
$$\xymatrix@-1pc{G_m \times R_m \times M_n \ar[r] \ar[d] & R_m \wedge M_n \ar[d] \\
            G_{m+n} \times M_{m+n} \ar[r]                      & M_{m+n}}$$
commute. For $G_\bullet$ an Eckmann-Hilton-$\mathcal I$-group, $R = M$ a commutative ring spectrum and the two actions equal the two conditions we just presented splice together precisely to the condition Dardalat and Pennig give in \cite{DP}. Without the assumption that $R$ be commutative it clearly becomes a centrality condition for the action and it is readily checked, that in this case the Dardalat-Pennig construction also gives a map $G_\bullet \rightarrow Gl_1(R)_\bullet$ but we shall make no use of this as our examples are commutative.

For an Eckmann-Hilton action of an $\mathcal I$-group $G_\bullet$ on an $R$-module $M$ it is now easily checked that the two maps $G_\bullet \boxtimes \omega_k M \longrightarrow \omega_k M$ under consideration agree.

We now have all the ingredients for the proof of the comparison theorem: 

\begin{proof}[Proof of the theorem]
Consider again the diagram 
$$\xymatrix@-1pc{ P\mathcal{O}(\ell^2) \ar[d] \ar[r]^-\simeq & |P\mathcal{O}'_\bullet| \ar[d]               & |P\mathcal{O}_\bullet| \ar[d] \ar[l]_-\simeq \\
             \mathrm{Aut}_{MSpin}(MSpin') \ar[r]       & \mathrm{hAut}_{\omega MSpin}(\omega MSpin') & |Gl_1(MSpin)_\bullet|. \ar[l]_-\simeq\ 
}$$
First, the action of $P\mathcal{O}_\bullet$ on $MSpin$ gives the right vertical arrow. Secondly, the $\mathcal I$-group $P\mathcal{O}'_\bullet$ acts on the $MSpin$-module $MSpin'$ giving the vertical map in the middle. The lower horizontal maps come from the general properties of $\omega$ discussed above. Now, $P\mathcal{O}_\bullet$ includes into $P\mathcal{O}'_\bullet$ and thus we obtain an action of $P\mathcal{O}_\bullet$ on $MSpin'$ by restriction. Together with its action on $MSpin$ this gives an Eckmann-Hilton action of $P\mathcal{O}_\bullet$ on the $MSpin$-module $MSpin'$ so that the right hand square is commutative. All asserted coherences are straight forward calculations and we leave them to the reader.

The left hand side is easier: We can form the constant $\mathcal I$-group with value $P\mathcal{O}(\ell^2)$. The group $P\mathcal O(\ell^2)$ includes into its  realisation $$|const_\mathcal I P\mathcal{O}(\ell^2)| = B\mathcal I \times P\mathcal{O}(\ell^2)$$ which in turn includes into that of $P\mathcal{O}'_\bullet$ via a monoid homomorphism. The lower horizontal map again comes from the obvious functoriality of $\omega$ and the diagram clearly commutes. We explained above how the theorem follows from the commutativity of the above diagram.
\end{proof}

Let us remark once more that the detour through the fibrant replacement and the diagram above are necessitated by the framework of May and Sigur\dh{}sson. The construction of the map $P\mathcal O_\bullet \rightarrow Gl_1(MSpin)_\bullet$ and its comparison to the homotopically constructed map $F_{Spin} \rightarrow Gl_1(MSpin)$ made no use of it.

\begin{urem}
Finally, let us explain the error in the work of Antieau, Gepner and Gomez. They show that the space of maps $K(\mathbb Z/2,2) \rightarrow BGl_1(KO)$ has two path-components, but their construction of an example of such a map through geometric means is faulty: In the appendix they describe an evident action of $P\mathcal{O}(\ell^2)$ on each space of the $K$-theory spectrum from \cite{Jo} but these actions do not actually commute with the suspension maps invalidating \cite[Proposition 4.1]{AGG}. Therefore, there seems no clear way to obtain a map $BP\mathcal{O}(\ell^2) \rightarrow BGl_1(KO)$ from their results. 
A modification of the spectrum from \cite{Jo} along the lines of our modification $KO'$ of $KO$ can certainly be used to produce a map as desired. However, it will not be immediate that such a map admits an $E_\infty$ refinement, something we obtain from our use of $\mathcal I$-space-monoids only.
\end{urem}

\end{document}